\documentclass[10pt, a4paper, smallheadings]{amsart}
\usepackage{amssymb}
\usepackage{amsmath}
\usepackage{amssymb}
\usepackage{amsthm}
\usepackage{bbm}
\usepackage{esint}
\usepackage{bm}
\usepackage{mathrsfs}
\usepackage{tikz-cd}
\usetikzlibrary{decorations.pathreplacing}

\usepackage{enumitem}
\setenumerate{label={\rm (\alph{*})}}
\usepackage[colorlinks=true,linkcolor=blue,urlcolor=blue]{hyperref}

\usepackage{graphicx}
\usepackage{tikz}

\usepackage{geometry}
\numberwithin{equation}{section}

\newtheorem{theorem}{Theorem}[section]
\newtheorem{lemma}[theorem]{Lemma}
\newtheorem{proposition}[theorem]{Proposition}

\newtheorem{corollary}[theorem]{Corollary}

\newtheorem{definition}[theorem]{Definition}

\newtheorem{remark}[theorem]{Remark}

\numberwithin{equation}{section}

\normalsize

\def\XXint#1#2#3{{\setbox0=\hbox{$#1{#2#3}{\int}$ }
\vcenter{\hbox{$#2#3$ }}\kern-.6\wd0}}

\def\dashint{\Xint-}

\newcommand{\bd}{\operatorname{BD}}
\newcommand{\bv}{\operatorname{BV}}
\newcommand{\ld}{\operatorname{LD}}
\newcommand{\dif}{\operatorname{d}\!}

\newcommand{\tr}{\operatorname{Tr}}
\newcommand{\R}{\mathbb{R}}
\newcommand{\locc}{\operatorname{loc}}

\newcommand{\dista}{\operatorname{dist}}

\newcommand{\curl}{\operatorname{curl}}

\newcommand{\ball}{\operatorname{B}}
\newcommand{\sym}{\operatorname{sym}}

\newcommand{\sobo}{\operatorname{W}}
\newcommand{\lebe}{\operatorname{L}}
\newcommand{\hold}{\operatorname{C}}
\newcommand{\besov}{\operatorname{B}}
\newcommand{\bmo}{\operatorname{BMO}}

\newcommand{\sg}{\bm{\varepsilon}}

\newcommand{\gm}{\operatorname{GM}}
\newcommand{\E}{\operatorname{E}\!}
\newcommand{\D}{\operatorname{D}\!}
\newcommand{\di}{\operatorname{div}}
\newcommand{\sgn}{\operatorname{sgn}}

\newcommand{\trace}{\operatorname{Tr}}

\newcommand{\rsym}{\mathbb{R}_{\sym}^{n\times n}}
\renewcommand{\dashint}{\fint}
\newcommand{\ext}{\operatorname{Ex}}

\newcommand{\mres}{\!\mathbin{\vrule height 1.6ex depth 0pt width
0.13ex\vrule height 0.13ex depth 0pt width 1.1ex}\!}
\usepackage{tikz-cd}
\begin{document}
\title{SOBOLEV REGULARITY FOR CONVEX FUNCTIONALS ON BD}
\author[F.~Gmeineder]{Franz Gmeineder}
\address[F.~Gmeineder]{University of Bonn, Mathematical Institute, Endenicher Allee 60, 53115 Bonn, Germany}
\author[J.~Kristensen]{Jan Kristensen}
\address[J.~Kristensen]{University of Oxford, Andrew Wiles Building, Radcliffe Observatory Quarter, Woodstock Rd, Oxford OX2 6GG, United Kingdom}
\thanks{
The results of this paper appear as an extended and partially improved version of the results as presented in the first author's doctoral thesis \cite{Gmeineder1}. As such, financial support through the EPRSC is gratefully acknowledged. Moreover, the authors would like to thank Gregory Seregin for discussions related to this work.}
\maketitle
\begin{abstract}
We establish the first Sobolev regularity and uniqueness results for minimisers of autonomous, convex variational integrals of linear growth which depend on the symmetric rather than the full gradient. This extends the results available in the literature for the BV--setting to the case of functionals whose full gradients are a priori not known to exist as finite matrix--valued Radon measures. 
\end{abstract}
\section{Introduction}
Let $\Omega$ be an open and bounded Lipschitz domain of $\R^{n}$. In this work we study the regularity of minimisers of autonomous convex variational integrals of the form
\begin{align}\label{eq:main}
\mathfrak{F}[u]:=\int_{\Omega}f(\sg(u))\dif x,\qquad u\colon\Omega\to\R^{n},
\end{align}
subject to suitable Dirichlet boundary conditions. Here, $\sg(u):=\frac{1}{2}(\D u + \D^{\mathsf{T}}u)$ denotes the \emph{symmetric gradient} and $f\in\hold^{2}(\R^{n\times n})$ is a variational integrand of \emph{linear growth}. By this we understand that there exist two constants $0<c_{0}\leq c_{1}<\infty$ and $c_{2}\in\R$ such that 
\begin{align}\label{eq:lg}
c_{0}|\xi|-c_{2}\leq f(\xi)\leq c_{1}(1+|\xi|)\qquad\text{for all}\;\xi\in\R_{\sym}^{n\times n}.
\end{align}
Due to the growth condition \eqref{eq:lg}, the functional $\mathfrak{F}$ is well-defined on Dirichlet classes $u_{0}+\sobo_{0}^{1,1}(\Omega;\R^{n})$ for $u_{0}\in\sobo^{1,1}(\Omega;\R^{n})$. However, as opposed to the superlinear, i.e., $p>1$, growth regime, there is \emph{no} Korn inequality on $\lebe^{1}$. This result follows from \emph{Ornstein's Non-Inequality} \cite{Ornstein} (see~\cite{CFM,KirchKrist,Woj} for more recent contributions) and implies that there is \emph{no} constant $C>0$ such that 
\begin{align}\label{eq:Ornstein}
\int_{\Omega}|\D u|\dif x \leq  C\,\int_{\Omega}|\sg(u)|\dif x,
\end{align}
holds for all $u\in \hold_{c}^{1}(\Omega;\R^{n})$. As a consequence, neither $\mathfrak{F}$ nor a suitable relaxation to the space $\bv(\Omega;\R^{n})$ is coercive on these spaces. Note that, as a consequence of non-reflexivity of $\sobo^{1,1}$, even if $\mathfrak{F}$ \emph{had} bounded minimising sequences in $\sobo^{1,1}$, these could not be proven to be weakly precompact in $\sobo^{1,1}$. Hence, in this case, the relaxation to the space of functions of bounded variation would be necessary. 

Basically, the non--validity of estimate \eqref{eq:Ornstein} is a consequence of unboundedness of singular integral operators on $\lebe^{1}$. In this context, the appropriate substitutes are given by the spaces $\ld(\Omega)$ and $\bd(\Omega)$. These consist of all $u\in\lebe^{1}(\Omega;\R^{n})$ for which the distributional symmetric gradients $\sg(u)$ belong to $\lebe^{1}(\Omega;\R_{\sym}^{n\times n})$ or can be represented by a $\R_{\sym}^{n\times n}$--valued finite Radon measure $\E u$ on $\Omega$, respectively; see Section~\ref{sec:bd}. In particular, Ornstein's Non--Inequality  then implies that in general $\bv(\Omega;\R^{n})\subsetneq\bd(\Omega)$ and that the \emph{full} distributional gradients of $\bd$--functions might not even exist as locally finite measures. Hence the chief question which we shall treat in the present work is to find conditions on the variational integrand $f\in\hold^{2}(\R_{\sym}^{n\times n})$ such that minimisers \emph{indeed qualify as elements of $\bv(\Omega;\R^{n})$ or Sobolev spaces $\sobo^{1,p}(\Omega;\R^{n})$}. At present, such results were only available for the full gradient (i.e., $\bv$-) case and, by Ornstein's Non-Inequality, do \emph{not} apply to the situation considered here. Before we turn to a description of our results, we first introduce the concept of minima and the class of integrands we shall work with.
\subsection{Generalised Minimisers}
To define the concept of minimisers we shall work with, let $\widetilde{u}_{0}\in\lebe^{1}(\partial\Omega;\R^{n})$ be a given Dirichlet datum. Since all $\sobo^{1,1}(\Omega;\R^{n})$, $\ld(\Omega)$ and $\bd(\Omega)$ have trace space $\lebe^{1}(\partial\Omega;\R^{n})$ (see Section~\ref{sec:bd} for more detail), we find $u_{0}\in\ld(\Omega)$ such that $\trace(u_{0})=\widetilde{u}_{0}$ $\mathcal{H}^{n-1}$-a.e. on $\partial\Omega$. We hence consider the variational principle 
\begin{align}\label{eq:varprin}
\text{to minimise}\;\;\mathfrak{F}\;\;\text{within a Dirichlet class}\;\mathscr{D}_{u_{0}}:=u_{0}+\ld_{0}(\Omega), 
\end{align}
where $\ld_{0}(\Omega)$ is the closure of $\hold_{c}^{1}(\Omega;\R^{n})$ with respect to $\|v\|_{\ld}:=\|v\|_{\lebe^{1}}+\|\bm{\varepsilon}(v)\|_{\lebe^{1}}$. By virtue of the growth condition \eqref{eq:lg}, it is easy to show that $\inf\mathfrak{F}[\mathscr{D}_{u_{0}}]>-\infty$. In turn, by a standard compactness principle in $\bd$ to be recalled for the reader's convenience in Section~\ref{sec:bd}, a suitable subsequence converges to some $u\in\bd(\Omega)$ in the weak*-sense. 

As $\mathfrak{F}$ given by \eqref{eq:main} is merely defined for elements of $\ld(\Omega)$, it must be relaxed in order to be defined for the weak*-limit $u\in\bd(\Omega)$. Here we take advantage of convexity, thereby reducing to the classical theory of convex functions of measures due to \textsc{Reshetyak} \cite{Reshetnyak}. To capture the behaviour of the integrand at infinity, we hereafter define the
 \emph{recession function} $f^{\infty}$ associated with $f$ by 
\begin{align}\label{eq:recession}
f^{\infty}(\xi):=\lim_{t\searrow 0}tf\left(\xi/t\right),\qquad\xi\in\R_{\sym}^{n\times n}. 
\end{align}
Using convexity and linear growth of $f$, it is easy to show that $f^{\infty}$ is well--defined and convex, too. Let $u\in\bd(\Omega)$ and consider the Lebesgue-Radon-Nikod\'{y}m decomposition of $\E u=\E^{a}u+\E^{s}u$ into its absolutely continuous and singular parts with respect to $\mathscr{L}^{n}$. Then we put 
\begin{align}\label{eq:relaxed}
\begin{split}
\overline{\mathfrak{F}}_{u_{0}}[u]   :=\int_{\Omega}f(\E^{a}u)\dif x & + \int_{\Omega}f^{\infty}\left(\frac{\dif \E^{s}u}{\dif |\E^{s}u|}\right)\dif |\E^{s}u| \\ & + \int_{\partial\Omega}f^{\infty}\big((\trace(u_{0})-\trace(u))\odot\nu_{\partial\Omega}\big)\dif\mathcal{H}^{n-1}
\end{split}
\end{align}
where as usual $\dif \E^{s}u /\dif |\E^{s}u|$ denotes the Lebesgue density of $\E^{s}u$ with respect to its total variation measure $|\E^{s}u|$ and $\nu_{\partial\Omega}$ is the outward unit normal to $\partial\Omega$ (see Section \ref{sec:bd} for the relevant notation). Here, the boundary integral term penalises deviations of $u$ from the Dirichlet data $u_{0}$. Then $\overline{\mathfrak{F}}_{u_{0}}$ coincides with the corresponding weak*-lower semicontinuous envelope\footnote{Here, because $f$ is autonomous, convex with the lower bound of \eqref{eq:lg}, it is irrelevant whether we choose the weak*-- or $\lebe^{1}$--relaxation; indeed, they coincide in this case.} of $\mathfrak{F}$ over $\mathscr{D}_{u_{0}}$. 

We then call $u\in\bd(\Omega)$ a \emph{generalised minimiser} for $\mathfrak{F}$ (with respect to $u_{0}$) if and only if there holds
\begin{align}\label{eq:GMdef}
\overline{\mathfrak{F}}_{u_{0}}[u]\leq \overline{\mathfrak{F}}_{u_{0}}[v]\qquad\text{for all}\;v\in\bd(\Omega). 
\end{align}
The set $\gm(\mathfrak{F};u_{0})$ consists of all generalised minimisers for $\mathfrak{F}$ (with respect to $u_{0}$). For brievity, we shall write $\overline{\mathfrak{F}}$ instead of $\overline{\mathfrak{F}}_{u_{0}}$, tacitly assuming that $u_{0}$ is fixed. Most crucially, we have $\inf \mathfrak{F}[\mathscr{D}_{u_{0}}]=\min\overline{\mathfrak{F}}[\bd(\Omega)]$. Moreover, generalised minima can be conveniently characterised as those maps $v\in\bd(\Omega)$ for which there exists a minimising sequence $(v_{k})\subset\mathscr{D}_{u_{0}}$ that converges strongly to $v$ in $\lebe^{1}(\Omega;\R^{n})$. These results, which are similar to the $\bv$-case but hard to explicitely trace back for the situation at our disposal, shall be collected and demonstrated in the appendix of the paper, cf. Section~\ref{sec:relaxation}.


\subsection{$\mu$--elliptic Integrands}
Throughout the present work we shall further assume that $f\in\hold^{2}(\R_{\sym}^{n\times n})$ is a \emph{$\mu$--elliptic} (and thus, in particular, convex) integrand. Reminiscent of the classical \textsc{Bernstein} genre (cf.~\cite{Serrin}, \cite[Ex.~3.2ff.]{GMS} and the references therein), this notion of ellipticity had been rediscovered and studied in a series of papers by \textsc{Bildhauer, Fuchs and Mingione} \cite{Bild1,Bild2,FuchsMingione,BildhauerFuchsMingione} concerning minimisation problems of the type \eqref{eq:main}, where $\sg$ was replaced by the full gradients and which we recall here for completeness:
\begin{definition}[$\mu$--ellipticity]
Let $1< \mu<\infty$. A $\hold^{2}$--integrand $f\colon \R_{\sym}^{n\times n}\to \R_{\geq 0}$ is called $\mu$--\emph{elliptic} if and only if there exist $0<\lambda\leq\Lambda<\infty$ such that for all $\mathbf{A},\mathbf{B}\in \R_{\sym}^{n\times n}$ there holds
\begin{align}\label{eq:muell}
\lambda\frac{|\mathbf{A}|^{2}}{(1+|\mathbf{B}|^{2})^{\frac{\mu}{2}}}\leq \langle f''(\mathbf{B})\mathbf{A},\mathbf{A}\rangle \leq \Lambda \frac{|\mathbf{A}|^{2}}{(1+|\mathbf{B}|^{2})^{\frac{1}{2}}}.
\end{align} 
\end{definition} 
As a direct consequence of the definition, $\mu$--elliptic integrands are automatically strongly convex. An important class of examples is provided by the one--parameter family of integrands $\{\varphi_{\mu}\}_{\mu >1}$, given by 
\begin{align*}
\varphi_{\mu}(\xi):=\int_{0}^{|\xi|}\int_{0}^{s}\big(1+t^{2}\big)^{-\frac{\mu}{2}}\dif t\dif s,\qquad\xi\in\R^{n\times n}. 
\end{align*}
Then, as shown in \cite[Ex.~3.9 and 4.17]{Bild1}, each $\varphi_{\mu}$ is $\mu$--elliptic. Moreover, for $\mu=3$ we cover the usual area integrand $\langle\cdot\rangle := \sqrt{1+|\cdot|^{2}}$. Let us further note that in the definition of $\mu$--ellipticity the case $\mu=1$ is explicitely excluded. Indeed, it is easy to show that $1$--elliptic $\hold^{2}$--integrands are of $L\log L$--growth and hence beyond the scope of integrands of linear growth. Moreover, by mapping properties of singular integrals of convolution type, minimisers of \eqref{eq:main} with $1$--elliptic $f$ possess full distributional gradients in $\lebe^{1}(\Omega;\R^{n\times n})$. In consequence, when studying regularity properties of such minima, we may directly test with the full difference quotients and hence no modification of the common difference quotient method is required. As a characteristic feature of integrands satisfying \eqref{eq:muell}, we note by $\mu>1$ that the growth behaviour from above and below differs \emph{on the level of second derivatives}, a fact which is not so common for functionals of $p$-growth but rather appears in the theory of functionals of $(p,q)$-growth, cf.~ \cite{Marc1,Marc2,Bild4,ELM,CKP,CKP1}.
\subsection{Background and Main Results}
By our method of proof, we shall distinguish two ellipticity regimes which come along with different obstructions. Before we embark on the description of our main results and novelties, we first summarise the results available so far.   
\subsubsection{Contextualisation}\label{sec:context}
To contextualise our results, let us briefly recall the results known for the full gradient case, that is, where the symmetric gradient in \eqref{eq:varprin} is replaced by the full gradient;  we are hereby lead to the Dirichlet problem in $\bv$. Employing a vanishing viscosity approach in the spirit of \textsc{Seregin} \cite{Seregin1,Seregin2,Seregin3,Seregin4}, \textsc{Bildhauer} \cite{Bild2,Bild1} established the first $\sobo^{1,1}$-regularity result for generalised minima for $\mu\leq 3$. In doing so, crucial use is made of the so-called \emph{local boundedness hypothesis}, namely, that the sequence of vanishing viscosity solutions is uniformly bounded in $\lebe^{\infty}$ on relatively compact subsets $K$ of $\Omega$, cf.~ \cite[Ass.~4.11]{Bild1}, \cite{Bild2} and \cite[Thm.~1.10]{BS1}. Up to date, for \emph{autonomous full gradient functionals} there are no Sobolev regularity results for the Dirichlet problem available beyond $\mu=3$. However, even though \textsc{Bildhauer}'s approach leads to $\sobo_{\locc}^{1,L\log L}$-regularity of \emph{one} particular generalised minimiser (namely, the limit of a special vanishing viscosity sequence) and the integrand is strictly convex, it does not rule out the possible presence of other, more irregular generalised minima. We recall that the recession function $f^{\infty}$ is positively $1$-homogeneous regardless of strict convexity of $f$. This is an obvious source of non-uniqueness, and as long the presence of the singular part of the gradient measures of \emph{all} generalised minima cannot be excluded, there might in fact exist other generalised minima that do not share the $\sobo_{\locc}^{1,L\log L}$-regularity. The uniqueness of generalised minima for the Dirichlet problem on $\bv$ has been addressed by \textsc{Beck \& Schmidt} \cite{BS1}. Here, the authors combine \textsc{Bildhauer}'s approach with the Ekeland variational principle to deduce that \emph{all} generalised minima of $\mu=3$-elliptic variational integralds share the aforementioned regularity. 

However, a common difficulty in deriving a higher differentiability result for functionals of the type \eqref{eq:main} under the linear growth assumption on $f$ is that, by Ornstein's Non--Inequality, the full distributional gradients of $\bd$--functions do not need to exist as Radon measures of finite total mass.  Hence, we shall consider fractional estimates instead and utilise the fact that -- as $\bv$ and $\bd$ embed into the same fractional Sobolev spaces -- $\bd$-maps behave similarly as $\bv$ on the fractional level.
\subsubsection{An unconditional regularity result for small $\mu$}
We begin with the following result which -- because of its restriction to ellipiticities $\mu < \frac{n+1}{n}$ -- applies to \emph{all} generalised minimisers. In effect, it appears as the generalisation of \cite[Sec.~5]{BS1} to problems of the form \eqref{eq:varprin}. 
\begin{theorem}\label{thm:main0}
Let $n\geq 2$. Suppose that $f\in\hold^{2}(\R_{\sym}^{n\times n})$ is a $\mu$-elliptic variational integrand of linear growth, i.e., satisfies \eqref{eq:lg} and \eqref{eq:muell} with $1<\mu<\frac{n+1}{n}$. Then all generalised minimisers belong to $\sobo_{\locc}^{1,1}(\Omega)$. More precisely, we have for some $p=p(\mu,n)>1$
\begin{align}
\gm(\mathfrak{F};u_{0})\subset \sobo_{\locc}^{1,p}(\Omega;\R^{n})\cap\ld(\Omega).
\end{align}
\end{theorem}
This theorem will be established in Section~\ref{sec:thmmain0}. As  for the $\bv$-case discussed above, a chief difficulty stems from the fact that the symmetric gradients of generalised minimisers are not a priori known to be absolutely continuous with respect to $\mathscr{L}^{n}$. Consequently, since we do not have uniqueness of generalised minima, a stabilisation procedure relying on the vanishing viscosity approach must be suitably modified. In doing so, we follow essentially the lines of \textsc{Beck \& Schmidt} \cite{BS1}. Here, starting from a given generalised minimiser, we construct a specific minimising sequence that weakly*-converges to the generalised minimiser so that each of its members almost minimises an appropriately stabilised functional. Since this sequence belongs to $\ld$, we are further in position to avoid manipulations on difference quotients of measures. In constructing the aforementioned specific minimising sequence, we make use of the Ekeland variational principle and employ it in the dual space $(\sobo_{0}^{1,\infty})^{*}$ to obtain perturbations that are weak enough to be dealt with using the available a priori-estimates. Higher differentiability estimates can also be obtained, cf.~Corollary~\ref{cor:nikolskii}.

\subsubsection{A result subject to the local $\bmo$-hypothesis.} 
In amplyfing the ellipticity range of Theorem~\ref{thm:main0}, we crucially exploit a slight generalisation of the local boundedness assumption as discussed in Section~\ref{sec:context} above. Namely, we shall suppose that a particular vanishing viscosity sequence is locally uniformly bounded in $\bmo$. In the $\bv$-setting, the local $\lebe^{\infty}$-hypothesis can be justified by use of a maximum principle or, for a class of integrands slightly more general than the radial ones, by use of a Moser iteration approach, cf.~\cite[Thm.~1.11]{BS1}. This, however, is far from clear in the $\bd$-context, cf. Remark~\ref{rem:onlyone} below.

Let us, however, briefly explain the choice of this hypothesis: First, it can be justified for a class of suitably regular $\hold^{2}$-integrands and sufficiently \emph{smooth and small} boundary data, the technical verification of which being slightly beyond the scope of the paper and being deferred to a forthcoming work. Second, it should also prove interesting for the $\bv$-case as a extension of the local $\lebe^{\infty}$-hypothesis.

In comparison with Theorem~\ref{thm:main0}, we can only derive a regularity result for \emph{one} particular generalised minimiser. As a main novelty though, we crucially utilise the fact that the dual solution (in the sense of convex duality) belongs to $\sobo_{\locc}^{1,2}(\Omega;\R_{\sym}^{n\times n})$. This can be established by analogous means as done by \textsc{Seregin} \cite{Seregin1,Seregin3,Seregin4}, and in conjunction with the local $\bmo$-hypothesis, we therefore use the impact of the \emph{dual solution} on the \emph{primal solution}. Indeed, the dual problem is substantially better behaved than the primal one: Whereas the recession part of the relaxed primal functional might lead to non--uniqueness for generalised minimisers provided a too weak ellipticity is assumed for $f$, the dual solution is always unique and enjoys the aforementioned higher Sobolev regularity. However, unlike elliptic variationals of $p$--growth, $p>1$, in our case we may not assume that $\sigma=f'(\sg(u))$, where $u\colon\Omega\to\R^{n}$ is a generalised minimiser. In fact, this would be the case if we knew a priori that $u\in\ld(\Omega)$, but this is not clear at the relevant stage of the proof. Our second main result then reads as follows; note that $u_{0}\in\sobo^{1,2}(\Omega)$ is not necessary but facilitates the statement of the theorem (cf.~ \eqref{eq:viscosityapprox} and the discussion afterwards):
\begin{theorem}\label{thm:main1}
Let $f\in\hold^{2}(\R_{\sym}^{n\times n})$ be a $\mu$-elliptic integrand of linear growth with $\mu<1+\frac{3}{2n}$ and let $u_{0}\in\sobo^{1,2}(\Omega;\R^{n})$. Suppose that the sequence of minimisers $(v_{j})$ of the stabilised functionals 
\begin{align*}
\mathfrak{F}_{j}[v]:=\mathfrak{F}[v] + \frac{1}{2j}\int_{\Omega}|\sg(v)|^{2}\dif x\qquad\text{over}\;u_{0}+\sobo^{1,2}(\Omega;\R^{n})
\end{align*}
are locally uniformly bounded in $\bmo$. Then there exists $u\in\gm(\mathfrak{F};u_{0})$ which arises as the weak*-limit of a suitable subsequence of $(v_{j})$, and this weak*-limit satisfies, for some $1<p<\infty$,
\begin{align}
u\in \ld(\Omega)\cap\sobo_{\locc}^{1,p}(\Omega;\R^{n}).
\end{align} 
\end{theorem}
This theorem, to be proved in Section~\ref{sec:convdual}, utilises a novel embedding for $\bd\cap\bmo$. The latter uses the \textsc{Doronsorro}-type characterisation of Besov spaces $\besov_{p,q}^{s}$ \cite{Doro} and should be of independent interest, see Section~\ref{sec:emb}.  As Theorems~\ref{thm:main0} and \ref{thm:main1} come along with higher fractional differentiability, we are hereby in position to derive Hausdorff dimension bounds for the singular set of generalised minima, cf.~Corollaries~\ref{cor:HausdorffCor1} and \ref{cor:SCNikolskii}.

As is well-known from the classical minimal surface example, another source of non-uniqueness stems from the non-attainment of boundary values; see the classical examples due to \textsc{Santi} \cite{Santi} or \textsc{Finn} \cite{Finn}. In this respect, the main part of the paper is concluded by investigating the impact of regularity on the uniqueness of generalised minima, see Section~\ref{sec:uniqueness} and Theorem~\ref{thm:unique} therein.

\subsection{Organisation of the Paper} 
In Section \ref{sec:setup} we fix notation, collect the requisite background facts regarding function spaces and record some auxiliary estimates. Section~\ref{sec:emb} is devoted to the embedding of $\bd\cap\bmo$ into fractional Sobolev spaces. In Section~\ref{sec:main}, we give the proofs of the aforementioned main results regarding the regularity and uniqueness of generalised minima. Finally, the appendix in Section~\ref{sec:appendix} discusses extensions of the main results and contains auxiliary material used in the main part. In particular, it covers the relaxation of the Dirichlet problem to $\bd$ and the existence of generalised minima which we tacitly assumed throughout. 
\section{Setup}\label{sec:setup}
\subsection{General Notation}\label{sec:notation}
Unless stated otherwise, we assume $\Omega$ to be an open and bounded Lipschitz domain in $\R^{n}$. Given $x_{0}\in\R^{n}$ and $r>0$, we denote $\ball(x_{0},r):=\{x\in\R^{n}\colon |x-x_{0}|<r\}$ the open ball with radius $r$ centered at $x_{0}$ and $\langle\cdot,\cdot\rangle$ denotes the euclidean inner product on finite dimensional real vector spaces. The $n$--dimensional Lebesgue measure is denoted $\mathscr{L}^{n}$ and the $(n-1)$--dimensional Hausdorff measure is denoted $\mathcal{H}^{n-1}$. Given two positive, real valued functions $f,g$, we indicate by $f\lesssim g$ that $f\leq C g$ with a constant $C>0$. If $U\subset\R^{n}$ is measurable with $\mathscr{L}^{n}(U)>0$ and $f\in\lebe^{1}(U;\R^{N})$, we put as usual 
\begin{align*}
(f)_{U}:=\dashint_{U}f\dif\mathscr{L}^{n}:=\frac{1}{\mathscr{L}^{n}(U)}\int_{U}f\dif\mathscr{L}^{n}. 
\end{align*}
For a given measurable function $f\colon\Omega\to\R^{m}$, a unit vector $e_{s}$, $s=1,...,n$, and a stepwidth $h\in\R\setminus\{0\}$, we define the \emph{forward finite difference} $\tau_{s,h}^{+}f(x)$ and the \emph{backward finite difference} $\tau_{s,h}^{-}$, respectively, by 
\begin{align*}
\tau_{s,h}^{+}f(x):=f(x+he_{s})-f(x),\quad \tau_{s,h}^{-}f(x):=f(x-he_{s})-f(x)
\end{align*}
for all $x\in\Omega$ with $\dista(x,\partial\Omega)>|h|$. Moreover, for such $x$ we put 
\begin{align*}
\Delta_{s,h}f(x):=\Delta_{s,h}^{+}f(x):=\frac{\tau_{s,h}^{+}f(x)}{h},\quad \Delta_{s,h}^{-}f(x):=\frac{\tau_{s,h}^{-}f(x)}{h}. 
\end{align*}
Finally, we denote by $\R_{\sym}^{n\times n}$ the symmetric $n\times n$--matrices with real entries and, given $u,v\in\R^{n}$, we denote their dyadic product $u\otimes v:=uv^{\mathsf{T}}$ and their symmetric dyadic product $u\odot v:=\frac{1}{2}(u\otimes v + v\otimes u)$. 
\subsection{Functions of Bounded Deformation}\label{sec:bd}
Here we recall the space of functions of bounded deformation as introduced in \cite{CMS,Suquet}. For more detailled background information, the reader is referred to \cite{ACD,ST,Baba,FS}.
Let $\Omega\subset\R^{n}$ be open. A measurable function $u\colon\Omega\to\R^{n}$ belongs to $\bd(\Omega)$ if and only if $u\in\lebe^{1}(\Omega;\R^{n})$ and its \emph{total deformation}
\begin{align}
|\E u|(\Omega):=\sup\Big\{\int_{\Omega}\langle v,\di(\varphi)\rangle\dif x\colon\;\varphi\in\hold_{c}^{1}(\Omega;\R_{\sym}^{n\times n}),\|\varphi\|_{\lebe^{\infty}(\Omega;\R_{\sym}^{n\times n})}\leq 1\Big\}
\end{align}
is finite, where the divergence has to be understood row--wise (note that we write $\E u$ for the distributional symmetric gradient when this is a measure and reserve $\sg(u)$ for weak symmetric gradients exclusively). The norm on $\bd(\Omega)$ is given by $\|u\|_{\bd(\Omega)}:=\|u\|_{\lebe^{1}(\Omega;\R^{n})}+|\E u|(\Omega)$, and endowed with this norm, $\bd(\Omega)$ is a Banach space. Since the norm topology is too strong for most applications, it is useful to consider the following convergences instead: We say that a sequence $(u_{k})\subset\bd(\Omega)$ converges to $u\in\bd(\Omega)$ in the \emph{weak*--sense} provided $u_{k}\to u$ strongly in $\lebe^{1}(\Omega;\R^{n})$ and $\E u_{k}\stackrel{*}{\rightharpoonup}\E u$ in the sense of $\R^{n\times n}$--valued measures as $k\to\infty$. Moreover, if $(u_{k})$ converges to $u$ in the weak*--sense and $|\E u_{k}|(\Omega)\to|\E u|(\Omega)$ as $k\to\infty$, then we say that $(u_{k})$ converges \emph{($\bd$--)strictly} to $u$ as $k\to\infty$. Lastly, we say that $(u_{k})$ converges to $u$ in the \emph{($\bd$--)area--strict sense} provided $u_{k}\to u$ strictly and 
\begin{align*}
\sqrt{1+|\E u_{k}|^{2}}(\Omega)\to \sqrt{1+|\E u|^{2}}(\Omega)\qquad\text{as}\;k\to\infty. 
\end{align*}
The concept of applying convex functions (so, e.g., the area--type integrand $\sqrt{1+|\cdot|^{2}}$) to a measure as done here will be carefully explained in Section \ref{sec:convexfunctions of measures} below.

Resembling the fact that $\bv(\Omega;\R^{N})$ arises as the weak*--closure of $\sobo^{1,1}(\Omega;\R^{N})$, $\bd(\Omega;\R^{N})$ is the weak*--closure of the space 
\begin{align}
\ld(\Omega):=\big\{u\in\lebe^{1}(\Omega;\R^{n})\colon\; \sg(u)\in\lebe^{1}(\Omega;\R_{\sym}^{n\times n})\big\}, 
\end{align}
where $\sg(u)$ is the distributional symmetric gradient, and the norm on $\ld(\Omega)$ is given by $\|u\|_{\ld(\Omega)}:=\|u\|_{\lebe^{1}(\Omega;\R^{n})}+\|\sg(u)\|_{\lebe^{1}(\Omega;\R^{n\times n})}$. We further define $\ld_{0}(\Omega)$ to be the closure of $\hold_{c}^{1}(\Omega;\R^{n})$ with respect to $\|\cdot\|_{\ld(\Omega)}$. The claimed property that $\bd(\Omega)$ is the weak*--closure of $\ld(\Omega)$ follows from the fact that $(\ld\cap\hold^{\infty})(\Omega)$ is dense in $\bd(\Omega)$ with respect to weak*- and strict convergence, see \cite{AnzeGiaq}. If $\Omega$ is a bounded Lipschitz subset of $\R^{n}$, then there exists a 
\begin{itemize}
\item surjective trace operator $\tr\colon \ld(\Omega)\to \lebe^{1}(\partial\Omega;\R^{n})$ which is continuous with respect to the $\ld$--norm;
\item surjective trace operator $\tr\colon\bd(\Omega)\to \lebe^{1}(\partial\Omega;\R^{n})$ which is continuous with respect to strict (but not weak*-) convergence. 
\end{itemize}
Given $u\in\bd(\Omega)$ and splitting the symmetric gradient measure $\E u$ into its absolutely continuous and singular parts with respect to Lebesgue measure, $\E u=\E^{a}u+\E^{s}u$, the above trace theorem particularly implies that the trivial extension $\overline{u}$ of $u$ to $\R^{n}$ satisfies
\begin{align*}
\E \overline{u} & =\E^{a}\overline{u}+\E^{s}\overline{u}=:\mathscr{E}\overline{u} \mathscr{L}^{n} + \E^{s}\overline{u} = \mathscr{E}u \mathscr{L}^{n}\mres\Omega + \E^{s}u\mres\Omega + (\tr(u)\odot\nu_{\partial\Omega})\mathcal{H}^{n-1}\mres\partial\Omega, 
\end{align*}
where $\nu_{\partial\Omega}$ is the outward unit normal to the boundary of the Lipschitz set $\Omega\subset\R^{n}$ and $\mathscr{E}u$ is the symmetric part of the approximate gradient of $u$; see \cite{AnzeGiaq,ACD,ST,Baba} for more information. 

By Ornstein's Non--Inequality, we have $\ld(\Omega)\not\hookrightarrow\sobo^{1,1}(\Omega;\R^{n})$ and thus $\bd(\Omega)\not\hookrightarrow \bv(\Omega)$ also. However, some additional information is available when passing to fractional spaces. We recall that, given $1\leq p < \infty$ and $0<s<1$, a measurable function $u\colon\Omega\to\R^{N}$ belongs to the fractional Sobolev space $\sobo^{s,p}(\Omega;\R^{N})$ if and only if $u\in\lebe^{p}(\Omega;\R^{N})$ and the \emph{Gagliardo seminorm} of $u$ is finite, i.e.,
\begin{align*}
[u]_{\sobo^{s,p}(\Omega;\R^{N})}^{p}:=\iint_{\Omega\times\Omega}\frac{|u(x)-u(y)|^{p}}{|x-y|^{n+sp}}\dif\,(x,y)<\infty. 
\end{align*}
The full norm on $\sobo^{s,p}$ is then given by $\|u\|_{\sobo^{s,p}}:=\|u\|_{\lebe^{p}}+[u]_{\sobo^{s,p}}$. We will also need the Besov spaces to be recalled next. Let $0<\alpha<1$ and $1\leq p,q \leq \infty$. In this situation, we define for $u\in\lebe_{\locc}^{1}(\R^{n};\R^{n})$ as above 
\begin{align*}
& [u]_{\besov_{p,q}^{\alpha}(\R^{n})}:=\sum_{s=1}^{n}\left(\int_{0}^{\infty}\left(\frac{\|\tau_{s,h}u\|_{\lebe^{p}(\R^{n})}}{t^{\alpha}}\right)^{q}\frac{\dif t}{t}\right)^{\frac{1}{q}} &\;\text{if}\;1\leq p,q<\infty,\\
& [u]_{\besov_{p,\infty}^{\alpha}(\R^{n})}:=\sup_{h\neq 0}\max_{s\in\{1,...,n\}}\frac{\|\tau_{s,h}u\|_{\lebe^{p}(\R^{n})}}{h^{\alpha}}&\;\text{if}\;1\leq p <\infty,\,q=\infty. 
\end{align*}
These quantities are referred to ($(\alpha,p,q)$-)Besov seminorms. The full Besov norms then are given by $\|u\|_{\besov_{p,q}^{\alpha}}:=\|u\|_{\lebe^{p}}+[u]_{\besov_{p,q}^{\alpha}}$, and we say that $u$ belongs to the Besov space $\besov_{p,q}^{\alpha}(\R^{n};\R^{n})$ if and only if $\|u\|_{\besov_{p,q}^{\alpha}}<\infty$. For the purposes of this paper, if $[u]_{\besov_{p,q}^{\alpha}}<\infty$, then we say that $u\in\dot{\besov}_{p,q}^{\alpha}$, the corresponding homogeneous Besov space. Note that $\besov_{p,p}^{\alpha}\simeq \sobo^{\alpha,p}$, and we shall sometimes call $\mathcal{N}^{\alpha,p}:=\besov_{p,\infty}^{\alpha}$ the $(\alpha,p)$-Nikolski\u{\i} space. The localised versions of these spaces are defined in the obvious manner. As a consequence of \cite[Thm.~2.7.1]{Triebel0}, we obtain
\begin{lemma}\label{lem:besovembedding}
Let $0<s<1$. Then we have $(\besov_{1,\infty}^{s})_{\locc}(\R^{n})\hookrightarrow \lebe_{\locc}^{q}(\R^{n})$ for any $1\leq q <\frac{n}{n-s}$. 
\end{lemma}
We refer the reader to \cite[Chpt.~4]{AdHe} and \cite[Chpts.~1 and 2]{Triebel0} for more background information on these spaces. Invoking the fractional Sobolev spaces, we have that both $\ld_{\locc}(\Omega)$ and $\bd_{\locc}(\Omega)$ continuously embed into $\sobo_{\locc}^{s,1}(\Omega;\R^{n})$ for any $0<s<1$; see Proposition \ref{prop:cheapembedding} below. For the sake of clarity of exposition, we shall sketch its proof in the appendix, however, note that the result itself can be strengthened using \textsc{Van Schaftingen}'s general theory of cancelling operators \cite{VS} (also see Lemma \ref{lem:bdembedding}).

These statements in turn rest on the \emph{Smith representation formula} \cite{Smith}: Given $u=(u^{1},...,u^{n})\in\hold_{c}^{\infty}(\R^{n};\R^{n})$, we may write 
\begin{align}\label{eq:smith}
u^{k}=\frac{2}{n\omega_{n}}\sum_{i,j=1}^{n}\frac{\partial^{2}u^{k}}{\partial x_{j}\partial x_{j}}*K_{ij},\qquad\text{where}\;K_{ij}(x)=\frac{x_{i}x_{j}}{|x|^{n}}\;\text{for}\;x\in\R^{n}\setminus\{0\}. 
\end{align} 
Setting $\sg(u):=(\sg(u)_{jk})_{jk}$, we observe that
\begin{align*}
\frac{\partial^{2}u^{k}}{\partial x_{i}\partial x_{j}}=\frac{\partial\sg(u)_{jk}}{\partial x_{i}}-\frac{\partial\sg(u)_{ij}}{\partial x_{k}}+\frac{\partial \sg(u)_{ki}}{\partial x_{j}}
\end{align*}
and hence, inserting this relation into \eqref{eq:smith}, we obtain after an integration by parts
\begin{align}\label{eq:representation}
u^{k}=\frac{2}{n\omega_{n}}\sum_{i,j=1}^{n}\big(\sg(u)_{jk}*\frac{\partial K_{ij}}{\partial x_{i}}-\sg(u)_{ij}*\frac{\partial K_{ij}}{\partial x_{k}}+\sg(u)_{ki}*\frac{\partial K_{ij}}{\partial x_{j}}\big)
\end{align}
for all $k=1,...,n$. This formula can be established by means of Fourier analysis and, upon differentiating, indicates the failure of Ornstein's Non-Inequality. 
We record the following result, an elementary proof of which is presented for the reader's convenience in the appendix, cf.~Section~\ref{sec:appendix}: 
\begin{proposition}[$\bd\hookrightarrow \sobo_{\locc}^{s,1}$]\label{prop:cheapembedding}
Let $\Omega$ be an open subset of $\R^{n}$ and $K$ a relatively compact subset of $\Omega$. Then for every $0<s<1$ there exists a constant $C=C(K,s)>0$ such that 
\begin{align}
\|u\|_{\sobo^{s,1}(K;\R^{n})}\leq C \|u\|_{\bd(\Omega)}
\end{align}
holds for all $u\in\bd(\Omega)$. 
\end{proposition}
Finally, we recall that for any connected and open set $\Omega\subset\R^{n}$, the nullspace of $\sg$ is given by the space of \emph{rigid deformations}
\begin{align*}
\mathcal{R}(\Omega):=\big\{r\colon\Omega \ni x\mapsto Ax+b\colon\;A\in\R^{n\times n},\;A^{\mathsf{T}}=-A,\;b\in\R^{n}\big\}.
\end{align*}
\subsection{Convex Analysis}\label{sec:convexanalysismain}
To exploit the convexity of the variational integrals studied in this paper, we shall now record some facts about the dual problem which will turn out useful for the study of the differentiability properties of minima for the primal problem. Given a real Banach space $X$ and a function $g\colon X \to\overline{\R}$, we recall its  \emph{polar function} $g^{*}\colon X^{*}\to\overline{\R}$ given by 
\begin{align*}
g^{*}(x^{*}):=\sup_{x\in X}\big(\langle x^{*},x\rangle_{X^{*}\times X} - g(x)\big),\qquad x^{*}\in X^{*}
\end{align*}
and its \emph{bipolar function} $g^{**}\colon X \to \overline{\R}$ given by 
\begin{align*}
g^{**}(x):=\sup_{x^{*}\in X^{*}}\Big(\langle x^{*},x\rangle_{X^{*}\times X} - g^{*}(x^{*})\Big),\qquad x\in X. 
\end{align*}
Note that if $g$ is lower semicontinuous, proper and convex, then $g=g^{**}$. In view of the convex minimisation problem \eqref{eq:varprin}, we note that since $f\colon\R^{n\times n}\to\R$ is lower semicontinuous and convex, $\partial f(\xi)\neq\emptyset$ for some $\xi\in\R^{n\times n}$ implies the duality relation
\begin{align*}
\eta\in\partial f(\xi)\Leftrightarrow f(\xi)+f^{*}(\eta)=\langle \eta,\xi\rangle. 
\end{align*}
Since $f$ is assumed to be of class $\hold^{2}$, the preceding relation implies that
\begin{align}\label{eq:dualityrel}
f(\xi)+f^{*}(f'(\xi))=\langle f'(\xi),\xi\rangle\qquad\text{holds for all}\;\xi\in\R^{n\times n}. 
\end{align}
Let us recall that the \emph{Lagrangian} $\ell$ is given by 
\begin{align}
\begin{split}
\ell(w,\chi) = \ell(u_{0}+\varphi,\chi) & := \int_{\Omega}\langle \chi,\bm{\varepsilon}(w)\rangle \dif x - \int_{\Omega}f^{*}(\chi)\dif x  = \ell(u_{0},\chi)+\int_{\Omega}\langle\chi,\bm{\varepsilon}(\varphi)\rangle\dif x, 
\end{split}
\end{align}
where $(w,\chi)=(u_{0}+\varphi,\chi)\in (u_{0}+\ld_{0}(\Omega))\times\lebe^{\infty}(\Omega;\R^{n\times n})$. Consequently, the \emph{dual functional} $\mathfrak{R}\colon \lebe^{\infty}(\Omega;\R^{n\times n})\to\overline{\R}$ is given by 
\begin{align}
\mathfrak{R}[\chi]:=\inf\big\{\ell(w,\chi)\colon\;w\in u_{0}+\ld_{0}(\Omega)\big\},\qquad \chi\in\lebe^{\infty}(\Omega;\R_{\sym}^{n\times n}), 
\end{align}
and the \emph{dual problem} is given by 
\begin{align}
\text{ maximise}\;\;\mathfrak{R}\;\text{over}\;\lebe_{\di}^{\infty}(\Omega;\R_{\sym}^{n\times n}):=\big\{\eta\in\lebe^{\infty}(\Omega;\R_{\sym}^{n\times n})\colon \di(\eta)\equiv 0\;\text{in}\;\mathscr{D}'(\Omega;\R^{n}\big\}, 
\end{align}
where $\di$ is the row-wise distributional divergence. Let us note that the choice $\lebe_{\di}^{\infty}$ instead of $\lebe^{\infty}$ stems from the fact that if $\eta\in (\lebe^{\infty}\setminus\lebe_{\di}^{\infty})(\Omega;\R_{\sym}^{n\times n})$, then $\mathfrak{R}[\eta]=-\infty$. So these $\eta$ are irrelevant to the maximisation problem at our disposal. For the purposes of Section~\ref{sec:convdual} it suffices to record that if $f\in\hold^{2}(\R_{\sym}^{n\times n})$ satisfies \eqref{eq:lg}, then we have the minimax principle
\begin{align}
\inf_{u_{0}+\ld_{0}(\Omega)}\mathfrak{F}=\max_{\lebe_{\di}^{\infty}(\Omega;\R_{\sym}^{n\times n})}\mathfrak{R}. 
\end{align}
This follows from abstract duality theory, cf. \textsc{Ekeland \& T\'{e}mam} \cite[Chpts.~III.4 and ~IV.1]{ET}.

\subsection{On the Space $(\sobo_{0}^{1,\infty})^{*}$}\label{sec:dualdef}
In order to work with suitably weak perturbations when applying the Ekeland variational principle, we record some properties of the dual space $(\sobo_{0}^{1,\infty})^{*}$ which seems natural for our purposes in the main body of the paper. A distribution $T\in\mathscr{D}'(\Omega;\R^{n})$ belongs to $(\sobo_{0}^{1,\infty}(\Omega;\R^{n}))^{*}$ if and only if the norm 
\begin{align}
\|T\|_{(\sobo_{0}^{1,\infty})^{*}}:=\sup\big\{ \langle T,\varphi\rangle\colon\;\varphi\in\sobo_{0}^{1,\infty}(\Omega;\R^{n})\;\text{and}\;\|\varphi\|_{\sobo_{0}^{1,\infty}(\Omega;\R^{n})}\leq 1\big\}<\infty,
\end{align}
whenever this expression makes sense; here, we work with the gradient norm 
\begin{align*}
\|\varphi\|_{\sobo_{0}^{1,\infty}(\Omega;\R^{n})}:=\|\nabla \varphi\|_{\lebe^{\infty}(\Omega;\R^{n\times n})}\qquad\text{for}\;\varphi\in \sobo_{0}^{1,\infty}(\Omega;\R^{n}).
\end{align*}
As a dual space, $(\sobo_{0}^{1,\infty}(\Omega;\R^{n}))^{*}$ is complete.
\begin{lemma}\label{lem:negsob}
Let $\Omega\subset\R^{n}$ be open and $K\subset \Omega$ a relatively compact Lipschitz subset of $\Omega$. Let $v\in\lebe^{1}(\Omega;\R^{n})$. Then for any $s=1,...,n$ and any $0<|h|<\dista(K,\partial\Omega)$ we have 
\begin{align*}
\|\Delta_{s,h}v\|_{(\sobo^{1,\infty}(K;\R^{n}))^{*}}\leq \|v\|_{\lebe^{1}(\Omega;\R^{n})}. 
\end{align*}
\end{lemma}
\begin{proof}
Let $\varphi\in\sobo_{0}^{1,\infty}(K;\R^{n})$ be arbitrary with $\|\varphi\|_{\sobo_{0}^{1,\infty}(K;\R^{n})}\leq 1$. Using integration by parts for difference quotients, we estimate
\begin{align*}
|\langle \Delta_{s,h}v,\varphi\rangle | = |\langle v,\Delta_{s,h}^{-}\varphi\rangle\| \leq \|v\|_{\lebe^{1}(\Omega;\R^{n})}\|\Delta_{s,h}\varphi\|_{\lebe^{\infty}(\Omega;\R^{n})}\leq \|v\|_{\lebe^{1}(\Omega;\R^{n})}\|\nabla \varphi\|_{\lebe^{\infty}(\Omega;\R^{n})} 
\end{align*}
and hence passing to the supremum over all admissible test maps $\varphi$ yields the claim. 
\end{proof}
\subsection{A $V$-function estimate}
We conclude this preliminary section by giving an version of an estimate for an auxiliary function in the spirit of \textsc{Acerbi \& Fusco} \cite{AcerbiFusco} that shall turn out convenient for our purposes. For $\alpha>0$ and $M\in\mathbb{N}$, we hereafter introduce the auxiliary map $V_{\alpha}\colon \R^{M}\to \R^{M}$ by
\begin{align}\label{eq:vdef}
V_{\alpha}(\xi):= (1+|\xi|^{2})^{\frac{1-\alpha}{2}}\xi,\qquad\xi\in \R^{M}. 
\end{align}
\begin{lemma}\label{lem:valpha}
Let $1<\alpha <2$ and define $V_{\alpha}$ by \eqref{eq:vdef}. Then we have for any measurable function $v\colon\R^{n}\to\R^{M}$, $h\in\R$ and $e_{s}\in\R^{n}$ with $|e_{s}|=1$ the estimate
\begin{align*}
|\tau_{s,h}V_{\alpha}(v(x))| \sim (1+|v(x+he_{s})|^{2}+|v(x)|^{2})^{\frac{1-\alpha}{2}}|\tau_{s,h}v(x)|.
\end{align*}
Moreover, there exists a constant $c>0$ such that for all $\xi\in\R^{M}$ there holds 
\begin{align*}
\min\{|\xi|,|\xi|^{2-\alpha}\}\leq c V_{\alpha}(\xi).
\end{align*}
Lastly, if $\Omega$ is an open and bounded set and $u\colon\Omega\to\R^{M}$ satisfies $V_{\alpha}(u)\in\lebe^{p}(\Omega;\R^{M})$, then we have 
\begin{align}\label{eq:Valphaintbound}
\int_{\Omega}|u|^{(2-\alpha)p}\dif x \leq \mathscr{L}^{n}(\Omega)+c(p)\int_{\Omega}|V_{\alpha}(u)|^{p}\dif x, 
\end{align}
where $c(p)>0$ is a constant depending only on $p$. 
\end{lemma}

\section{An Embedding for $\bd\cap\bmo$}\label{sec:emb}
In this section we prove an embedding result for $\bd\cap\bmo$ into certain fractional Sobolev spaces which will constitute a substantial part of the proof of the main theorem. The proof combines an argument firstly utilised by \textsc{Dorronsoro} \cite{Doro} -- that, from a regularity perspective has also proven useful in different contexts, cp.~\cite{KM2} -- and an embedding result of $\bd$ into certain fractional Sobolev spaces. Let us recall that a locally integrable map $u\colon\R^{n}\to\R^{N}$ belongs to $\bmo(\R^{n};\R^{N})$ if and only if its sharp (centered) maximal function given by
\begin{align}\label{eq:sharp}
(\mathcal{M}^{\#}u)(x):=\sup_{\substack{Q\,\text{cube centered at}\,x}}\;\dashint_{Q}|u-(u)_{Q}|\dif y,\qquad x\in\R^{n}, 
\end{align}
belongs to $\lebe^{\infty}(\R^{n})$. When are working on a domain $\Omega$, then we say that a measurable map $u\colon\Omega\to\R^{N}$ belongs to $\bmo_{\locc}(\Omega;\R^{N})$ provided for each $K\Subset \Omega$ there holds $\mathcal{M}_{K}^{\#}u\in\lebe^{\infty}(K)$, where
\begin{align}\label{eq:sharp1}
(\mathcal{M}_{K}^{\#}u)(x):=\sup_{\substack{Q\subset\Omega\,\\ Q\,\text{cube centered at}\,x}}\;\dashint_{Q}|u-(u)_{Q}|\dif y,\qquad x\in K. 
\end{align}
The main result of this section is then as follows. \index{Interpolation! for $\bd\cap\bmo$}
\begin{theorem}\label{thm:bdbmoemb}
Let $n\geq 2$. Let $1<p<\infty$ and $\varepsilon>0$ such that 
\begin{align}\label{eq:pepsiloncondition}
0<\varepsilon < \min\left\{\frac{(n-1)\big(1-\frac{1}{p}\big)}{1+pn-p},\frac{1}{p}\right\}.
\end{align}
Then 
\begin{align}\label{eq:mainbdbmo}
\bd(\R^{n})\cap\bmo(\R^{n};\R^{n})\subset \sobo^{\frac{1}{p}-\varepsilon,p}(\R^{n};\R^{n}). 
\end{align}
\end{theorem}
Before embarking on the proof of Theorem \ref{thm:bdbmoemb}, we wish to make some remarks.
\begin{remark}[{\cite[Rem.~3]{BBM}}]\label{rem:BBMexample}\emph{
It is important to note that the preceding theorem is \emph{false} for $\varepsilon=0$. Indeed, if it was true in this case, the injections $\sobo^{1,1}(\R^{n};\R^{n})\hookrightarrow \bd(\R^{n})$ and $\lebe^{\infty}(\R^{n};\R^{n})\hookrightarrow \bmo(\R^{n};\R^{n})$ would yield $(\sobo^{1,1}\cap\lebe^{\infty})(\R^{n};\R^{n})\hookrightarrow \sobo^{\frac{1}{p},p}(\R^{n};\R^{n})$. However, as pointed out by \textsc{Bourgain, Brezis \& Mironescu} \cite{BBM}, this embedding in general fails: Indeed, a localisation argument would then yield $(\sobo^{1,1}\cap\lebe^{\infty})((-1,1))\hookrightarrow \sobo^{\frac{1}{p},p}((-1,1))$ for any $1<p<\infty$. Consider the sequence $(u_{k})$ given by 
\begin{align*}
u_{k}(x):=\begin{cases} -1&\;\text{if}\;-1<x\leq -\frac{1}{2k}\\
2kx&\;\text{if}\;-\frac{1}{2k}\leq x \leq \frac{1}{2k}\\
1&\;\text{if}\;\frac{1}{2k}\leq x <1. 
\end{cases}
\end{align*}
Then $(u_{k})$ is uniformly bounded both in $\sobo^{1,1}((-1,1))$ and $\lebe^{\infty}((-1,1))$. Moreover, it converges weakly* in $\bv((-1,1))$ to $\sgn$ which, however, does not belong to $\sobo^{\frac{1}{2},2}((-1,1))$. 
}
\end{remark}
To prove the theorem, we first recall from \cite{Doro} a mean-value characterisation of $\besov_{p,q}^{s}$. This turns out useful for the proof of Theorem~\ref{thm:bdbmoemb} as it allows to access the additional $\bmo$--regularity which in turn is defined in terms of maximal operators. 
\begin{lemma}\label{lem:Doro}
Let $0<s<1$ and $1\leq p< \infty$. A function $u\in\lebe^{p}(\R^{n})$ belongs to $\sobo^{s,p}(\R^{n})$ if and only if 
\begin{align}\label{eq:dorronsoro}
[u]_{s,p}^{*}:=\left(\int_{\R^{n}}\int_{0}^{\infty}\left\vert\sup_{\substack{Q\ni x \\ |Q|=t^{n}}}\dashint_{Q}\frac{|u-(u)_{Q}|^{p}}{t^{s}}\dif \xi\right\vert^{p} \frac{\dif t}{t}\dif x\right)^{\frac{1}{p}} < \infty. 
\end{align}
Moreover, the expression on the left is equivalent to the usual Gagliardo--seminorm $[\cdot]_{s,p}$.
\end{lemma}
\begin{proof}
It suffices to note that $\besov_{p,p}^{s}\simeq \sobo^{s,p}$ for all $0<s<1$ and $1<p<\infty$. By \cite[Thm.~1]{Doro}, the claim follows. 
\end{proof}
The supremum appearing in the integrand of \eqref{eq:dorronsoro} is taken over all cubes having sidelength $t$ and containing $x$. For the following it is important to bound such quantities in terms of centered maximal operators, and hence we briefly pause and give the required modification. We hereafter recall from \cite{DS} that for a locally integrable map $h\colon\R^{n}\to\R^{N}$ its \emph{sharp (centered) maximal operator of order} $0<\alpha\leq 1$ is given by 
\begin{align}\label{eq:sharpfrac}
h_{\alpha}^{\#}(x):=(\mathcal{M}_{\alpha}^{\#}h)(x):=\sup_{Q\,\text{cube centered at}\,x}\frac{1}{\ell(Q)^{n+\alpha}}\int_{Q}|h-(h)_{Q}|\dif y,\qquad x\in\R^{n}, 
\end{align}
where $\ell(Q)$ is the sidelength of the cube $Q$.
\begin{lemma}\label{lem:dororeduction}
For each $\alpha>0$ there exists a number $C=C(n,\alpha)>0$ such that for all $u\in\lebe_{\locc}^{1}(\R^{n})$, all $x_{0}\in\R^{n}$ and all $t>0$ we have 
\begin{align*}
\frac{1}{t^{\alpha}}\sup\left\{\dashint_{Q}|u-(u)_{Q}|\dif y\colon\;Q\ni x_{0},\;\mathscr{L}^{n}(Q)=t^{n}\right\}\leq C(\mathcal{M}_{\alpha}^{\#}u)(x_{0}).
\end{align*}
provided both sides are well--defined and finite. 
\end{lemma}
The easy proof of this lemma is given for the reader's convenience in the appendix, cf. Section~\ref{sec:appendix}. The second ingredient is an embedding result of $\bd$ into fractional Sobolev spaces: 
\begin{lemma}\label{lem:bdembedding}
Let $n\geq 2$ and $0<s<1$. Then $\bd(\R^{n})\hookrightarrow \sobo^{s,\frac{n}{n-1+s}}(\R^{n};\R^{n})$. Moreover, there exists a constant $C>0$ such that for every ball $\ball$ and every $u\in\bd(\ball)$ there exists $R\in\mathcal{R}(\ball)$ such that 
\begin{align*}
\|u-R\|_{\sobo^{s,\frac{n}{n-1+s}}(\ball;\R^{n})}\leq C |\E u|(\ball).
\end{align*}
\end{lemma}
\begin{proof}
In the terminology of \cite{VS}, the symmetric gradient is an elliptic and cancelling operator, cf.~\cite[Prop.~6.4]{VS}. 
 Hence, by \cite[Thm.~8.1]{VS}, there exists a constant $C>0$ such that 
\begin{align*}
\|\varphi\|_{\dot{\sobo}^{s,n/(n-1+s)}(\R^{n};\R^{n})}\leq C \|\sg(\varphi)\|_{\lebe^{1}(\R^{n};\R_{\sym}^{n\times n})}
\end{align*}
holds for all $\varphi\in\hold_{c}^{\infty}(\R^{n};\R^{n})$, where $\dot{\sobo}^{s,\frac{n}{n-1+s}}(\R^{n};\R^{n})$ denotes the respective homogeneous fractional Sobolev space. For the general statement, let $u\in\bd(\Omega)$ and choose a sequence $(u_{k})\subset \hold_{c}^{\infty}(\R^{n};\R^{n})$ such that $u_{k}\to u$ strictly and pointwisely $\mathscr{L}^{n}$--a.e.  as $k\to\infty$. Then we obtain, using Fatou's lemma, 
\begin{align*}
\|u\|_{\dot{\sobo}^{s,\frac{n}{n-1+s}}} & \leq \liminf_{k\to\infty} \|u_{k}\|_{\dot{\sobo}^{s,\frac{n}{n-1+s}}} \leq C \liminf_{k\to\infty} \|\sg(u_{k})\|_{\lebe^{1}} = C|\E u|(\R^{n}). 
\end{align*}
Now we use the fact that on $\hold_{c}^{\infty}(\R^{n};\R^{n})$, the homogeneous fractional Sobolev norm is equivalent to the Gagliardo seminorm and arrive at the desired estimate $[u]_{\sobo^{s,p}(\R^{n};\R^{n})}\leq C |\E u|(\R^{n})$. 
Finally, since $1<\frac{n}{n-1+s}<\frac{n}{n-1}$, we have $\lebe^{1}\cap\lebe^{\frac{n}{n-1}}\hookrightarrow \lebe^{\frac{n}{n-1+s}}$ by standard interpolation on $\lebe^{p}$--spaces. Then we use \textsc{Strauss}' embedding \cite{Strauss} $\bd(\R^{n})\hookrightarrow (\lebe^{1}\cap\lebe^{\frac{n}{n-1}})(\R^{n})$. In conclusion, we obtain $\|u\|_{\lebe^{\frac{n}{n-1+s}}(\R^{n};\R^{n})}\leq C\|u\|_{\bd(\R^{n})}$ and in conjunction with the first part of the proof, establishes $\|u\|_{\sobo^{s,\frac{n}{n-1+s}}(\R^{n};\R^{n})}\leq C \|u\|_{\bd(\R^{n})}$. Now, since $\ball$ has Lipschitz boundary, we may pick a bounded linear extension operator $\ext\colon\bd(\ball)\to\bd(\R^{n})$. In consequence, we find by the usual Poincar\'{e} inequality on $\bd$ (cf.~\cite[Thm.~6.5]{ACD}) that there exists $R\in\mathcal{R}(\ball)$ with
\begin{align*}
\|(u-R)\|_{\sobo^{s,\frac{n}{n-1+s}}(\ball)} & \leq \|\ext(u-R)\|_{\sobo^{s,\frac{n}{n-1+s}}(\R^{n})} \\ & \leq \|\ext(u-R)\|_{\bd(\R^{n})} \leq \|u-R\|_{\bd(\ball)} \leq C |\E u|(\ball).
\end{align*}
The proof is complete. 
\end{proof}
\begin{remark}\label{rem:Ws1emb}\emph{
A similar (local) embedding can be achieved for $\sobo^{s,1}$, but this does not follow from the previous lemma as $\sobo^{s,q}\not\hookrightarrow \sobo^{s,p}$ provided $q>p$, cf.~\textsc{Mironescu \& Sickel} \cite{MironescuSickel}.}
\end{remark}
We now come to the
\begin{proof}[Proof of Theorem \ref{thm:bdbmoemb}]
Let $u\in(\bd\cap\bmo)(\R^{n})$. For fixed $x\in\R^{n}$, a cube $Q$ of sidelength $t$ and $x\in Q$, we denote $g_{x,t,Q}(\xi):=u(\xi)-(u)_{Q}$, $\xi\in\R^{n}$. For $1\leq p <\infty$, we recall the (inhomogeneous) \emph{Calder\'{o}n space} $\mathscr{C}^{\alpha,p}(\R^{n})$ defined by 
\begin{align*}
\mathscr{C}^{\alpha,p}(\R^{n}):=\big\{v\in\lebe^{p}(\R^{n})\colon\;\mathcal{M}_{\alpha}^{\#}v\in\lebe^{p}(\R^{n})\big\}
\end{align*}
and equip it with the canonical norm $\|v\|_{\mathscr{C}^{\alpha,p}}:=\|v\|_{\lebe^{p}}+\|v_{\alpha}^{\#}\|_{\lebe^{p}}$ (cp.~Section~\ref{sec:appendix}); the inhomogeneous Calder\'{o}n space $\dot{\mathscr{C}}^{\alpha,p}$ is given by the closure of $\hold_{c}^{\infty}$ with respect to the seminorm $\|\mathcal{M}_{\alpha}^{\#}\cdot\|_{\lebe^{p}}$. Our argument is centered around the Dorronsoro--type characterisation of the Sobolev spaces $\sobo^{s,p}$, Lemma \ref{lem:Doro}. In the situation of the theorem, we put $s:=\frac{1}{p}-\varepsilon$. For $\mathscr{L}^{n}$--a.e. $x\in\R^{n}$, let $\delta(x)>0$ be arbitrary. We split the right hand side term of \eqref{eq:dorronsoro} as 
\begin{align*}
|[u]_{s,p}^{*}|^{p} & :=\int_{\R^{n}}\int_{0}^{\delta(x)}\left\vert\sup_{\substack{Q\ni x \\ |Q|=t^{n}}}\frac{1}{t^{s}}\dashint_{Q}|g_{x,t,Q}(\xi)|\dif \xi\right\vert^{p}\frac{\dif t}{t}\dif x \\ & +\int_{\R^{n}}\int_{\delta(x)}^{\infty}\left\vert\sup_{\substack{Q\ni x \\ |Q|=t^{n}}}\frac{1}{t^{s}}\dashint_{Q}|g_{x,t,Q}(\xi)|\dif \xi\right\vert^{p}\frac{\dif t}{t}\dif x =: \int_{\R^{n}}\mathbf{I}_{\delta}(x)+\mathbf{II}_{\delta}(x)\dif x 
\end{align*}
with an obvious definition for $\mathbf{I}_{\delta}(x)$ and $\mathbf{II}_{\delta}(x)$. Firstly, we have 
\begin{align*}
\mathbf{I}_{\delta}(x) & = \int_{0}^{\delta(x)}\left\vert\sup_{\substack{Q\ni x \\ |Q|=t^{n}}}\frac{1}{t^{s+\varepsilon}}\dashint_{Q}|g_{x,t,Q}(\xi)|\dif \xi\right\vert^{p}\frac{\dif t}{t^{1-\varepsilon p}} \\ & \leq C \int_{0}^{\delta(x)}(u_{s+\varepsilon}^{\#}(x))^{p}\frac{\dif t}{t^{1-\varepsilon p}}\;\;\;\;\;\text{(by Lemma~\ref{lem:dororeduction})}\\
& \leq \frac{C}{\varepsilon p}\delta(x)^{\varepsilon p}(u_{s+\varepsilon}^{\#}(x))^{p}. 
\end{align*}
On the other hand, we have by definition of $\mathcal{M}^{\#}$
\begin{align*}
\mathbf{II}_{\delta}(x) & = \int_{\delta(x)}^{\infty}\left\vert\sup_{\substack{Q\ni x \\ |Q|=t^{n}}}\dashint_{Q}|g_{x,t,Q}(\xi)|\dif \xi\right\vert^{p}\frac{\dif t}{t^{1+sp}} \\
& \leq C \int_{\delta(x)}^{\infty}(u^{\#}(x))^{p}\frac{\dif t}{t^{1+sp}}\;\;\;\;\;\text{(by Lemma~\ref{lem:dororeduction})}\\ &\leq \frac{C}{sp}\delta(x)^{-sp}(u^{\#}(x))^{p}. 
\end{align*}
Collecting estimates, we therefore find 
\begin{align}\label{eq:intermediate-sp-estimate}
[u]_{s,p}^{*} \leq C(s,p,\varepsilon) \int_{\R^{n}}\delta(x)^{\varepsilon p}(u_{s+\varepsilon}^{\#}(x))^{p} + \delta(x)^{-sp}(u^{\#}(x))^{p}\dif x. 
\end{align}
We choose for $\mathscr{L}^{n}$--a.e. $x\in\R^{n}$
\begin{align*}
\delta(x):=\left(\frac{u^{\#}(x)}{u_{s+\varepsilon}^{\#}(x)}\right)^{\frac{1}{s+\varepsilon }}
\end{align*}
and note that we may assume without loss of generality that $u_{s+\varepsilon}^{\#}(x)>0$ since $u$ is constant otherwise and thus the claim is trivial. Inserting this choice of $\delta$ into \eqref{eq:intermediate-sp-estimate}, we obtain 
\begin{align*}
[u]_{s,p}^{*} & \leq C(s,p,\varepsilon)\int_{\R^{n}}(u^{\#}(x))^{\frac{\varepsilon p}{s+\varepsilon}}(u_{s+\varepsilon}^{\#}(x))^{p-\frac{\varepsilon p}{s+\varepsilon}}\dif x \leq C\|u^{\#}\|_{\lebe^{\infty}(\R^{n})}^{\frac{\varepsilon p}{s+\varepsilon}}\int_{\R^{n}}(u_{s+\varepsilon}^{\#})^{\frac{ps}{s+\varepsilon}}\dif x =(**). 
\end{align*}
Now note that by Lemma \ref{lem:logconvex} below, the fractional maximal functions are $\log$--convex in their smoothness indices. As a consequence, we obtain for $0<t<s$ and $1<q<\infty$
\begin{align}\label{eq:interpolationembedding}
(\dot{\mathscr{C}}^{s,q}\cap\bmo)(\R^{n})\hookrightarrow \dot{\mathscr{C}}^{t,\frac{sq}{t}}(\R^{n})\qquad\text{for all}\;0<t<s. 
\end{align}
Indeed, write $t=\lambda\cdot 0 + (1-\lambda)s$ with $0<\lambda<1$ to deduce for $1<q<\infty$ and any $v\in (\dot{\mathscr{C}}^{s,q}\cap\bmo)(\R^{n})$ that
\begin{align*}
(v_{t}^{\#}(x))^{r}\leq (v_{0}^{\#}(x))^{\lambda r}(v_{s}^{\#}(x))^{(1-\lambda)r}
\end{align*}
and hence \eqref{eq:interpolationembedding} follows since $(1-\lambda)s=t$ implies $r=sq/t$. 

We return to the estimation of $(**)$. Because $0<\varepsilon p <1$, we have $0<1-\varepsilon p <1$. By assumption, we have $s+\varepsilon=1/p\in (0,1)$ and $ps/(s+\varepsilon)=p(1-\varepsilon p)$. Define $\Phi(\theta):=\theta n/(n-1+\theta)$, $\theta\in (0,1)$. Since $n\geq 2$, there holds $|\Phi|< 1$ and we have both $\lim_{\theta \searrow 0}\Phi(\theta)=0$ and $\lim_{\theta\nearrow 1}\Phi(\theta)= 1$. Clearly, $\Phi$ is continuous and hence $\Phi\colon (0,1)\to (0,1)$ is bijective. Therefore, choosing $\widetilde{\theta}:=(1-\varepsilon p)(n-1)/(n-1+\varepsilon p)$, we see that $\Phi(\widetilde{\theta})=1-\varepsilon p$. With this choice of $\widetilde{\theta}$, the embedding \eqref{eq:interpolationembedding} gives
\begin{align}\label{eq:mostimportantclaim}
(\dot{\mathscr{C}}^{\widetilde{\theta},n/(n-1+\widetilde{\theta})}\cap\bmo)(\R^{n})\hookrightarrow \dot{\mathscr{C}}^{\frac{1}{p},p(1-\varepsilon p)}(\R^{n}) 
\end{align}
since $\widetilde{\theta}>\frac{1}{p}$: Indeed, since $0<\varepsilon p <1-\frac{1}{p}$ by assumption, we deduce $p(1-\varepsilon p)-1>0$, and therefore with $\gamma=\varepsilon p$
\begin{align*}
\widetilde{\theta}=\frac{(1-\gamma)(n-1)}{n-1+\gamma}>\frac{1}{p} & \Leftrightarrow (p-\gamma p)(n-1)>n-1+\gamma \Leftrightarrow pn-\gamma p n - p + \gamma p >n-1+\gamma \\
& \Leftrightarrow pn-p-n+1>\gamma +\gamma pn-\gamma p = \gamma (1+pn-p)\\
& \Leftrightarrow \frac{pn-p-n+1}{1+pn-p}>\gamma = \varepsilon p \\
& \Leftrightarrow \frac{n-1-n/p+1/p}{1+pn-p}>\varepsilon
\end{align*}
which is true by assumption \ref{eq:pepsiloncondition}. Now we choose $\widetilde{\theta}\in (0,1)$ such that there holds $\bd(\R^{n})\hookrightarrow\sobo^{\widetilde{\theta},n/(n-1+\widetilde{\theta})}(\R^{n};\R^{n})$. This is possible by Lemma~\ref{lem:bdembedding} and, using Lemma~\ref{lem:devoresharpleylemma}, embedding the latter Sobolev-Slobodeckji\u{\i} space into the Calder\'{o}n space $\mathscr{C}^{\widetilde{\theta},n/(n-1+\widetilde{\theta})}$, we see that $\bd\cap\bmo\hookrightarrow\mathscr{C}^{\widetilde{\theta},n/(n-1+\widetilde{\theta})}\cap\bmo$. So we are in position to conclude by \eqref{eq:mostimportantclaim}. The proof is complete.
\end{proof}
\begin{remark}\emph{
The preceding proof, in particular, the choice of $x$--dependent $\delta$, uses the so--called \emph{Hedberg trick} as explained in \cite[Prop.~3.1.2]{AdHe}.}
\end{remark}
In the proof of Theorem \ref{thm:bdbmoemb}, we used the $\log$--convexity of the fractional maximal operator with respect to its index, a fact whose proof we give now: 
\begin{lemma}\label{lem:logconvex}
The function $s\mapsto \mathcal{M}_{s}^{\#}v(x)$ is $\log$--convex on $(0,1)$ for any locally integrable function $v\colon\R^{n}\to\R^{N}$. That is, for any $s,t\in (0,1)$ and any $\lambda\in (0,1)$ there holds 
\begin{align}\label{eq:logconvex}
\mathcal{M}_{\lambda s+ (1-\lambda)t}^{\#}v(x)\leq (\mathcal{M}_{ s}^{\#}v(x))^{\lambda}(\mathcal{M}_{t}^{\#}v(x))^{1-\lambda}.
\end{align} 
\end{lemma}
\begin{proof}
Let $v\in\lebe_{\locc}^{1}(\R^{n})$. If the right side of \eqref{eq:logconvex} is infinite, there is nothing to prove, so we may assume without loss of generality that $\mathcal{M}_{s}^{\#}v(x), \mathcal{M}_{t}^{\#}v(x)<\infty$. Let $Q\subset\R^{n}$ be a non--degenerate cube. Then we have 
\begin{align*}
\frac{1}{\ell(Q)^{n+\lambda s + (1-\lambda)s}}\int_{Q}|v-(v)_{Q}|\dif x & =\frac{1}{\ell(Q)^{\lambda(n+s)}}\left(\int_{Q}|v-(v)_{Q}|\dif x\right)^{\lambda}\\ & \times \frac{1}{\ell(Q)^{(1-\lambda)(n+t)}}\left(\int_{Q}|v-(v)_{Q}|\dif x\right)^{1-\lambda} \\ & \leq (\mathcal{M}_{s}^{\#}v(x))^{\lambda}(\mathcal{M}_{t}^{\#}v(x))^{1-\lambda}. 
\end{align*}
We then pass to the supremum over all cubes $Q$ to deduce the claim. 
\end{proof}
For the sake of better traceability, we explicitely note the following. 
\begin{corollary}\label{cor:bvembedding}
Let $1<p<\infty$ and $\varepsilon>0$ such that $p\varepsilon<1$. Then for any $N\geq 1$ we have 
\begin{align}\label{eq:mainbdbmo}
(\bv\cap\bmo)(\R^{n};\R^{N})\subset \sobo^{\frac{1}{p}-\varepsilon,p}(\R^{n};\R^{N}). 
\end{align}
\end{corollary}
The previous corollary follows along the same lines as the proof of Theorem \ref{thm:bdbmoemb}, now using the standard Sobolev embedding $\bv(\R^{n};\R^{N})\hookrightarrow \sobo^{\theta,\frac{n}{n-1+\theta}}(\R^{n};\R^{N})$ for $0<\theta<1$. This slightly improves the embedding $(\bv\cap\lebe^{\infty})(\R^{n})\hookrightarrow \besov_{p,\infty}^{1/p}(\R^{n})$ for $1<p<\infty$ as given in Lemma 38.1 in \textsc{Tartar}'s monograph \cite{Tartar}. 
\section{Viscosity Approximations}\label{sec:main}
\subsection{The Ekeland--type Approximation for $1<\mu<\frac{n+1}{n}$}\label{sec:ekeland}
To avoid manipulations on measures when working with the Euler--Lagrange equation satisfied by the minimiser $u\in\bd(\Omega)$, we shall consider approximate problems which allow us to work with $\ld$--maps first. More precisely, starting from an arbitrary minimising sequence, we shall employ Ekeland's variational principle to construct another minimising sequence which is close to the original sequence, however, features convenient optimality properties. For the reader's convenience, we therefore first recall 
\begin{lemma}[Ekeland Variational Principle, {\cite[Thm.~5.6]{Giusti}}]\label{lem:ekeland}
Let $(X,d)$ be a \emph{complete} metric space and $J\colon X \to\R\cup\{\infty\}$, $J\not\equiv\infty$, a lower semicontinuous functional which is bounded from below. Fix $\varepsilon>0$. If $u\in X$ is such that 
\begin{align*}
J(u)\leq \inf_{X}J+\varepsilon, 
\end{align*}
then there exists $v\in X$ with the following properties: $J(v)\leq J(u)$, $d(u,v)\leq \sqrt{\varepsilon}$ and for all $w\neq v$ we have
\begin{align*}
J(v)<J(w)+\sqrt{\varepsilon}d(v,w). 
\end{align*}
\end{lemma}
Since our strategy to prove uniform higher integrability by means of finite differences relies on suitable Nikolski\u{\i}--type estimates, we will need to apply Ekeland's variational principle with respect to a metric which is considerably weaker than the symmetric--gradient metric $\dif\,(u,v):=\|\sg(u)-\sg(v)\|_{\lebe^{1}(\Omega;\R^{n\times n})}$ on suitable Dirichlet classes. Here we again follow \cite{BS1}, however, invoke the metric induced by the $(\sobo_{0}^{1,\infty})^{*}$--norm as discussed in Section \ref{sec:dualdef}.
\begin{lemma}\label{lem:lsc1}
Given $p>1$, let $F\colon\R_{\sym}^{n\times n}\to \R$ be a convex function such that 
\begin{align}\label{eq:auxgrowthbd}
c|\xi|^{p}-\vartheta\leq F(\xi)\leq \Theta(1+|\xi|^{p})
\end{align}
holds for all $\xi\in\R_{\sym}^{n\times n}$ with three constants $c,\vartheta,\Theta>0$. Given $u_{0}\in \sobo^{1,p}(\Omega;\R^{n})$, the functional 
\begin{align}
\mathscr{F}[u]:=\begin{cases} \displaystyle\int_{\Omega}F(\sg(u))\dif x&\;\text{if}\;u\in u_{0}+\sobo_{0}^{1,p}(\Omega;\R^{n}),\\
+\infty&\;\text{if}\;u\in ((\sobo_{0}^{1,\infty})^{*}\setminus(u_{0}+\sobo_{0}^{1,p}))(\Omega;\R^{n})
\end{cases}
\end{align}
is lower semicontinuous with respect to norm convergence on $(\sobo_{0}^{1,\infty}(\Omega;\R^{n}))^{*}$. 
\end{lemma}
\begin{proof}
Let $u,u_{1},u_{2},...\in (\sobo_{0}^{1,\infty})^{*}(\Omega;\R^{n})$ be such that $u_{k}\to u$ strongly in $(\sobo_{0}^{1,\infty}(\Omega;\R^{n}))^{*}$ as $k\to\infty$. If $\liminf_{l\to\infty}\mathscr{F}[u_{k(l)}]=\infty$, we are done and so we may assume without loss of generality that there exists a subsequence $(u_{k(l)})\subset (u_{k})$ such that $(u_{k(l)})\subset u_{0}+\sobo_{0}^{1,p}(\Omega;\R^{n})$ and $\liminf_{l\to\infty}\mathscr{F}[u_{k(l)}]<\infty$. We now choose another subsequence $(u_{k(l(i))})\subset (u_{k(l)})$ such that $\lim_{i\to\infty}\mathscr{F}[u_{k(l(i))}]=\liminf_{l\to\infty}\mathscr{F}[u_{k(l)}]$, and put $U_{i}:=u_{k(l(i))}$ for $i\in\mathbb{N}$. By virtue of the coercive bound on $F$, we deduce that $(\sg(U_{i}))$ is bounded in $\lebe^{p}(\Omega;\R^{n})$ and so, writing $U_{i}=u_{0}+v_{i}$ with $v_{i}\in\sobo_{0}^{1,p}(\Omega;\R^{n})$ for each $i\in\mathbb{N}$, we use the symmetric gradient variant of Poincar\'{e}'s inequality in $\sobo_{0}^{1,p}$ to obtain
\begin{align*}
\int_{\Omega}|U_{i}|^{p}\dif x \lesssim \int_{\Omega}|u_{0}|^{p}+|v_{i}|^{p}\dif x  \lesssim \int_{\Omega}|u_{0}|^{p}\dif x +  \int_{\Omega}|\sg(v_{i})|^{p}\dif x \leq C\qquad\text{for all}\;i\in\mathbb{N}. 
\end{align*}
As $p>1$, we may use Korn's inequality to deduce that $(U_{i})$ has a subsequence $(U_{i(m)})$ which converges weakly to some $v\in\sobo^{1,p}(\Omega;\R^{n})$ and, by continuity of the trace operator on $\sobo^{1,p}$ with respect to weak convergence, $v\in u_{0}+\sobo_{0}^{1,p}(\Omega;\R^{n})$, too. Since $\Omega$ is assumed to be Lipschitz throughout, using the compact embedding $\sobo^{1,p}(\Omega;\R^{n})\hookrightarrow\hookrightarrow\lebe^{p}(\Omega;\R^{n})$ and passing to a further subsequence if necessary, we can even assume strong convergence $u_{k(l)}\to v$ in $\lebe^{p}(\Omega;\R^{n})$ as $l\to\infty$. Since $\lebe^{p}(\Omega;\R^{n})\hookrightarrow(\sobo_{0}^{1,\infty}(\Omega;\R^{n}))^{*}$ as $\lebe^{p}\hookrightarrow \lebe^{1}\hookrightarrow (\sobo_{0}^{1,\infty})^{*}$, we conclude that $u=v$ $\mathscr{L}^{n}$--a.e.. Now, by virtue of the growth bound \eqref{eq:auxgrowthbd} and convexity of $F$, standard arguments\footnote{In fact, using convexity and Korn's inequality, this follows as in the case for convex $p$-growth full gradient functionals. Alternatively, using that the convex variational integral $\mathscr{F}|_{\sobo^{1,p}}$ is $\mathcal{A}$-quasiconvex with $\mathcal{A}=\curl\curl$ in the terminology of \textsc{Fonseca \& M\"{u}ller} {\cite{FM1}}, the statement follows from \cite[Thm.~3.7]{FM1}, too.} yield lower semicontinuity of $\mathscr{F}|_{\sobo^{1,p}(\Omega;\R^{n})}$ with respect to weak convergence on $\sobo^{1,p}(\Omega;\R^{n})$. Summarising, we then deduce 
\begin{align*}
\mathscr{F}[u] = \mathscr{F}[v] & \leq \liminf_{m\to\infty} \mathscr{F}[U_{i(m)}] \stackrel{(U_{i(m)})\subset (u_{k(l(i))})}{=} \lim_{i\to\infty}\mathscr{F}[u_{k(l(i))}]=\liminf_{l\to\infty}\mathscr{F}[u_{k(l)}]. 
\end{align*}
The proof is complete. 
\end{proof}
Next, a lemma on the growth behaviour of $\mu$-elliptic integrands; recall that by \emph{linear growth} we understand condition~\eqref{eq:lg} throughout. 
\begin{lemma}\label{lem:muellbound}
Let $f\in\hold^{2}(\R_{\sym}^{n\times n})$ be a convex integrand of linear growth with $1<\mu<\infty$. Then, for any $\theta > 0$ there exists $c_{\theta},C_{\theta}>0$ and $\vartheta>0$ such that 
\begin{align*} 
& \theta|\cdot|^{2}-\vartheta\leq f+\theta|\cdot|^{2}\leq C_{\theta}(1+|\xi|^{2}) \qquad\text{for all}\;\xi\in\R_{\sym}^{n\times n}.
\end{align*}
\end{lemma}
The proof of this statement follows immediately from \eqref{eq:lg}. We now come to the precise construction of a good approximation of a given generalised minimiser. Here we follow closely \cite{BS1} with the requisite modifications. Let hereafter $u\in\bd(\Omega)$ be a generalised minimiser of $\mathfrak{F}$. Then, by Proposition~\eqref{eq:nogaps}, we have $\inf\mathfrak{F}[\mathscr{D}]=\min\overline{\mathfrak{F}}[\bd(\Omega)]$. Here, $\mathscr{D}:=u_{0}+\ld_{0}(\Omega)$, and $\overline{\mathfrak{F}}$ is given by \eqref{eq:relaxed} (note that we suppress the subscript $u_{0}$ for notational brievity). We then find a sequence $(w_{k})\subset \mathscr{D}$ such that 
\begin{align}
w_{k}\to u\;\;\text{in}\;\lebe^{1}(\Omega;\R^{n})\;\;\text{and}\;\;(\mathscr{L}^{n},\E w_{k})\to (\mathscr{L}^{n},\E u)\;\;\text{strictly}
\end{align}
as $k\to\infty$. By the \textsc{Reshetnyak} continuity theorem (see Proposition~\ref{prop:Reshetnyak}) and $w_{k}\in \ld(\Omega)$ for all $k\in\mathbb{N}$, we deduce that $\mathfrak{F}[w_{k}]\to \inf \mathfrak{F}[\mathscr{D}]$ as $k\to\infty$ so that $(w_{k})$ indeed is a minimising sequence for $\mathfrak{F}$. Moreover, possibly passing to a subsequence, we may assume that 
\begin{align}\label{eq:closenessass}
\mathfrak{F}[w_{k}]\leq \inf \mathfrak{F}[\mathscr{D}]+\frac{1}{8k^{2}}\qquad\text{for all}\;k\in\mathbb{N}. 
\end{align} 
Next recall that, due to $\mu$--ellipticity and linear growth, $f$ is Lipschitz with some Lipschitz constant $L>0$. 
As to the boundary values, we find a sequence $(u_{k}^{\partial\Omega})\subset \sobo^{1,2}(\Omega;\R^{n})$ satisfying 
\begin{align}\label{eq:bdryapprox}
\|u_{k}^{\partial\Omega}-u_{0}\|_{\ld(\Omega)}\leq \frac{1}{8Lk^{2}}\qquad\text{for all}\;k\in\mathbb{N}
\end{align}
and thus, putting $\mathscr{D}_{k}:=u_{k}^{\partial\Omega}+\sobo_{0}^{1,2}(\Omega;\R^{n})$, we deduce by $w_{k}\in u_{0}+\ld_{0}(\Omega)$ that there exists a sequence $(v_{k})\subset \mathscr{D}_{k}$ such that 
\begin{align*}
\|(v_{k}-u_{k}^{\partial\Omega})-(w_{k}-u_{0})\|_{\ld(\Omega)}\leq \frac{1}{8Lk^{2}}
\end{align*}
and hence 
\begin{align}\label{eq:LDapproximation1}
\|v_{k}-w_{k}\|_{\ld(\Omega)}\leq \frac{1}{4Lk^{2}}\qquad\text{for all}\;k\in\mathbb{N}.
\end{align}
Note that, relying on the extension results from Section~\ref{sec:bd}, such an approximating sequence for the boundary values can be obtained by first extending the boundary values to an $\ld$-map on the entire $\R^{n}$ and then mollifying. Now, since $f$ is Lipschitz with constant $L$, we firstly calculate for arbitrary $\psi\in\sobo_{0}^{1,2}(\Omega;\R^{n})$
\begin{align*}
\inf_{u_{0}+\sobo_{0}^{1,2}(\Omega;\R^{n})}\mathfrak{F} & \leq \mathfrak{F}[u_{0}+\psi]  = \left(\int_{\Omega}f(\sg(u_{0}+\psi))-f(\sg(u_{k}^{\partial\Omega}+\psi))\dif x\right) + \int_{\Omega}f(\sg(u_{k}^{\partial\Omega}+\psi))\dif x\\ & \leq L \int_{\Omega}|u_{0}-u_{k}^{\partial\Omega}|\dif x + \int_{\Omega}f(\sg(u_{k}^{\partial\Omega}+\psi))\dif x  \leq \frac{1}{8k^{2}} + \int_{\Omega}f(\sg(u_{k}^{\partial\Omega}+\psi))\dif x
\end{align*}
so that infimisation over $\psi\in\sobo_{0}^{1,2}(\Omega;\R^{n})$ yields 
\begin{align}\label{eq:Dirichletcompare}
\inf_{\mathscr{D}}\mathfrak{F} = \inf_{u_{0}+\sobo_{0}^{1,2}(\Omega;\R^{n})}\mathfrak{F} \leq \inf_{\mathscr{D}_{k}}\mathfrak{F} + \frac{1}{8k^{2}}. 
\end{align}
Here we have used that, by smooth approximation, $u_{0}+\sobo_{0}^{1,2}(\Omega;\R^{n})$ is norm-dense in $\mathscr{D}$. Thus, again using that $f$ has Lipschitz constant $L$ in conjunction with \eqref{eq:LDapproximation1} in the first, \eqref{eq:closenessass} in the second and \eqref{eq:Dirichletcompare} in the last step, we eventually obtain
\begin{align}\label{eq:bbound}
\mathfrak{F}[v_{k}] \leq \mathfrak{F}[w_{k}]+\frac{1}{4k^{2}}\leq \inf \mathfrak{F}[\mathscr{D}]+\frac{3}{8k^{2}}\leq \inf \mathfrak{F}[\mathscr{D}_{k}]+\frac{1}{2k^{2}}
\end{align}
for all $k\in\mathbb{N}$. Now put, for $k\in\mathbb{N}$,
\begin{align}\label{eq:defAk}
\begin{split}
&f_{k}(\xi):=f(\xi)+\frac{1}{2k^{2}A_{k}}(1+|\xi|^{2}),\qquad\xi\in\R_{\sym}^{n\times n},\\
&A_{k}:=1+\int_{\Omega}\big(1+|\sg(v_{k})|^{2}\big)\dif x.
\end{split}
\end{align}
We define 
\begin{align*}
\mathfrak{F}_{k}[w]:=\begin{cases} \displaystyle\int_{\Omega}f_{k}(\sg(w))\dif x &\;\text{provided}\;w\in\mathscr{D}_{k}\\
+\infty&\;\text{provided}\;w\in (\sobo_{0}^{1,\infty}(\Omega;\R^{n}))^{*}\setminus\mathscr{D}_{k}.
\end{cases}
\end{align*}
Now we aim to apply Lemma~\ref{lem:lsc1} for each $k\in\mathbb{N}$ to the particular choice $p=2$ and $F=f_{k}$. In combination with Lemma~\ref{lem:muellbound}, it is then routine to check that the assumptions of Lemma~\ref{lem:lsc1} are in fact satisfied and hence each $\mathfrak{F}_{k}$ is lower semicontinuous with respect to norm convergence in $(\sobo_{0}^{1,\infty}(\Omega;\R^{n}))^{*}$. Moreover, we find because of $v_{k}\in\mathscr{D}_{k}$ in the first, \eqref{eq:bbound} in the second and by definition of $\mathfrak{F}_{k}$ in the third step
\begin{align}\label{eq:negunifbound}
\mathfrak{F}_{k}[v_{k}]\leq \mathfrak{F}[v_{k}]+\frac{1}{2k^{2}} \stackrel{\eqref{eq:bbound}}{\leq} \inf\mathfrak{F}[\mathscr{D}_{k}]+\frac{1}{k^{2}}\leq \inf \mathfrak{F}_{k}[(\sobo_{0}^{1,\infty}(\Omega;\R^{n}))^{*}]+\frac{1}{k^{2}}. 
\end{align}
We are now in position to apply Ekeland's variational principle, Lemma \ref{lem:ekeland}, to find a sequence $(u_{k})\subset(\sobo_{0}^{1,\infty}(\Omega;\R^{n}))^{*}$ such that 
\begin{align}\label{eq:intermed00}
\begin{split}
&\|u_{k}-v_{k}\|_{(\sobo_{0}^{1,\infty}(\Omega;\R^{n}))^{*}}\leq \frac{1}{k},\\
&\mathfrak{F}_{k}[u_{k}] \leq \mathfrak{F}_{k}[w]+\frac{1}{k}\|w-u_{k}\|_{(\sobo_{0}^{1,\infty}(\Omega;\R^{n}))^{*}}\;\;\text{for all}\;w\in(\sobo_{0}^{1,\infty})^{*}(\Omega;\R^{n})\;\text{and}\;k\in\mathbb{N}. 
\end{split}
\end{align}
Applying the second part of \eqref{eq:intermed00} to $w=v_{k}$, we then find for some $\ell>0$
\begin{align*}
\int_{\Omega}|\sg(u_{k})|\dif x &  \stackrel{\text{Lemma}~\ref{lem:muellbound}}{\leq} \mathfrak{F}[u_{k}]+\ell \leq \mathfrak{F}_{k}[u_{k}]+\ell \\ & \,\;\stackrel{\eqref{eq:intermed00}(ii)}{\leq}  \mathfrak{F}_{k}[v_{k}]+\frac{1}{k}\|v_{k}-u_{k}\|_{(\sobo_{0}^{1,\infty}(\Omega;\R^{n}))^{*}} + \ell \stackrel{\eqref{eq:intermed00}(i)}{\leq}  \mathfrak{F}_{k}[v_{k}]+\frac{1}{k^{2}}+\ell\\
& \;\;\;\;\stackrel{\eqref{eq:negunifbound}}{\leq} \inf \mathfrak{F}_{k}[(\sobo_{0}^{1,\infty}(\Omega;\R^{n}))^{*}]+\frac{2}{k^{2}}+\ell \leq C, 
\end{align*}
where $C>0$ is a finite constant independent of $k\in\mathbb{N}$; note that we clearly have that $\inf \mathfrak{F}_{k}[(\sobo_{0}^{1,\infty}(\Omega;\R^{n}))^{*}]<\infty$.
In particular, we deduce that 
\begin{align}\label{eq:LDbound}
(u_{k})\;\;\text{is uniformly bounded in}\;\ld(\Omega). 
\end{align}
Finally, we record the following lemma on the perturbed Euler--Lagrange equations satisfied by the individual $u_{k}$'s.
\begin{lemma}[Approximate Euler--Lagrange Equation]\label{lem:eulerlagrange}
Let $f_{k}$ and $u_{k}$ be defined as above. Then for all $k\in\mathbb{N}$ we have 
\begin{align}\label{eq:elmain}
\left\vert \int_{\Omega}\langle f'_{k}(\sg(u_{k})),\sg(\varphi)\rangle\dif x \right\vert \leq \frac{1}{k}\|\varphi\|_{(\sobo_{0}^{1,\infty}(\Omega;\R^{n}))^{*}}
\end{align}
for all $\varphi\in\sobo_{0}^{1,2}(\Omega;\R^{n})$.
\end{lemma}
\begin{proof}
Fix $k\in\mathbb{N}$ and let $\varphi\in\sobo_{0}^{1,2}(\Omega;\R^{n})$ be arbitrary. Then for every $\varepsilon>0$ we have $u_{k}\pm\varepsilon\varphi\in\mathscr{D}_{k}$. Consequently, we obtain by the second line of \eqref{eq:intermed00} 
\begin{align*}
\mathfrak{F}_{k}[u_{k}]-\mathfrak{F}_{k}[u_{k}\pm\varepsilon\varphi]\leq \frac{\varepsilon}{k}\|\varphi\|_{(\sobo_{0}^{1,\infty}(\Omega;\R^{n}))^{*}}. 
\end{align*}
This gives 
\begin{align*}
-\frac{1}{k}\|\varphi\|_{(\sobo_{0}^{1,\infty}(\Omega;\R^{n}))^{*}}& \leq \frac{\mathfrak{F}_{k}[u_{k}\pm\varepsilon\varphi]-\mathfrak{F}_{k}[u_{k}]}{\varepsilon} \stackrel{\varepsilon\searrow 0}{\longrightarrow} \pm\left(\int_{\Omega}\langle f'_{k}(\sg(u_{k})),\sg(\varphi)\rangle\dif x \right)
\end{align*}
from which \eqref{eq:elmain} follows at once. 
\end{proof}
To explain the advantages of the technically slightly intricate construction of the particular minimising sequence $(u_{k})$, let us first make the following
\begin{remark}[Euler--Lagrange for Measures]\label{rem:anzellotti}
{\normalfont It is posssible to directly work on the Euler--Lagrange equation satisfied by a generalised minimiser $u\in\gm(\mathfrak{F})$. Indeed, transferring \textsc{Anzellotti}'s work \cite{Anz} to functionals of type \eqref{eq:main}, one is able to show that 
\begin{align*}
\int_{\Omega}\langle f'(\E^{a}u),\E^{a}\varphi\rangle\dif x & + \int_{\Omega}\left\langle (f^{\infty})'\left(\frac{\dif \E^{s}u}{\dif |\E^{s}u|}\right),\frac{\dif \E^{s}\varphi}{\dif|\E^{s}\varphi|}\right\rangle\dif|\E^{s}\varphi| \\ & = \int_{\partial\Omega}\left\langle (f^{\infty})'\left(\frac{u_{0}-u}{|u_{0}-u|}\odot\nu_{\partial\Omega}\right),\varphi\odot\nu_{\partial\Omega} \right\rangle\dif\mathcal{H}^{n-1}, 
\end{align*}
for all $\varphi\in\bd(\Omega)$ with $|\E^{s}\varphi|\ll|\E^{s}u|$ such that $\varphi(x)=0$ $\mathcal{H}^{n-1}$--a.e. on $\{x\in\partial\Omega\colon u(x)=u_{0}(x)\}$, where $\nu_{\partial\Omega}$ is the outward unit normal to $\partial\Omega$. However, it seems difficult to apply the difference quotient technique directly on the Euler-Lagrange equation for measures so that we rather choose approximation procedures.}
\end{remark}
To conclude with, note that Lemma \ref{lem:eulerlagrange} enables us to work with difference quotients applied to functions and to eventually deduce uniform estimates for the single $u_{k}$'s. In particular, by arbitariness of the generalised minimiser $u\in\gm(\mathfrak{F})$ as was assumed in this section, we have constructed a sequence converging to $u$ in a suitable sense, and hence uniform estimates on the $u_{k}$'s will be inherited by $u$. Note that by starting from an arbitrary $u\in\gm(\mathfrak{F})$, regularity for all generalised minimisers will hereby be established. 
\subsection{On Projections onto $\mathcal{R}$}
In this intermediate section, we give a technical result which might be clear to the experts though hard to trace in the literature.
\begin{lemma}\label{lem:rigidchoose}
Let $U\subset\R^{n}$ be an open, bounded and connected set with Lipschitz boundary. Let $1<p<\infty$. Then for every $1\leq q \leq p$ there exists a finite constant $c_{q}>0$ such that for all $u\in\sobo^{1,p}(U;\R^{n})$ there exists $b\in\mathcal{R}(U)$ such that
\begin{align}\label{eq:Korncrucial}
& \int_{U}|u-b|^{q}\dif x \leq c_{q} \int_{U}|\sg(u)|^{q}\dif x\;\;\;\text{and}\;\;\int_{U}|\D\,(u-b)|^{p}\dif x \leq c_{p} \int_{U}|\sg(u)|^{p}\dif x.
\end{align}
\end{lemma}
The key in this lemma is that we can choose \emph{one} particular rigid deformation to validate both inequalities.
\begin{proof}
Denote $X_{p}(U)$ either $\sobo^{1,p}(U;\R^{n})$ provided $1<p<\infty$ or $\ld(U)$ provided $p=1$ and denote 
\begin{align}
\mathcal{R}_{X_{p}}^{\bot}(U):=\left\{\varphi\in X_{p}(\Omega)\colon\;\int_{U}\langle \varphi,\psi\rangle \dif x = 0\;\;\text{for all}\;\psi\in\mathcal{R}(\Omega) \right\}. 
\end{align}
Then, by straightforward adaptation of \cite[Eq.~(3.25)ff.]{FS}, we find that for every $1\leq p < \infty$ there exists $c_{p}>0$ such that $\|v\|_{\lebe^{p}(U;\R^{n})}\leq c_{p}\|\sg(v)\|_{\lebe^{p}(U;\R^{n\times n})}$ holds for all $\mathcal{R}_{X_{p}}^{\bot}(U)$. We now consider an $\lebe^{2}$-orthonormal basis $\{b_{1},...,b_{m}\}$ of the finite dimensional space $\mathcal{R}(\Omega)$. We then define the $\lebe^{2}$-orthogonal projections $\Pi\colon\lebe^{2}(\Omega;\R^{n})\to\mathcal{R}(\Omega)$ by 
\begin{align}
\Pi\varphi :=\sum_{j=1}^{m}\langle b_{j},\varphi\rangle_{\lebe^{2}(\Omega;\R^{n})}b_{j},\qquad\varphi\in\lebe^{2}(\Omega;\R^{n}).
\end{align}
Note that, since $\mathcal{R}(\Omega)$ consists of affine-linear polynomials and so $\mathcal{R}(\Omega)\subset\lebe^{\infty}(\Omega;\R^{n})$, $\Pi\varphi$ is also well-defined for $\varphi\in\lebe^{1}(\Omega;\R^{n})$. We moreover have for all $1\leq p \leq \infty$ 
\begin{align*}
\int_{\Omega}|\Pi\varphi|^{p}\dif x & \leq \Big( \sum_{j=1}^{m}\|b_{j}\|_{\lebe^{\infty}(\Omega;\R^{n})}^{2}\Big)\Big(\int_{\Omega}|\varphi|^{p}\dif x\Big)^{\frac{1}{p}},
\end{align*}
and from here we see that $\Pi$ is indeed $\lebe^{p}$-stable for all $1\leq p<\infty$. 
Let now $u\in\sobo^{1,p}(U;\R^{n})$ for $1<p<\infty$, so that, in particular, $u\in\ld(U)$. We then have $u-\Pi u\in \mathcal{R}_{X_{p}}^{\bot}(U)$ \emph{regardless of $p$} and hence deduce the first part of \eqref{eq:Korncrucial}. For the second one, recall that by Korn's inequality, $\|\D u\|_{\lebe^{p}(U;\R^{n\times n})}\leq C ( \|u\|_{\lebe^{p}(U;\R^{n})}+\|\sg(u)\|_{\lebe^{p}(U;\R^{n})})$. Replacing $u$ by $u-\Pi u$ in this inequality and invoking the first part of \eqref{eq:Korncrucial}, we establish the second part of \eqref{eq:Korncrucial} and the proof is complete. 
\end{proof}

\subsection{Proof of Theorem~\ref{thm:main0}}\label{sec:thmmain0}
We now turn to the proof of Theorem~\ref{thm:main1} and begin with the following auxiliary lemma; recall that $\tau_{s,h},\tau_{s,h}^{+},\tau_{s,h}^{-}$ are defined in Section~\ref{sec:notation}.
\begin{lemma}\label{lem:Valpha1}
Let $f\in\hold^{2}(\R_{\sym}^{n\times n})$ be a $\mu$-elliptic integrand, $1<\mu<\infty$, and define for $1<\alpha<2$ the auxiliary map $V_{\alpha}$ by \eqref{eq:vdef}. Then there exists a constant $C=C(\alpha,n)>0$ such that for all $u\in \lebe_{\locc}^{1}(\Omega;\R^{n})$, all relatively compact Lipschitz subsets $K\Subset\Omega$, all $h\in\R$ with $|h|<\dista(K,\partial\Omega)$, $s\in\{1,...,n\}$ and $\mathscr{L}^{n}$-a.e. $x\in K$ there holds 
\begin{align}
\frac{|\tau_{s,h}V_{\alpha}(u(x))|^{2}}{(1+|u(x+he_{s})|^{2}+|u(x)|^{2})^{\frac{2(1-\alpha)+\mu}{2}}}\leq C\frac{|\tau_{s,h}u(x)|^{2}}{(1+|u(x)|^{2}+|u(x+he_{s})|^{2})^{\frac{\mu}{2}}}.
\end{align}
\end{lemma} 
\begin{proof}
We now use the auxiliary estimate given by ~Lemma~ \ref{lem:valpha} to conclude for the auxiliary function $V_{\alpha}(\xi):=(1+|\xi|^{2})^{\frac{1-\alpha}{2}}\xi$ with $1<\alpha<2$ that 
\begin{align*}
\frac{|\xi-\eta|^{2}}{(1+|\eta|^{2}+|\xi|^{2})^{\frac{\mu}{2}}}& \leq C \frac{|V_{\alpha}(\xi)-V_{\alpha}(\eta)|^{2}}{(1+|\xi|^{2}+|\eta|^{2})^{\frac{\mu+2(1-\alpha)}{2}}}. 
\end{align*}
Applying this to $\xi=u(x+he_{s})$ and $\eta=u(x)$, we conclude.
\end{proof}
After these preparations, we now come to the
\begin{proof}[Proof of Theorem~\ref{thm:main0}]
Let $\mu$ be as in the theorem. We then fix an arbitrary generalised minimiser $u\in\gm(\mathfrak{F};u_{0})$ and consider the sequence $(u_{k})$ constructed in Section~\ref{sec:ekeland}, cf. \eqref{eq:intermed00}. Let $k\in\mathbb{N}$ be arbitrary but fixed; then $u_{k}$ satisfies the approximate Euler-Lagrange equation~\eqref{eq:elmain}. Let $x_{0}\in\Omega$, $0<r<R<\dista(x_{0},\partial\Omega)$ and pick $\rho\in \hold_{c}^{1}(\ball(x_{0},R);[0,1])$ with $\mathbbm{1}_{\ball(x_{0},r)}\leq \rho \leq \mathbbm{1}_{\ball(x_{0},R)}$. Then, let $\Omega_{1}$ be the connected component of $\Omega$ that contains $x_{0}$; we may assume that $\Omega_{1}$ itself has Lipschitz boundary.  Due to Lemma~\ref{lem:rigidchoose} and the fact that $u_{k}|_{\Omega_{1}}\in\sobo^{1,2}(\Omega_{1};\R^{n})$, we first choose a rigid deformation $b_{k}\in\mathcal{R}(\Omega)$ such that with $c=c(\Omega_{1})$
\begin{align}\label{eq:rigidchoose}
\int_{\Omega_{1}}|u_{k}-b_{k}|\dif x \leq c\int_{\Omega_{1}}|\sg(u_{k})|\dif x\;\;\text{and}\;\;\int_{\Omega_{1}}|\D\,(u_{k}-b_{k})|^{2}\dif x \leq c\int_{\Omega_{1}}|\sg(u_{k})|^{2}\dif x.
\end{align}
For $|h|<\dista(\partial\!\ball(x_{0},R);\partial\Omega_{1})$ and $s\in\{1,...,n\}$, we then choose $\varphi:=\tau_{s,h}^{-}\big(\rho^{2}\tau_{s,h}^{+}(u_{k}-b_{k})\big)\in\sobo^{1,2}(\Omega;\R^{n})(\hookrightarrow (\sobo_{0}^{1,\infty}(\Omega;\R^{n})^{*})$ as a test map in \eqref{eq:elmain}. Since $\sg$ and $\tau_{s,h}^{-}$ commute, this yields with $\widetilde{u}_{k}:=u_{k}-b_{k}$
\begin{align}\label{eq:refinedEL}
\left\vert\int_{\Omega}\langle f'_{k}(\sg(u_{k})),\tau_{s,h}^{-}(\sg(\rho^{2}\tau_{s,h}^{+}\widetilde{u}_{k}))\rangle\dif x \right\vert \leq \frac{1}{k}\|\tau_{s,h}^{-}(\rho^{2}\tau_{s,h}^{+}\widetilde{u}_{k})\|_{(\sobo_{0}^{1,\infty}(\Omega;\R^{n})^{*}}.
\end{align}
We now proceed in three steps. 

\emph{Step 1. Recast of \eqref{eq:refinedEL}.} By discrete integration by parts in \eqref{eq:refinedEL}, we find by the product rule for $\sg$ that
\begin{align}\label{eq:intermed1}
\begin{split}
\mathbf{I} & :=\int_{\Omega}\langle\tau_{s,h}^{+}f'_{k}(\sg(u_{k})),\rho^{2}\tau_{s,h}\sg(u_{k})\rangle \dif x   \\ & \leq  \left\vert \int_{\Omega}\langle \tau_{s,h}^{+}f'_{k}(\sg(u_{k})),2\rho\D\rho\odot\tau_{s,h}\widetilde{u}_{k}\rangle\dif x \right\vert +\frac{1}{k}\|\tau_{s,h}^{-}(\rho^{2}\tau_{s,h}\widetilde{u}_{k})\|_{(\sobo_{0}^{1,\infty})^{*}} \\ & \leq \left\vert \int_{\Omega}\langle \tau_{s,h}^{+}f'(\sg(u_{k})),2\rho\D\rho\odot\tau_{s,h}\widetilde{u}_{k}\rangle\dif x \right\vert + \frac{1}{A_{k}k^{2}}\left\vert\int_{\Omega}\langle\tau_{s,h}\sg(u_{k}),2\rho\D\rho\odot\tau_{s,h}\widetilde{u}_{k}\rangle\dif x\right\vert \\
& + \frac{1}{k}\|\tau_{s,h}^{-}(\rho^{2}\tau_{s,h}\widetilde{u}_{k})\|_{(\sobo_{0}^{1,\infty})^{*}}\\
& =: \mathbf{II}+\mathbf{III}+\mathbf{IV}, 
\end{split}
\end{align}
where $A_{k}$ is defined by \eqref{eq:defAk}. 

\emph{Step 2. Key Estimates.} We now estimate the single terms $\mathbf{I},...,\mathbf{IV}$. As to $\mathbf{I}$, we introduce for each $k\in\mathbb{N}$ the bilinear forms 
\begin{align*}
\mathcal{B}_{k,h}(x)(\xi,\zeta):=\int_{0}^{1}\langle f''_{k}\big(\sg(u_{k})+t\tau_{s,h}\sg(u_{k})\big)\xi,\zeta\rangle\dif t,\qquad \xi,\zeta\in\R_{\sym}^{n\times n}. 
\end{align*}
Consequently, by the fundamental theorem of calculus we deduce 
\begin{align*}
\mathbf{I} & = \int_{\Omega}\langle f'_{k}(\sg(u_{k}(x+he_{s})))-f'_{k}(\sg(u_{k}(x))),\rho^{2}\sg(u_{k})\rangle\dif x\\
& = \int_{\Omega}\left\langle \int_{0}^{1}\frac{\dif}{\dif t}f'_{k}\big(\sg(u_{k})+t\tau_{s,h}\sg(u_{k})\big)\dif t,\rho^{2}\tau_{s,h}\sg(u_{k})\right\rangle\dif x \\
& = \int_{\Omega}\mathcal{B}_{k,h}\big(\rho\tau_{s,h}\sg(u_{k}),\rho\tau_{s,h}\sg(u_{k})\big) \dif x.
\end{align*}
By $\mu$-ellipticity and the definition of $f_{k}$, these are strongly elliptic bilinear forms. We briefly pause to comment on the strategy. Usually, one would now apply the Cauchy-Schwarz inequality to the first term of the right hand side of the first inequality in \eqref{eq:intermed1} and then conveniently absorbs, but this we \emph{do not pursue here}. In fact, this would give rise to the term
\begin{align}\label{eq:critical}
4\int_{\Omega}\mathcal{B}_{k,h}(\rho\D\rho\odot\tau_{s,h}u_{k},\rho\D\rho\odot\tau_{s,h}u_{k})\dif x 
\end{align}
which is unclear to us how to be controlled by the estimates available so far. Instead, we use the fact that by Lipschitz continuity of $f$, $f'$ is bounded and so, in particular $|\tau_{s,h}^{+}f'(\sg(u_{k}))|\leq M$ for some $M>0$. We now go back to the local embedding provided by Proposition~\ref{prop:cheapembedding} and hence obtain for every $0<\beta<1$ the embedding $\ld(\Omega)\hookrightarrow \sobo_{\locc}^{\beta,1}(\Omega;\R^{n})$. By standard means, this yields $\ld(\Omega_{1})\hookrightarrow \sobo^{\beta,1}(\Omega_{1};\R^{n})\hookrightarrow (\besov_{1,\infty}^{\beta})_{\locc}(\Omega_{1};\R^{n})$. In particular, we obtain by Remark~\ref{rem:Ws1emb} in conjunction with Lemma~\ref{lem:rigidchoose}
\begin{align}\label{eq:besovrigid}
\begin{split}
\sup_{\substack{s\in\{1,...,n\}\\ |h|<\dista(\ball(x_{0},R),\partial\Omega_{1})}} & \int_{\ball(x_{0},R)}\frac{|\tau_{s,h}\widetilde{u}_{k}|}{h^{\beta}}\dif x  \leq c(\Omega_{1})\iint_{\Omega_{1}\times \Omega_{1}}\frac{|\widetilde{u}_{k}(x)-\widetilde{u}_{k}(y)|}{|x-y|^{n+\beta}}\dif x\dif y\\
& \leq c(\Omega_{1},\beta)\|\widetilde{u}_{k}\|_{\ld(\Omega_{1})} \\ & =  c(\Omega_{1},\beta)(\|\widetilde{u}_{k}\|_{\lebe^{1}(\Omega_{1};\R^{n})}+\|\sg(u_{k})\|_{\lebe^{1}(\Omega_{1};\R^{n\times n})})\\
& \leq C(\Omega_{1},\beta)\|\sg(u_{k})\|_{\lebe^{1}(\Omega_{1};\R^{n\times n})}\qquad (\text{by}\;\eqref{eq:rigidchoose}).
\end{split}
\end{align}
In consequence, for every $0<\beta<1$ (to be fixed later) we find $C(\beta)>0$ such that 
\begin{align*}
\mathbf{II} & \leq M \int_{\ball(x_{0},R)}|\tau_{s,h}\widetilde{u}_{k}|\dif x = M h^{\beta}\int_{\ball(x_{0},R)}\frac{|\tau_{s,h}\widetilde{u}_{k}|}{h^{\beta}}\dif x \\ 
& \leq C(\beta)Mh^{\beta}\|\sg(u_{k})\|_{\lebe^{1}(\Omega;\R^{n\times n})} 
\end{align*}
so that by \eqref{eq:LDbound} and possibly enlarging $C(\beta)$, we end up with
\begin{align}\label{eq:estimateII}
\mathbf{II} \leq C(\beta)Mh^{\beta}\|\sg(\widetilde{u}_{k})\|_{\lebe^{1}(\Omega;\R^{n\times n})}.
\end{align}
As to $\mathbf{III}$, we apply the Cauchy-Schwarz inequality to find for $\delta>0$ sufficiently small
\begin{align*}
\mathbf{III} & \leq \frac{\delta}{A_{k}k^{2}}\int_{\Omega}|\rho\tau_{s,h}\sg(u_{k})|^{2}\dif x +\frac{C(\delta)}{A_{k}k^{2}}\int_{\Omega}|\nabla\rho\odot\tau_{s,h}\widetilde{u}_{k}|^{2}\dif x \\
& = \frac{\delta}{A_{k}k^{2}}\int_{\Omega}|\rho\tau_{s,h}\sg(u_{k})|^{2}\dif x +\frac{C(\delta,\rho)h^{2}}{A_{k}k^{2}}\int_{\ball(x_{0},R)}|\Delta_{s,h}\widetilde{u}_{k}|^{2}\dif x\\
& \leq \frac{\delta}{A_{k}k^{2}}\int_{\Omega}|\rho\tau_{s,h}\sg(u_{k})|^{2}\dif x +\frac{C(\delta,\rho)h^{2}}{A_{k}k^{2}}\int_{\Omega_{1}}|\partial_{s}\widetilde{u}_{k}|^{2}\dif x. 
\end{align*}
The ultimate term now is controlled by Korn's inequality. To be precise, we have by \eqref{eq:rigidchoose}
\begin{align*}
\int_{\ball(x_{0},R)}|\partial_{s}\widetilde{u}_{k}|^{2}\dif x & \leq \int_{\Omega_{1}}|\nabla \widetilde{u}_{k}|^{2}\dif x \leq c(\Omega_{1}) \int_{\Omega_{1}}|\sg(u_{k})|^{2}\dif x  \leq c(\Omega_{1}) \int_{\Omega}|\sg(u_{k})|^{2}\dif x 
\end{align*}
and so, by the definition of $A_{k}$,
\begin{align}\label{eq:estimateIII}
\mathbf{III} \leq \frac{\delta}{A_{k}k^{2}}\int_{\Omega}|\rho\tau_{s,h}\sg(u_{k})|^{2}\dif x + \frac{C(\delta,\rho,\Omega)h^{2}}{k^{2}} =:\mathbf{III}_{1}^{(\delta)}+\mathbf{III}_{2}^{(\delta)}.
\end{align}
Ad $\mathbf{IV}$. Here we estimate for $\gamma>0$ to be specified later 
\begin{align}\label{eq:estimateIV}
\begin{split}
\mathbf{IV} & = \frac{1}{k}\|\tau_{s,h}^{-}(\rho^{2}\tau_{s,h}\widetilde{u}_{k})\|_{(\sobo_{0}^{1,\infty})^{*}} = \frac{h^{1+\gamma}}{k}\|\Delta_{s,h}^{-}(\rho^{2}\frac{\tau_{s,h}}{h^{\gamma}}\widetilde{u}_{k})\|_{(\sobo_{0}^{1,\infty})^{*}} \\
& \leq \frac{h^{1+\gamma}}{k}\|\rho^{2}\frac{\tau_{s,h}}{h^{\gamma}}\widetilde{u}_{k}\|_{\lebe^{1}}\qquad(\text{by Lemma~\ref{lem:negsob}})\\
& \leq c\frac{h^{1+\gamma}}{k}\|\widetilde{u}_{k}\|_{\sobo^{\gamma,1}(\Omega_{1};\R^{n})} \leq c\frac{h^{1+\gamma}}{k}\|\widetilde{u}_{k}\|_{\ld(\Omega_{1})} \stackrel{\eqref{eq:LDbound}}{\leq} c\frac{h^{1+\gamma}}{k}, 
\end{split}
\end{align}
where we have employed a similar argument as in \eqref{eq:besovrigid}. 

In an intermediate step, let $0\leq t\leq 1$ and $a,b\in\R^{n\times n}$ be arbitrary. There holds (with some fixed $C>0$ independent of $t,a$ and $b$)
\begin{align}\label{eq:basicinequality}
(1+|a+tb|^{2})^{\frac{1}{2}}\leq C(1+|a|^{2}+|b|^{2})^{\frac{1}{2}}.
\end{align}
Now, we estimate from below by virtue of $\mu$-ellipticity of $f$ and the definition of $f_{k}$ 
\begin{align*}
\mathbf{I} & \geq \int_{\Omega}\int_{0}^{1}\langle f''\big(\sg(u_{k})+t\tau_{s,h}\sg(u_{k})\big)\rho\tau_{s,h}\sg(u_{k}),\rho\tau_{s,h}\sg(u_{k})\rangle\dif t\dif x + \frac{1}{A_{k}k^{2}}\int_{\Omega}|\rho\tau_{s,h}\sg(u_{k})|^{2}\dif x\\
& \geq \lambda \int_{\Omega}\int_{0}^{1}\frac{|\rho\tau_{s,h}\sg(u_{k})|^{2}}{(1+|\sg(u_{k})+t\tau_{s,h}\sg(u_{k})|^{2})^{\frac{\mu}{2}}}\dif t\dif x+ \frac{1}{A_{k}k^{2}}\int_{\Omega}|\rho\tau_{s,h}\sg(u_{k})|^{2}\dif x\\
& \stackrel{\eqref{eq:basicinequality}}{\geq} \widetilde{\lambda}\int_{\Omega}\frac{|\rho\tau_{s,h}\sg(u_{k})|^{2}}{(1+|\sg(u_{k})(x)|^{2}+|\sg(u_{k})(x+he_{s})|^{2})^{\frac{\mu}{2}}}\dif x + \frac{1}{A_{k}k^{2}}\int_{\Omega}|\rho\tau_{s,h}\sg(u_{k})|^{2}\dif x =: \mathbf{I}'
\end{align*}
after dimishing $\lambda>0$ to $\widetilde{\lambda}>0$ if necessary. Now let $1<\alpha<2$ to be fixed later. Then Lemma~\ref{lem:Valpha1} to $\sg(u_{k})$ yields 
\begin{align}\label{eq:estimateI'}
\begin{split}
\mathbf{I} \geq \mathbf{I}' & \geq c(\widetilde{\lambda})\int_{\Omega}\rho\frac{|\tau_{s,h}V_{\alpha}(\sg(u_{k}(x)))|^{2}}{(1+|\sg(u_{k})(x+he_{s})|^{2}+|\sg(u_{k})(x)|^{2})^{\frac{2(1-\alpha)+\mu}{2}}}\dif x \\ & + \frac{1}{A_{k}k^{2}}\int_{\Omega}|\rho\tau_{s,h}\sg(u_{k})|^{2}\dif x =: \mathbf{I}'_{1}+\mathbf{I}'_{2}.
\end{split}
\end{align}
We now gather the estimates given so far and put 
\begin{align}\label{eq:omegadef}
\omega_{k,h,s}(x) := \frac{1}{(1+|\sg(u_{k}(x))|^{2}+|\sg(u_{k}(x+he_{s}))|^{2})^{\frac{\mu+2(1-\alpha)}{2}}}\qquad\text{for}\,\mathscr{L}^{n}\text{-a.e.}\,x\in\ball(x_{0},R).
\end{align}
for brievity. We choose $\delta>0$ in \eqref{eq:estimateIII} such that $\delta<1$. In consequence, we may absorb $\mathbf{III}_{1}^{(\delta)}$ into $\mathbf{I}'_{2}$ in the overall inequality. Hence, by \eqref{eq:intermed1}, \eqref{eq:estimateI'}, \eqref{eq:estimateII}, \eqref{eq:estimateIII} and \eqref{eq:estimateIV} we invoke \eqref{eq:LDbound} to end up with 
\begin{align}
\begin{split}
c(\widetilde{\lambda})  \int_{\Omega}|\rho\tau_{s,h}V_{\alpha}(\sg(u_{k}(x)))|^{2}\omega_{k,h,s}(x)\dif x & + \frac{1-\delta}{A_{k}k^{2}}\int_{\Omega}|\rho\tau_{s,h}\sg(u_{k})|^{2}\dif x \\
& \leq C(\beta)Mh^{\beta}\\ &  + \frac{C(\delta,\rho,\Omega)h^{2}}{k^{2}} + C\frac{h^{1+\gamma}}{k}.
\end{split}
\end{align}
Since we may assume without loss of generality that $|h|<1$ and by positivity of the second term on the left hand side of the previous inequality, we find by dividing the previous inequality by $h^{\beta}$ 
\begin{align}\label{eq:keyest1}
\sup_{k\in\mathbb{N}}\int_{\Omega}\left\vert\frac{\tau_{s,h}V_{\alpha}(\sg(u_{k}(x)))}{h^{\frac{\beta}{2}}} \right\vert^{2}\omega_{k,h,s}(x)\dif x < \infty.
\end{align}

\emph{Step 3. Conclusion.} We go back to \eqref{eq:keyest1} and deduce by Young's inequality that 
\begin{align*}
\int_{\Omega}\rho\left\vert\frac{\tau_{s,h}V_{\alpha}(\sg(u_{k}(x)))}{h^{\frac{\beta}{2}}} \right\vert\dif x & =
\int_{\Omega}\rho\left\vert\frac{\tau_{s,h}V_{\alpha}(\sg(u_{k}(x)))}{h^{\frac{\beta}{2}}} \right\vert\omega_{k,h,s}^{\frac{1}{2}}\frac{\dif x}{\omega_{k,h,s}^{\frac{1}{2}}}\\
& \leq \frac{1}{2}\int_{\Omega}\rho^{2}\left\vert\frac{\tau_{s,h}V_{\alpha}(\sg(u_{k}(x)))}{h^{\frac{\beta}{2}}} \right\vert^{2}\omega_{k,h,s}\dif x + \frac{1}{2}\int_{\ball(x_{0},R)}\frac{\dif x}{\omega_{k,h,s}(x)}\\
& =: \mathbf{V}+\mathbf{VI}.
\end{align*}
The term $\mathbf{V}$ is bounded by \eqref{eq:keyest1}. As to $\mathbf{VI}$, we recall that by \eqref{eq:LDbound}, $(u_{k})$ is uniformly bounded in $\ld(\Omega)$. In consequence, $\mathbf{IV}$ is uniformly bounded in $k$ and $h$ provided 
\begin{align}
\mu+2(1-\alpha)\leq 1,\;\;\;\text{that is,}\;\;\;\frac{\mu+1}{2}\leq\alpha.
\end{align}
At this stage, let us recall that this appears subject to the condition $1<\alpha<2$ from Lemma~\ref{lem:Valpha1}. Since $\mu>1$, the lower bound is satisfied in any case, but the upper bound requires $\mu<3$ which is satisfied as well for the growth regime we are considering. Now, summarising, we obtain 
\begin{align}
\sup_{k\in\mathbb{N}}\int_{\Omega}\rho\left\vert\frac{\tau_{s,h}V_{\alpha}(\sg(u_{k}(x)))}{h^{\frac{\beta}{2}}} \right\vert\dif x<\infty
\end{align}
and hence, by arbitrariness of $x_{0}$ and $\rho$, infer that $(V_{\alpha}(\sg(u_{k})))$ is locally uniformly bounded in $\besov_{1,\infty}^{\beta/2}$. By Lemma~\ref{lem:besovembedding}, we obtain that for any $0<\delta<\frac{n}{n-\beta/2}$, $(V_{\alpha}(\sg(u_{k})))$ is locally uniformly bounded in $\lebe^{\frac{2n}{2n-\beta}-\delta}$. Now we invoke Lemma~\ref{lem:valpha}, cf.~\eqref{eq:Valphaintbound}, to deduce that for any relatively compact Lipschitz set $K\Subset\Omega$ there exists 
\begin{align}\label{eq:mainbound2}
\sup_{k\in\mathbb{N}}\int_{K}|\sg(u_{k})|^{(2-\alpha)\big(\frac{2n}{2n-\beta}-\delta\big)}\dif x = C(\alpha,\delta,\beta) < \infty, 
\end{align}
and we now choose $\beta,\alpha$ and $\delta$ in a suitable way. First we note that 
\begin{align}\label{eq:mucond1}
\mu < 1+\frac{1}{n} \Longrightarrow \mu+1 < 2+\frac{1}{n}\Longrightarrow \frac{\mu+1}{2}< 1+ \frac{1}{2n}
\end{align}
and so we find and fix $\alpha$ such that 
\begin{align}\label{eq:alphacond1}
\frac{\mu+1}{2}< \alpha < 1+ \frac{1}{2n}.
\end{align}
From the previous inequality, we deduce 
\begin{align*}
\alpha < 1+\frac{1}{2n} \Longrightarrow 1-\alpha >-\frac{1}{2n} \Longrightarrow 2-\alpha > 1-\frac{1}{2n}=\frac{2n-1}{2n}\Longrightarrow (2-\alpha)\frac{2n}{2n-1}>1.
\end{align*}
Now we may send $\beta\nearrow 1$ and $\delta\searrow 0$ to deduce that there exists $\beta<1$ and $\delta>0$ such that 
\begin{align}\label{eq:pchoice1}
p:=(2-\alpha)\Big(\frac{2n}{2n-\beta}-\delta\Big)>1
\end{align}
as well. Then we infer from \eqref{eq:mainbound2} that for every ball $\ball\Subset\Omega$ there holds
\begin{align*}
\sup_{k\in\mathbb{N}}\int_{\ball}|\sg(u_{k})|^{p}\dif x <\infty. 
\end{align*}
Then, by Poincar\'{e}'s inequality, we find rigid deformations $d_{k}\in\mathcal{R}(\ball)$ such that the sequence $(u_{k}-d_{k})|_{\ball}$ is uniformly bounded in $\lebe^{p}(\ball;\R^{n})$. Since $1<p<\infty$, Korn's inequality and reflexivity of $\sobo^{1,p}$ allow to extract a subsequence $(u_{k(l)}|_{K})$ which converges weakly to some $v\in\sobo^{1,p}(\ball;\R^{n})$. Since $u_{k}|_{\ball}\stackrel{*}{\rightharpoonup}u|_{\ball}$, we conclude that $v=u|_{\ball}$ and $\E^{s}u$ must vanish on $\ball$. The proof is complete. 
\end{proof}
Let us now comment on some aspects of the proof. 

\begin{itemize}
\item[\emph{1.}] Even though briefly mentioned in the proof, let us stress again that it is precisely the term \eqref{eq:critical} where the proof mostly differs from the $\bv$-case. If we worked with $\mu$-elliptic functionals \eqref{eq:varprin} where $\sg$ is replaced by $\D$ and we thus are in the $\bv$-framework, the suitable adaptation of the approximation procedure outlined in Section~\ref{sec:ekeland} (cf.~\cite[Sec.~5]{BS1}) yields that the constructed sequence $(u_{k})$ is uniformly bounded in $\bv(\Omega;\R^{n})$. Then, by the upper bound provided by the $\mu$-ellipticity (cf.~\eqref{eq:muell}) the term from \eqref{eq:critical} would be controlled by 
\begin{align*}
\int_{\Omega}\mathcal{B}_{k,h}(\rho\D\rho\odot\tau_{s,h}u_{k},\rho\D\rho\odot\tau_{s,h}u_{k})\dif x \leq h^{2}\int_{\Omega}\frac{|\rho\Delta_{s,h}u_{k}|^{2}}{(1+|\sg(u_{k})|^{2})^{\frac{1}{2}}}\dif x, 
\end{align*}
and as in \cite[Lem.~5.3]{BS1}, the last term can be controlled as $(u_{k})$ then would be bounded in $\sobo^{1,1}(\Omega;\R^{n})$. Hence it is at this stage, where the $\bd$-case differs from the $\bv$-case.
\item[\emph{2.}] In bounding the term $\mathbf{IV}$ of the previous proof (cf.~\eqref{eq:estimateIV}), it is not necessary to work with $\gamma>0$; in fact, for the purposes of this proof, $\gamma=0$ would do. In this case, we could estimate $\|\Delta_{s,h}(\rho^{2}\tau_{s,h}u_{k})\|_{(\sobo_{0}^{1,\infty})^{*}}\leq \|\rho^{2}\tau_{s,h}u_{k}\|_{\lebe^{1}}\leq  C \|u_{k}\|_{\lebe^{1}} \leq C$ by \eqref{eq:LDbound}. However, for the iteration to be sketched below, it is necessary to work with $\gamma>0$ so that we included the argument already in the above proof. 
\end{itemize}
As we mentioned above, Theorem~\ref{thm:main0} implies that $\gm(\mathfrak{F};u_{0})\subset\sobo_{\locc}^{1,p}(\Omega;\R^{n})\cap\ld(\Omega)$ and so we may now use the additional integrability information to amplify the regularity of generalised minimisers. Here, as $V_{\alpha}$ has a regularising effect on $\sg(u_{k})$, we directly work on $\sg(u_{k})$. Let $1<\mu<1+\frac{1}{n}$. Going back to \eqref{eq:pchoice1} subject to \eqref{eq:alphacond1}, optimising $p$ yields that $\gm(\mathfrak{F};u_{0})\subset \ld(\Omega)\cap\sobo_{\locc}^{1,q}(\Omega;\R^{n})$ for all 
\begin{align}\label{eq:qdef}
1\leq q <(2-\mu)\frac{2n}{2n-1}.
\end{align}
Exemplarily, we show how for a certain range of ellipticities we can even obtain (almost) second derivative estimates.
\begin{corollary}\label{cor:nikolskii}
Suppose that $f\in\hold^{2}(\R_{\sym}^{n\times n})$ is a $\mu$-elliptic integrand of linear growth with $1<\mu<\frac{4n}{4n-1}$. Then for all generalised minimisers $u\in\gm(\mathfrak{F};u_{0})$ we have $u\in\sobo^{2,Q}_{\locc}(\Omega;\R_{\sym}^{n})$ for some $Q=Q(\mu)>1$. 
\end{corollary}
\begin{proof}
Let us firstly note that the condition on $\mu$ implies 
\begin{align}\label{eq:nikolimplication}
\mu < \frac{4n}{4n-1}\Longrightarrow \mu\left(1+\frac{2n}{2n-1}\right)<\frac{4n}{2n-1}\Longrightarrow \mu < (2-\mu)\frac{2n}{2n-1}, 
\end{align}
and so we deduce that $u\in\ld(\Omega)\cap\sobo_{\locc}^{1,\mu}(\Omega;\R^{n})$ by the above argument. Denote $(u_{k})$ the Ekeland approximation sequence as above. We now go back to the proof of Theorem~\ref{thm:main0}, step 1, let $x_{0}\in\Omega$ be arbitrary and choose $0<r<R<\dista(\partial\Omega;\ball(x_{0},R)$ together with a localisation function $\rho\in\hold_{c}^{2}(\ball(x_{0},R);[0,1])$. We then put, for $s\in\{1,...,n\}$ and $|h|$ sufficiently small, $\varphi:=-\Delta_{s,h}^{-}(\rho^{2}\Delta_{s,h}\widetilde{u}_{k})$. Then we insert $\varphi$ into \eqref{eq:refinedEL}, write 
\begin{align*}
\sg(\varphi)=-\Delta_{s,h}^{-}(\rho^{2}\Delta_{s,h}\sg(u_{k}))-\Delta_{s,h}^{-}(2\rho\D\rho\odot\Delta_{s,h}\widetilde{u}_{k}) 
\end{align*}
and thereby end up with 
\begin{align}\label{eq:jayjay}
\begin{split}
\mathbf{J}_{1}:=\left\vert \int_{\Omega}\langle \Delta_{s,h}f'_{k}(\sg(u_{k})),\rho^{2}\Delta_{s,h}\sg(u_{k})\rangle\dif x \right\vert & \leq \left\vert\int_{\Omega}\langle f'_{k}(\sg(u_{k})),\Delta_{s,h}^{-}(2\rho\D\rho\odot\Delta_{s,h}\widetilde{u}_{k})\rangle\dif x \right\vert \\ &  + \frac{1}{k}\|\rho^{2}\Delta_{s,h}\widetilde{u}_{k}\|_{\lebe^{1}(\Omega;\R^{n})} =: \mathbf{J}_{2}+\mathbf{J}_{3}.
\end{split}
\end{align}
Similarly as in the proof of Theorem~\ref{thm:main0}, we find that for some $c>0$
\begin{align}\label{eq:J1estimate}
c\int_{\Omega}\frac{|\rho\Delta_{s,h}\sg(u_{k})|^{2}}{(1+|\sg(u_{k})(x+he_{s})|^{2}+|\sg(u_{k})|^{2})^{\frac{\mu}{2}}}\dif x + \frac{1}{k^{2}A_{k}}\int_{\Omega}|\rho\Delta_{s,h}\sg(u_{k})|^{2}\dif x \leq \mathbf{J}_{1}. 
\end{align}
Now, as to $\mathbf{J}_{2}$, we split and estimate by Lipschitz continuity of $f$ in the first inequality and $g_{k}:=f_{k}-f$, 
\begin{align*}
\mathbf{J}_{2} & \leq C\int_{\Omega}|\Delta_{s,h}^{-}(\rho\D\rho\odot\Delta_{s,h}\widetilde{u}_{k})|\dif x +\left\vert \int_{\Omega}\langle \Delta_{s,h}g'_{k}(\sg(u_{k})),2\rho\D\rho\odot\Delta_{s,h}\widetilde{u}_{k}\rangle\dif x\right\vert =:\mathbf{J}_{2}^{(1)}+\mathbf{J}_{2}^{(2)}.
\end{align*}
Ad $\mathbf{J}_{2}^{(1)}$. Let $1<\widetilde{q}<2$ to be fixed later. We find by estimating difference quotients against differentials and Young's inequality
\begin{align}
\begin{split}
\mathbf{J}_{2}^{(1)} & \leq C\int_{\Omega}|\partial_{s}(\D\rho\odot\rho\Delta_{s,h}\widetilde{u}_{k})|\dif x \\ & \leq C(\rho) + C(\rho)\int_{\Omega}|\partial_{s}(\rho\Delta_{s,h}\widetilde{u}_{k})|\dif x \\
& \leq C(\rho) + C(\rho)\mathscr{L}^{n}(\Omega)^{\widetilde{q}'}+C(\rho)\int_{\Omega}|\partial_{s}(\rho\Delta_{s,h}\widetilde{u}_{k})|^{\widetilde{q}}\dif x \\
& \leq C(\rho) + C(\rho)\mathscr{L}^{n}(\Omega)^{\widetilde{q}'}+C(\rho,\widetilde{q})\int_{\Omega}|\sg(\rho\Delta_{s,h}\widetilde{u}_{k})|^{\widetilde{q}}\dif x\;\;\;\;\;\;\text{(by Korn)}\\
& \leq C(\rho,\Omega,\widetilde{n}) + C(\rho,\widetilde{q})\int_{\ball(x_{0},R)}|\Delta_{s,h}\widetilde{u}_{k}|^{\widetilde{q}}\dif x+C(\rho,\widetilde{q})\int_{\Omega}|\rho\Delta_{s,h}\sg(u_{k})|^{\widetilde{q}}\dif x \\
& =: \mathbf{K}_{1}+\mathbf{K}_{2}+\mathbf{K}_{3}. 
\end{split}
\end{align}
If we choose $\widetilde{q}$ sufficiently close to $1$, then we are in position to utilise the fact that $u_{k}\in \sobo_{\locc}^{1,q}$ uniformly in $k$ with $q$ provided by \eqref{eq:qdef} and hence can assume without loss of generality that $\mathbf{K}_{1}$ and $\mathbf{K}_{2}$ are uniformly bounded with respect to $k$. At this stage we fix $\widetilde{q}$ as follows.  Because of $
\frac{\widetilde{q}}{2-\widetilde{q}}\searrow 1$ as $\widetilde{q}\searrow 1$ and \eqref{eq:nikolimplication}, we find $\widetilde{q}>1$ such that for some $q$
\begin{align}\label{eq:agoodchoice}
\mu \leq \mu\frac{\widetilde{q}}{2-\widetilde{q}}<q<(2-\mu)\frac{2n}{2n-1}.
\end{align}
Then, by \eqref{eq:qdef}, we have local uniform boundedness of $(\sg(v_{k}))$ in $\lebe^{q}$ and thus in $\lebe^{\mu(2-\widetilde{q})/\widetilde{q}}$. Since $2/(2-\widetilde{q})$ is the H\"{o}lder conjugate of $\frac{2}{\widetilde{q}}$, we then find by Young's inequality for $\delta>0$
\begin{align}\label{eq:K3estimate}
\begin{split}
\mathbf{K}_{3} & \leq \delta C(\rho,\widetilde{q})\int_{\Omega}\frac{|\rho\Delta_{s,h}\sg(u_{k})|^{2}}{(1+|\sg(u_{k}(x+he_{s})|^{2}+|\sg(u_{k})|^{2})^{\frac{\mu}{2}}}\dif x \\
& + C(\delta,\rho,\widetilde{q})\int_{\ball(x_{0},R)}(1+|\sg(u_{k}(x+he_{s})|^{2}+|\sg(u_{k})|^{2})^{\frac{\mu}{2}\frac{\widetilde{q}}{2}\frac{2}{2-\widetilde{q}}}\dif x. 
\end{split}
\end{align}
Now observe that by \eqref{eq:agoodchoice} and the remark afterwards, the ultimate term can be bounded independently of $k$. Then we choose $0<\delta<c$, where $c>0$ now is the constant on the left hand side of \eqref{eq:J1estimate} and absorb it into the left hand side of \eqref{eq:J1estimate}. 

Ad $\mathbf{J}_{2}^{(2)}$. By definition of $g_{k}$, we consequently obtain by Young's inequality
\begin{align*}
\mathbf{J}_{2}^{(2)} & \leq \frac{1}{k^{2}A_{k}}\int_{\Omega}|\langle\Delta_{s,h}\sg(u_{k}),2\rho\Delta\rho\odot\Delta_{s,h}\widetilde{u}_{k}\rangle|\dif x  \\ & \leq \frac{1}{2k^{2}A_{k}}\int_{\Omega}|\rho\Delta_{s,h}\sg(u_{k})|^{2}\dif x + \frac{2C(\rho)}{k^{2} A_{k}}\int_{\ball(x_{0},R)}|\Delta_{s,h}\widetilde{u}_{k}|^{2}\dif x. 
\end{align*}
We then absorb the first term on the very right hand side into the left hand side of \eqref{eq:J1estimate}. Moreover, by Korn's inequality, we see similarly as in the proof of Theorem~\ref{thm:main0} that the second term on the right hand side of the previous inequality is bounded uniformly in $k$. Finally, for $\mathbf{J}_{3}$ we recall the fact that because $(u_{k})$ is locally uniformly bounded in $\sobo^{1,q}$ with $q>1$, it is locally uniformly bounded in $\sobo^{1,1}$. This yields uniform boundedness of $\mathbf{J}_{3}$. Summarising the estimates obtained so far, we come up with 
\begin{align}\label{eq:secondderivativesuniform}
\sup_{k\in\mathbb{N}}\int_{\Omega}\frac{|\rho\Delta_{s,h}\sg(u_{k})|^{2}}{(1+|\sg(u_{k})(x+he_{s})|^{2}+|\sg(u_{k})|^{2})^{\frac{\mu}{2}}}\dif x <\infty.
\end{align}
Eventually, repeating the argument that lead to \eqref{eq:K3estimate}, we easily find that for some $Q>1$ there holds $\sup_{k\in\mathbb{N}}\|\Delta_{s,h}\sg(u_{k})\|_{\lebe^{Q}(\ball(x_{0},r);\R^{n\times n})}<\infty$. From here and the arbitrariness of $x_{0}$ and $r$ we deduce by standard means that $\partial_{s}\sg(u)\in\lebe_{\locc}^{Q}(\Omega;\R^{n\times n})$. Then Korn's inequality yields that $\partial_{s}u\in\sobo_{\locc}^{1,Q}(\Omega;\R^{n})$ and so, by arbitrariness of $s\in\{1,...,n\}$, $u\in\sobo_{\locc}^{2,Q}(\Omega;\R^{n})$. The proof is complete. 
\end{proof}
A standard application of the measure density lemma \cite[Prop.~2.7]{Giusti} then yields the following bound on the set of non-Lebesgue points of $\sg(u)$ as will be needed in a forthcoming study \cite{GmPR}:
\begin{corollary}\label{cor:HausdorffCor1}
Let $1<\mu<\frac{4n}{4n-1}$ and let $f\in\hold^{2}(\R_{\sym}^{n\times n})$ be a $\mu$-elliptic integrand of linear growth. Then for any $u\in\gm(\mathfrak{F};u_{0})$ there holds $\dim_{\mathcal{H}}(\Sigma_{u})<n-1$, where 
\begin{align}\label{eq:singset}
\Sigma_{u}:=\left\{x_{0}\in\Omega\colon\;\limsup_{r\searrow 0}\dashint_{\ball(x_{0},r)}|\sg(u)-z|\dif\mathscr{L}^{n}>0\;\text{for all}\;z\in\R_{\sym}^{n\times n}\right\}.
\end{align}
\end{corollary}

\subsection{Proof of Theorem~\ref{thm:main0} and convex duality for the extended range of $\mu$}\label{sec:convdual}
We now extend the range of $\mu$ as provided by Theorem~\ref{thm:main0}. Toward this aim, we employ the dual solution in the sense of convex duality and utilise its $\sobo_{\locc}^{1,2}$-regularity in conjunction subject a local $\bmo$-hypothesis to be discussed below. However, note that by our method of proof and as opposed to Theorem~\ref{thm:main0}, we get a \emph{result only for one particular generalised minimiser}, cf.~Remark~\ref{rem:onlyone}.

Precisely, for a given $\mu$-elliptic integrand of linear growth $f\colon\R_{\sym}^{n\times n}\to\R$, we  now consider the auxiliary variational principle
\begin{align}\label{eq:viscosityapprox}
\text{to minimise}\; \mathfrak{F}_{j}(v):=\mathfrak{F}[v]+\frac{1}{2j} \int_{\Omega}|\sg(v)|^{2}\dif x\;\text{over}\;\mathscr{D}_{u_{0}}:=u_{0}+\sobo_{0}^{1,2}(\Omega;\R^{n}).
\end{align}
Here we assume $u_{0}\in\sobo^{1,2}(\Omega;\R^{n})$ for simplicity. The general case $u_{0}\in\ld(\Omega)$ can be accomplished by smooth approximation and thereby can be tackled by another approximation layer which we skip here. By Korn's inequality, \eqref{eq:viscosityapprox} has a unique minimiser $v_{j}\in\mathscr{D}_{u_{0}}$ for each $j\in\mathbb{N}$, and in fact, $(v_{j})$ converges to \emph{one} generalised minimiser, and we will give estimates on the single $v_{j}$'s that eventually inherit to this particular generalised minimiser. This approach has been pursued by \textsc{Seregin} \cite{Seregin1,Seregin2,Seregin3,Seregin4} in a class of related problems and adapted by \textsc{Bildhauer} et al. \cite{Bild1,Bild3,Bild2} to the $\bv$-setting. We now collect some properties of the above viscosity approximation which essentially follow from \textsc{Seregin}'s aformentioned works; for more detail, the reader is also referred to \cite[Chpt.~5.4.3]{Gmeineder1}.

For the time being, it suffices to focus on the following
\begin{lemma}\label{lem:gathering}
Denote $(v_{j})\subset\sobo^{1,2}(\Omega;\R^{n})$ the sequence of viscosity approximations obtained in \eqref{eq:viscosityapprox}. Then the following holds:
\begin{enumerate}
\item $(v_{j})$ is uniformly bounded in $\ld(\Omega)$. 
\item $(v_{j})\subset\sobo_{\locc}^{2,2}(\Omega;\R^{n})$. 
\item There exists a non-relabelled subsequence such that $\mathfrak{F}_{j}[v_{j}]\to \inf_{\bd(\Omega)}\overline{\mathfrak{F}}_{u_{0}}$ as $j\to\infty$. 
\item Put $\tau_{j}:=f'(\sg(v_{j}))$ and $\sigma_{j}:=f'_{j}(\sg(v_{j}))$. Then the sequence $(\tau_{j})$ is uniformly bounded in $\lebe^{\infty}(\Omega;\R^{n\times n})$ and $(\sigma_{j})$ is uniformly bounded in $\sobo_{\locc}^{1,2}(\Omega;\R^{n\times n})$. 
\item There exists $\sigma\in (\lebe_{\di}^{\infty}\cap\sobo_{\locc}^{1,2})(\Omega;\R^{n})$ such that, for a suitable, non-relabelled subsequence we have $\sigma_{j}\rightharpoonup \sigma$ weakly in $\lebe^{2}(\Omega;\R^{n\times n})$ as $j\to\infty$. This map $\sigma$ is a solution to the dual problem as introduced in Section~\ref{sec:convexanalysismain}.
\end{enumerate}
Moreover, there holds 
\begin{align}\label{eq:prelimminmax}
\inf_{u_{0}+\ld_{0}(\Omega)}\mathfrak{F}[w]= \sup_{\chi\in\lebe_{\di}^{\infty}(\Omega;\R^{n\times n})}\mathcal{R}[\chi], 
\end{align}
where $\lebe_{\di}^{\infty}(\Omega;\R^{n\times n})$ denotes the linear space of all $v\in\lebe^{\infty}(\Omega;\R^{n\times n})$ which are solenoidal in the sense of distributions. 
\end{lemma}
The previous lemma does not require the following condition, which however plays in a key role in the proof of Theorem~\ref{thm:main1} as employed below. Precisely, we shall require that the viscosity approximation sequence from above satisfies for each relatively compact Lipschitz subset $K\subset\Omega$ with $\dista(K;\partial\Omega)>0$
\begin{align}\tag{LBMO}\label{eq:locBMO}
\sup_{j\in\mathbb{N}}\sup_{x\in K}\mathcal{M}_{K}^{\#}(v_{j})(x)<\infty. 
\end{align}
We now combine the results of Lemma~\ref{lem:gathering} and the improved embedding from $\bd\cap\bmo$ from Theorem~\ref{thm:bdbmoemb} to deduce the Sobolev regularity assertion of Theorem~\ref{thm:main1}. 

\begin{proof}[Proof of Theorem~\ref{thm:main0}]
Let $(v_{j})\subset u_{0}+\sobo_{0}^{1,2}(\Omega;\R^{n})$ be the sequence of viscosity approximations defined after~\eqref{eq:viscosityapprox}. Let $x_{0}\in\Omega$ and $r>0$ such that $\ball(x_{0},2r)\Subset\Omega$. For $h\in\R$ with $|h|<\dista(x_{0},\partial\Omega)-2r$ we pick a cut--off function $\rho\in\hold_{c}^{1}(\ball(x_{0},r);[0,1])$ with $\mathbbm{1}_{\ball(x_{0},r)}\leq \rho \leq \mathbbm{1}_{\ball(x_{0},2r)}$ and put $\varphi_{j}:=\tau_{s,h}^{-}(\rho^{2}\tau_{s,h}v_{j})$. By assumption, we have $\varphi_{j}\in\sobo_{0}^{1,2}(\Omega;\R^{n})$ and thus $\varphi_{j}$ is admissible in the weak Euler--Lagrange equation
\begin{align}\label{eq:anotherEuler}
\int_{\Omega}\langle f'_{j}(\sg(v_{j})),\sg(\varphi)\rangle\dif x = 0\qquad\text{for all}\;\varphi\in\sobo_{0}^{1,2}(\Omega;\R^{n})
\end{align}
which follows directly from minimality of $v_{j}$ for $\mathfrak{F}_{j}$. In consequence, we obtain for each $j\in\mathbb{N}$
\begin{align}\label{eq:gregoryEL}
\int_{\Omega}\langle f'_{j}(\sg(v_{j})),\sg(\tau_{s,h}^{-}(\rho^{2}\tau_{s,h}v_{j}))\rangle\dif x = 0.
\end{align}
Since $\sg$ and $\tau_{s,h}^{-}$ commute, we find by discrete integration by parts that 
\begin{align*}
\mathbf{I} := \int_{\Omega}\langle \tau_{s,h}(f'_{j}(\sg(v_{j}))),\rho^{2}\tau_{s,h}\sg(v_{j})\rangle\dif x = - \int_{\Omega}\langle \tau_{s,h}(f'_{j}(\sg(v_{j}))),2\rho\D\rho\odot\tau_{s,h}v_{j}\rangle\dif x =:\mathbf{II}. 
\end{align*}
Using the $\mu$--ellipticity condition, we will now suitably estimate $\mathbf{I}$ from below.
It is clear that $\mathbf{I}_{2}\geq 0$. Since $f$ is assumed to be $\mu$--elliptic, we find similar to the estimation that lead to \eqref{eq:estimateI'}
\begin{align}
\mathbf{I}\geq \mathbf{I}_{1} & \geq c\int_{\Omega}\rho^{2}\frac{|\tau_{s,h}\sg(v_{j}(x))|^{2}}{(1+|\sg(v_{j})|^{2}+|\sg(v_{j}(x+he_{s}))|^{2})^{\frac{\mu}{2}}}\dif x
\end{align}
for all $j\in\mathbb{N}$, where $c>0$ is an absolute constant. We now pause to estimate $\mathbf{II}$ conveniently and so make use of the uniform $\bmo$--hypothesis \eqref{eq:locBMO} and the regularity of the dual solution as stated in Lemma~\ref{lem:gathering}. To be more precise, recalling the notation $\sigma_{j}:=f'_{j}(\sg(v_{j}))$ for $j\in\mathbb{N}$, Lemma~\ref{lem:gathering}(c) asserts that for any relatively compact Lipschitz subset $K$ of $\Omega$ we have $\sup_{j\in\mathbb{N}}\|\sigma_{j}\|_{\dot{\sobo}^{1,2}(K;\R^{n\times n})}<\infty$. On the other hand, the local uniform $\bmo$--hypothesis in conjunction with the interpolation result of Theorem~\ref{thm:bdbmoemb} applied to $p=2$ yields that for $\varepsilon>0$ sufficiently small there holds 
\begin{align}
\sup_{j\in\mathbb{N}}\|v_{j}\|_{\dot{\sobo}^{\frac{1}{2}-\varepsilon,2}(K;\R^{n})}<\infty. 
\end{align}
Deferring the precise value of $\varepsilon>0$ to the end of the proof, we now estimate $\mathbf{II}$ for $|h|<\dista(x_{0},\partial\Omega)-2r$ by
\begin{align*}
|\mathbf{II}| & \leq  C(\rho)\int_{\ball(x_{0},2r)}|\tau_{s,h}(f'_{j}(\sg(v_{j})))|\,|\tau_{s,h}v_{j}|\dif x \\
& \leq  C(\rho)h^{1+\frac{1}{2}-\varepsilon}\int_{\ball(x_{0},2r)}|\Delta_{s,h}(f'_{j}(\sg(v_{j})))|\,\left\vert\frac{\tau_{s,h}v_{j}}{h^{\frac{1}{2}-\varepsilon}}\right\vert\dif x \\
& \leq C(\rho)h^{\frac{3}{2}-\varepsilon}\left(\int_{\ball(x_{0},2r)}|\Delta_{s,h}\sigma_{j}|^{2}\dif x\right)^{\frac{1}{2}}\left(\int_{\ball(x_{0},2r)}\left\vert\frac{\tau_{s,h}v_{j}}{h^{\frac{1}{2}-\varepsilon}}\right\vert^{2}\dif x\right)^{\frac{1}{2}}\qquad\text{(by H\"{o}lder)}\\
& \leq C(\rho)h^{\frac{3}{2}-\varepsilon}\|\sigma_{j}\|_{\dot{\sobo}^{1,2}(\ball(x_{0},2r);\rsym)}\|v_{j}\|_{\dot{\sobo}^{\frac{1}{2}-\varepsilon,2}(\ball(x_{0},2r);\R^{n})}\\
& \leq C(\rho)h^{\frac{3}{2}-\varepsilon}, 
\end{align*}
and here $C(\rho)>0$ does not depend on $j\in\mathbb{N}$. Gathering estimates and recalling the shorthand~\eqref{eq:omegadef} (with the obvious modifications) for some $1<\alpha<2$ to be fixed later, we find that
\begin{align}\label{eq:semibesov}
\int_{\ball(x_{0},2r)}\left\vert \frac{\tau_{s,h}V_{\alpha}(\sg(v_{j}))}{h^{\frac{3}{4}-\frac{\varepsilon}{2}}}\right\vert^{2}\omega_{j,h,s}(x)\dif x & \leq C(\rho).
\end{align}
By Lemma~\ref{lem:gathering}, $(v_{j})$ is uniformly bounded in $\ld(\Omega)$. Therefore, it is easily seen that $\sup_{j\in\mathbb{N}}\|\omega_{j,h,s}^{-1}\|_{\lebe^{1}(\ball(x_{0},2r))}<\infty$ provided (recall that $1<\alpha<2$ is assumed throughout)
\begin{align}\label{eq:parametercondition1}
\mu+2(1-\alpha)\leq 1,\;\;\;\text{that is,}\;\;\;\frac{\mu+1}{2}\leq \alpha (<2). 
\end{align}
Henceforth, assuming condition \eqref{eq:parametercondition1} to be in action in all of what follows, we obtain by Young's inequality 
\begin{align*}
\int_{\ball(x_{0},2r)}\left\vert \frac{\tau_{s,h}V_{\alpha}(\sg(v_{j}))}{h^{\frac{3}{4}-\frac{\varepsilon}{2}}}\right\vert\dif x & \leq  \int_{\ball(x_{0},2r)}\left\vert \frac{\tau_{s,h}V_{\alpha}(\sg(v_{j}))}{h^{\frac{3}{4}-\frac{\varepsilon}{2}}}\right\vert^{2}\omega_{j,h,s}(x)\dif x + \int_{\ball(x_{0},2r)}\frac{\dif x}{\omega_{j,h,s}(x)}\\
& \leq C(\rho)\qquad(\text{by \eqref{eq:semibesov} and \eqref{eq:parametercondition1}}), 
\end{align*}
where $C(\rho)>0$ again does not depend on $j\in\mathbb{N}$. From here we conclude that the sequence $(V_{\alpha}(\sg(v_{j}))|_{\ball(x_{0},r)})$ is uniformly bounded in $\dot{\besov}_{1,\infty}^{\frac{3}{4}-\frac{\varepsilon}{2}}(\ball(x_{0},r);\R_{\sym}^{n\times n})$. At this point, we recall from Lemma~\ref{lem:besovembedding} that for all $\delta>0$ sufficiently small there holds
\begin{align*}
\dot{\besov}_{1,\infty}^{\frac{3}{4}-\frac{\varepsilon}{2}}(\ball(x_{0},2r))\hookrightarrow \lebe^{q}(\ball(x_{0},r))\qquad\text{for all}\; q \leq \frac{n}{n-\big(\frac{3}{4}-\frac{\varepsilon}{2}\big)}-\delta = \frac{4n}{4n-3+2\varepsilon}-\delta.
\end{align*}
We may therefore deduce that for $\varepsilon>0$ and $\delta>0$ sufficiently small, we have 
\begin{align*}
\sup_{j\in\mathbb{N}}\int_{\ball(x_{0},r)}|V_{\alpha}(\sg(v_{j}))|^{\frac{4n}{4n-3+2\varepsilon}-\delta}\dif x <\infty. 
\end{align*}
By Lemma~\ref{lem:valpha}, the previous estimate implies 
\begin{align}\label{eq:uniformepsestimateSC}
\sup_{j\in\mathbb{N}}\int_{\ball(x_{0},r)}|\sg(v_{j})|^{q_{\alpha,n,\varepsilon,\delta}}\dif x :=\sup_{j\in\mathbb{N}}\int_{\ball(x_{0},r)}|\sg(v_{j})|^{(2-\alpha)(\frac{4n}{4n-3+2\varepsilon}-\delta)}\dif x <\infty, 
\end{align}
with an obvious definition of the exponent $q_{\alpha,n,\varepsilon,\delta}>0$. In conclusion, if $q_{\alpha,n,\varepsilon,\delta}>1$, then Korn's inequality yields uniform boundedness of $(\sg(v_{j})|_{\ball(x_{0},r)})$ in $\lebe^{q_{\alpha,n,\varepsilon,\delta}}(\ball(x_{0},r);\R_{\sym}^{n\times n})$ and  hence, by arbitrariness of $x_{0}\in\Omega$ and $r>0$, the claim follows. 

To establish $q_{\alpha,n,\varepsilon,\delta}>1$, let us note that the latter is equivalent to 
\begin{align*}
\alpha <\frac{4n+3-2\varepsilon-2\delta(4n-3+2\varepsilon)}{4n-\delta(4n-3+2\varepsilon)}.
\end{align*}
Sending $\varepsilon,\delta\searrow 0$, we find that $q_{\alpha,n,\varepsilon,\delta}>1$ can be achieved for sufficiently small $\varepsilon,\delta>0$ if and only if 
\begin{align}
\alpha < 1+ \frac{3}{4n}. 
\end{align}
On the other hand, recalling \eqref{eq:parametercondition1}, we must therefore have 
\begin{align}\label{eq:alphaboundbelow}
\frac{\mu+1}{2}\leq \alpha < 1+ \frac{3}{4n}, 
\end{align}
an equation which is solvable for $1<\alpha<2$ if and only if $(\mu+1)/2 < 1+\frac{3}{4n}$. The latter inequality is solvable for $\mu>1$ if and only if 
\begin{align}\label{eq:mufixSC}
\mu < 1+\frac{3}{2n},
\end{align}
which is exactly the exponent claimed in the theorem, and so we may argue as in the proof of Theorem~\ref{thm:main0} to conclude. The proof is complete. 
\end{proof}

\begin{remark}\label{rem:onlyone}\emph{
Theorem~\ref{thm:main1} establishes the higher Sobolev regularity for one generalised minimiser only. The reason why an Ekeland-type strategy as pursue in Section~\ref{sec:ekeland} is unclear to us to work in this extended range of $\mu$ is Lemma~\ref{lem:gathering}. In this case, the Ekeland approximation sequence must uniformly satisfy the local $\bmo$-bound and so, following \cite[Sec.~5.1]{BS1}, we need to stabilise not with the quadratic Dirichlet integral but an $n$-th order Dirichlet integral (as $\sobo^{1,n}\hookrightarrow\bmo$). However, then it is unclear to us to employ the required $\sobo_{\locc}^{1,2}$-bounds on the respective terms which \emph{should} converge to the dual solution. This, however, we intend to treat in a future publication.}

\emph{Let us further remark that it does not seem obvious how to use possibly good (e.g., radial) structure of the integrands to deduce even higher regularity such as on the $\hold^{k,\alpha}$-scale. This essentially stems from the fact that the symmetric gradients seem to destroy such good structure, and basically rules out the possiblity of employing De Giorgi- or Moser-type strategies.}
\end{remark}
In analogy with Corollary~\ref{cor:nikolskii}, we can extract some more information from the above proof. 
The exponent $q_{\alpha,n,\varepsilon,\delta}>1$ as given in \eqref{eq:uniformepsestimateSC} is optimised for the smallest admissible value of $\alpha$. This, in turn is given by $(\mu+1)/2$ by \eqref{eq:alphaboundbelow} subject to the condition $\mu<1+\frac{3}{2n}$ from \eqref{eq:mufixSC}. 

Now, sending $\alpha\searrow (\mu+1)/2$ for the admissible range of $\mu$ yields by \eqref{eq:uniformepsestimateSC} that for all $\gamma>0$ and $\kappa>0$ suitably small we have
\begin{align}\label{eq:mufixSC3}
\sg(u)\in \lebe_{\locc}^{(2-\frac{\mu+1}{2}-\gamma)(\frac{4n}{4n-3}-\kappa)}(\Omega;\R_{\sym}^{n\times n}). 
\end{align}
We consider now the condition (with $\Gamma:=\kappa(4n-3)$)
\begin{align}\label{eq:anothermucond}
\begin{split}
\mu \stackrel{!}{\leq}(2-\frac{\mu+1}{2}-\gamma)(\frac{4n}{4n-3}-\kappa) & \Leftrightarrow (8n-6)\mu\leq (3-\mu-2\gamma)(4n-\Gamma) \\
&\Leftrightarrow 8n\mu - 6\mu +4n\mu\leq 12n-8\gamma n-\Gamma(3-\mu-2\gamma)\\
&\Leftrightarrow \mu(12n-6)\leq 12n-8\gamma n-\Gamma(3-\mu-2\gamma)\\
&\Leftrightarrow \mu \leq \frac{2n}{2n-1}-\frac{8\gamma n+\Gamma(3-\mu-2\gamma)}{12n-6}.
\end{split}
\end{align}
Sending $\gamma,\kappa\searrow 0$ which in turn implies $\Gamma\searrow 0$, we obtain that $\sg(u)\in\lebe_{\locc}^{\mu}(\Omega;\R_{\sym}^{n\times n})$ provided $\mu<\frac{2n}{2n-1}$. Let us carefully note that this is more restrictive than the bound provided by \eqref{eq:mufixSC}: In fact, 
\begin{align*}
\frac{2n}{2n-1}<\frac{2n+3}{2n}\Leftrightarrow 4n^{2}<4n^{2}-2n+6n-3 \Leftrightarrow \frac{3}{4}<n \stackrel{n\in\mathbb{N}}{\Leftrightarrow}n\geq 1.
\end{align*} 
\begin{corollary}\label{cor:SCNikolskii}
Let $\Omega\subset\R^{n}$ be a bounded Lipschitz subset and let $f\in\hold^{2}(\R^{n\times n})$ be a $\mu$-elliptic variational integrand of linear growth with $1<\mu<\frac{2n}{2n-1}$ and suppose that the sequence of viscosity approximations $(v_{j})$ satisfies \eqref{eq:locBMO}. Then the following holds for the weak*-limit $u\in\gm(\mathfrak{F};u_{0})$:
\begin{enumerate}
\item There exists $Q=Q(\mu)>1$ such that $u\in\ld(\Omega)\cap\sobo_{\locc}^{2,Q}(\Omega;\R^{n})$.
\item There holds $\dim_{\mathcal{H}}(\Sigma_{u})<n-1$, where $\Sigma_{u}$ is defined as in \eqref{eq:singset}. 
\end{enumerate}
\end{corollary}
Since the verification of this corollary is along the lines of the proof of Corollary~\ref{cor:HausdorffCor1}, we only point out the requisite key points in a 
\begin{proof}[Sketch of Proof.]
We now switch to the situation of the proof of Corollary~\ref{cor:HausdorffCor1}, where the overall main point is the derivation of inequality~\eqref{eq:secondderivativesuniform}. We test \eqref{eq:gregoryEL} by $\varphi:=-\Delta_{s,h}^{-}(\rho^{2}\Delta_{s,h}v_{j})$. The proof then evolves along the same lines, and it is only crucial to estimate the terms corresponding to $\mathbf{J}_{2}$ as in \eqref{eq:jayjay}; note that the term $\mathbf{J}_{3}$ now \emph{does not arise}. The term $\mathbf{J}_{2}^{(2)}$ is handled analogously, now using that $\frac{1}{j}\int_{\Omega}|\sg(v_{j})|^{2}<C$ uniformly in $j$, cf.~Lemma~\ref{lem:gathering}(c). To deal with the equivalent of $\mathbf{J}_{2}^{(1)}$, the critical part is the estimation \eqref{eq:K3estimate}. Here the exponent appearing inside the second integral on the right hand side must be estimated by virtue of the uniform local higher integrability. This, in turn, is ensured by $1<\mu<\frac{2n}{2n-1}$, cf.~\eqref{eq:anothermucond}. We can then conclude as before to arrive at the result. 
\end{proof}
Finally, note that because of Corollary~\ref{cor:bvembedding}, the strategy pursued in this section offers a difference quotient alternative to \cite{Bild1} subject to the respective ellipticity regime; however, note that here much stronger results apply,cf.~\cite{Bild1,Bild2} for more information. 

\subsection{Uniqueness of Generalised Minimisers}\label{sec:uniqueness}

A consequence of Theorem \ref{thm:main0} is the following result on the uniqueness of generalised minimisers. Similarly to functionals of linear growth depending on the gradient (see \cite[Sec.~5]{BS1}), uniqueness of generalised minimisers can only be obtained modulo rigid deformations, that is, elements of the nullspace of $\sg$: 
\begin{theorem}[Uniqueness]\label{thm:unique}
Let $f\colon\R_{\sym}^{n\times n}\to\R$ be a $\mu$--elliptic integrand of linear growth with $1<\mu<\frac{n+1}{n}$. Suppose that $\Omega$ is an open, bounded and connected Lipschitz subset of $\R^{n}$. Then any two generalised minimisers $u,v\in\gm(\mathfrak{F};u_{0})$ differ by a rigid deformation, that is, there exists $R\in\mathcal{R}(\Omega)$ such that $u=v+R$ holds $\mathscr{L}^{n}$--a.e. in $\Omega$. 
\end{theorem}
\begin{proof}
Let $u,v\in\gm(\mathfrak{F};u_{0})$ be two generalised minimisers with respect to a prescribed Dirichlet class $\mathscr{D}_{u_{0}}:=u_{0}+\ld_{0}(\Omega)$. Since $f$ is $\mu$--elliptic with $1<\mu\leq 1+\frac{1}{n})$, both $u$ and $v$ belong to $\ld(\Omega)$ by Theorem \ref{thm:main0}. We will show $\sg(u)=\sg(v)$, and this will imply the claim: Indeed, since $\sg(w)=0$ is equivalent to $w\in\mathcal{R}(\Omega)$ provided $\Omega$ is connected, we deduce that there exists $R\in\mathcal{R}(\Omega)$ such that $u=v+R$. To prove the claim, suppose that $\sg(u)\neq\sg(v)$ on a measurable set $U$ with $\mathscr{L}^{n}(U)>0$. Then we obtain, using that $f$ is strictly convex and both $\E^{s}u$ and $\E^{s}v$ vanish identically in $\Omega$, 
\begin{align*}
\overline{\mathfrak{F}}\left[\tfrac{1}{2}(u+v)\right]<\frac{1}{2}(\overline{\mathfrak{F}}[u]+\overline{\mathfrak{F}}[v])=\min\overline{\mathfrak{F}}[\bd(\Omega)], 
\end{align*}
an obvious contradiction. The proof is complete.
\end{proof}
Building on the results of the previous sections, particularly to the proof of the higher Sobolev regularity of generalised minimisers, we now briefly comment on the uniqueness issues addressed in the introduction. In general, the failure of uniqueness of minima of variational integrals \eqref{eq:varprin} is mostly due to two reasons (compare \cite{BS1}): Going back to the relaxed functional $\overline{\mathfrak{F}}$ given by \eqref{eq:relaxed}, positive homogeneity of $f^{\infty}$ implies that $f^{\infty}$ is not strictly convex even if $f$ is. Thus a possible reason for non--uniqueness is the presence of the singular part of minimisers which genuinely only effects the recession parts of $\overline{\mathfrak{F}}$. The second reason for non--uniqueness is a possible non--attainment of the correct boundary values which is partly addressed in  
\begin{proposition}\label{prop:uniqueness1}
Let $\Omega$ be a convex Lipschitz subset of $\R^{n}$. Suppose that generalised minima of the variational integral $\mathfrak{F}$ given by \eqref{eq:varprin} are unique modulo rigid deformations. If one generalised minimiser $u$ attains the correct boundary values in the sense that $\trace (u-u_{0})=0$ $\mathcal{H}^{n-1}$--a.e. on $\partial\Omega$, then $\gm(\mathfrak{F};u_{0})=\{u\}$. 
\end{proposition}
\begin{proof}
Let $R\in\mathcal{R}(\Omega)\setminus\{0\}$ be an arbitrary non--zero rigid deformation and denote $\overline{R}$ its continuous extension to $\overline{\Omega}$. Then we have 
\begin{align}\label{eq:uniquenessestimate}
\overline{\mathfrak{F}}[u+R]=\overline{\mathfrak{F}}[u]+\int_{\partial\Omega}f^{\infty}\left(-\overline{R}\odot\nu_{\partial\Omega}\right)\dif\mathcal{H}^{n-1}
\end{align}
because $\trace(u-u_{0})=0$ $\mathcal{H}^{n-1}$--a.e. on $\partial\Omega$. Since the mapping $T\colon\partial\Omega\to\R_{\sym}^{n\times n}$ given by $T(x):=-\overline{R}\odot\nu_{\partial\Omega}$ for $x\in\partial\Omega$ is continuous and $f^{\infty}\colon\R_{\sym}^{n\times n}\to\R_{\geq 0}$ is continuous too, it suffices to show that there exists $z\in\partial\Omega$ such that $|\overline{R}(z)\odot\nu_{\partial\Omega}(z)|>0$. Indeed, in this case we conclude by homogeneity of $f^{\infty}$ and positivity of $f$ that the boundary integral on the right side of \eqref{eq:uniquenessestimate} is strictly positive so that $u+R$ is not a minimiser of $\overline{\mathfrak{F}}$ over $\bd(\Omega)$. The proof is the concluded by Proposition \ref{prop:equivalencegenmins} which provides the required characterisation of generalised minima in terms of $\overline{\mathfrak{F}}$. For simplicity, we shall argue for the unit ball $\Omega=\ball$ and only sketch the respective generalisation to arbitrary open Lipschitz domains $\Omega$ below. Write $\overline{R}(z)=Az+b$. If $|\overline{R}(z)\odot\nu_{\partial\!\ball}(z)|=0$ for all $z\in\partial\!\ball$, then $Az\odot\nu_{\partial\!\ball}(z)=-b\odot\nu_{\partial\!\ball}(z)$ for all $z\in\partial\!\ball$. Since $\nu_{\partial\!\ball}(z)=z$ for any $z\in\partial\!\ball$, this particularly implies $Ae_{k}\odot e_{k}=-b\odot e_{k}$ for all $k=1,...,n$. These identities imply 
\begin{align*}
Ae_{k}\odot e_{k}=\left(
  \begin{array}{ccccccc} 
   0& \hdots  & a_{1k} & 0 & \hdots \\
   0 & \hdots  & \vdots & 0 & \hdots \\
   a_{1k} & \hdots  & a_{kk}  & \hdots & a_{nk} \\
   0& \hdots  & \vdots & 0 & \hdots \\   
   0& \hdots  & a_{nk} & 0 & \hdots \\   
  \end{array}
  \right) = - \left(\begin{array}{ccccccc} 
   0& \hdots  & b_{1} & 0 & \hdots \\
   0 & \hdots  & \vdots & 0 & \hdots \\
   b_{1} & \hdots  & b_{k}  & \hdots & b_{n} \\
   0& \hdots  & \vdots & 0 & \hdots \\   
   0& \hdots  & b_{n} & 0 & \hdots \\   
  \end{array}
  \right)=-b\odot e_{k}. 
\end{align*}
and hence $a_{jk}=-b_{j}$ for all $j,k=1,...,n$. In consequence, $a_{jj}=-b_{j}$ for all $j=1,...,n$, but by scew--symmetry of $A$, $a_{jj}=0$ and thus $b_{j}=0$ for all $j=1,...,n$. This further implies $a_{jk}=0$ for all $j,k=1,...,n$ and thus $\overline{R}\equiv 0$. If $\Omega$ is not a ball, then one may argue similarly, now using the fact that for any open, bounded and convex Lipschitz subset $\Omega$ of $\R^{n}$ there exist linearly independent $z_{1},...,z_{n}\in\partial\Omega$ such that $\nu_{\partial\Omega}(z_{1}),...,\nu_{\partial\Omega}(z_{n})$ are linearly independent too. The details are left to the interested reader.
\end{proof}
The previous lemma is an adaptation of \cite[Lem.~5.5]{BS1} to the symmetric gradient situation. Finally, the second possible source of non-uniqueness is given by the boundary behaviour of generalised minima. This is in the spirit of \textsc{Santi}'s example \cite{Santi} which has been revisited and adapted to the vectorial case by \textsc{Beck \& Schmidt} (cf.~\cite[Thm.~1.17]{BS1}). As such, we believe that is possible by a similar adapation as has been given in Proposition~\ref{prop:uniqueness1} above to generalise \cite[Thm.~1.16]{BS1} to the symmetric gradient situation. More precisely, we conjecture that if $f\colon\R_{\sym}^{n\times n}\to\R$ is a convex integrand with \eqref{eq:lg} such that for every $\eta\in\R^{n}\setminus\{0\}$, $f_{\eta}\colon\xi\mapsto f^{\infty}(\eta\odot \xi)$ is a strictly convex norm\footnote{in the sense that if $f_{\eta}(\xi_{1})=f_{\eta}(\xi_{2})=f_{\eta}(\lambda \xi_{1}+(1-\lambda)\xi_{2})$ for $\xi_{1},\xi_{2}\in\R^{n}$ and $0<\lambda<1$, then $\xi_{1}=\xi_{2}$.} and if generalised minima are unique modulo rigid deformations, then the set of all of generalised minima can be written as 
\begin{align*}
\gm(\mathfrak{F};u_{0})= \big\{u+\lambda R\colon\;-1\leq \lambda \leq 1 \big\}
\end{align*}
for some fixed $u\in\gm(\mathfrak{F};u_{0})$ and $R\in\mathcal{R}(\Omega)$. However, the verification of this is beyond the scope of this paper and shall be addressed in a future work.

\section{Appendix}\label{sec:appendix}
\subsection{Extensions of Theorems~\ref{thm:main0} and \ref{thm:main1} to nonautonomous problems} Let us now briefly comment on the situation where $f$ has additional $x$-dependence. If $f\in\hold^{2}(\overline{\Omega}\times\R_{\sym}^{n\times n})$ is an integrand that satisfies essentially the assumptions of \cite[Ass.~4.22]{Bild1}, that is, $f$ satisfies \eqref{eq:lg} uniformly in $x$ together with
\begin{align}
\begin{cases}
\sup_{x\in\overline{\Omega}}\sup_{\xi\in\R_{\sym}^{n\times n}}\max\{|\D_{\xi}f(x,\xi)|,|\D_{x}^{2}\D_{\xi}f(x,\xi)|,|\D_{x}\D_{\xi}f(x,\xi)|\}<\infty,\\ 
\lambda\frac{|\xi|^{2}}{(1+|z|^{2})^{\mu/2}}\leq \langle \D_{\xi}^{2}f(z)\xi,\xi\rangle \leq \Lambda \frac{|\xi|^{2}}{(1+|z|^{2})^{1/2}},\\
|\langle \D_{x}\D_{\xi}^{2}f(x,\xi)\eta,\eta'\rangle|\leq C(|\langle\D_{\xi}^{2}f(x,\xi)\eta,\eta'\rangle| + |\eta|\,|\eta'|/(1+|\xi|^{2}))
\end{cases}
\end{align}
for all $x\in\overline{\Omega},\eta,\eta',\xi,z\in\R_{\sym}^{n\times n}$, then the results of this paper carry over in a straightforward manner to the situation of interest; in fact, as we work with finite differences, these assumptions can even be weakened, but this is left to the interested reader; also see the discussion in \cite[App.~C]{BS1}. If the smoothness of the $x$-dependence is diminished, a merger of the arguments outlined in this work with the foundational work of \textsc{Mingione} (cp.~\cite{Min1,Min2,KM1}) leads to the correspondingly modified theorems. 
\subsection{Proofs of auxiliary results}
In this section, we provide the proofs of minor auxiliary results used in the main body of the paper. We begin with the 
\begin{proof}[Proof of Proposition~\ref{prop:cheapembedding}]
Let $u\in\hold_{c}^{\infty}(\R^{n};\R^{n})$ and fix $s<t<1$. By the \textsc{Smith} representation formula \eqref{eq:smith}, $\Phi\colon\hold_{c}^{\infty}(\R^{n};\R_{\sym}^{n\times n})\ni\sg(u)\mapsto u\in \hold_{c}^{\infty}(\R^{n};\R^{n})$  is a Riesz potential operator of order one acting by convolution. Hence, a routine estimation yields for $x,y\in\R^{n}$, $x\neq y$ and $0<t<1$
\begin{align}\label{eq:fractional}
\begin{split}
|u(x)-u(y)| & = |\Phi(\bm{\varepsilon}(u))(x)-\Phi(\bm{\varepsilon}(u))(y)| \\ 
& \leq C |x-y|^{t}\int_{\R^{n}} |\bm{\varepsilon}(u)(z)|\left(\frac{1}{|z-x|^{n-1+t}}+\frac{1}{|y-z|^{n-1+t}}\right)\dif z.
\end{split}
\end{align}
Fixing a ball $\ball=\ball(z,r)\subset\R^{n}$, dividing \eqref{eq:fractional} by $|x-y|^{n+s}$ and integrating with respect to $x,y\in\ball$ consequently yields by symmetry for a suitable cut-off function $\eta\in\hold_{c}^{\infty}(\R^{n};[0,1])$
\begin{align*}
\iint_{\ball\times\ball}\frac{|u(x)-u(y)|}{|x-y|^{n+s}}\dif\,(x,y) & \leq C\iint_{\ball\times\ball}\underbrace{\int_{\R^{n}}\frac{|\bm{\varepsilon}(u)(z)|}{|z-x|^{n-1+s}}\dif z}_{:=F(x)}\frac{\dif\,(x,y)}{|x-y|^{n+s-t}}\\
& \leq C \int_{\ball'}\Big(\int_{\R^{n}}\frac{\eta(x)F(x)}{|x-y|^{n+s-t}}\dif x\Big)\dif y, 
\end{align*}
where $\ball'=\ball(x,R)$ for some suitable $0<r<R<\infty$. Now we use Young's convolution inequality twice and employ the fact that $x\mapsto |x|^{-n-s+t}$ and $x\mapsto |x|^{-n+1-s}$ are integrable over any ball $\ball(x,R)$ as long as $R<\infty$ to conclude. The rest follows by standard localisation and approximation arguments which we omit here. 
\end{proof}
\begin{proof}[Proof of Lemma~\ref{lem:valpha}]
By \cite[Lemma 2.2]{AcerbiFusco}, for every $-\frac{1}{2} < \gamma <0$ and $\mu\geq 0$ there exists a constant $c=c(M)>0$ such that for all $\xi,\eta\in\R^{M}$ there holds 
\begin{align*}
(2\gamma+1)|\xi-\eta|\leq \frac{|(\mu^{2}+|\xi|^{2})^{\gamma}\xi-(\mu^{2}+|\eta|^{2})^{\gamma}\eta|}{(\mu^{2}+|\xi|^{2}+|\eta|^{2})^{\gamma}}\leq \frac{c(M)}{2\gamma+1}|\xi-\eta|. 
\end{align*}
Applying this with $\mu=1$ and $\gamma=(1-\alpha)/2$ yields the claim as $-\frac{1}{2}<\gamma<0$ if and only if $1<\alpha<2$. 

Now let $\xi\in\R^{M}$ with $|\xi|\geq 1$ and let $1<\alpha<2$. Then $(1-\alpha)/2<0$. Hence, since $t\mapsto t^{(1-\alpha)/2}$ is monotonically decreasing on $\R_{>0}$, 
\begin{align*}
|\xi|>1 & \Rightarrow |\xi|^{2}>1 \Rightarrow 2|\xi|^{2}>1+|\xi|^{2} \Rightarrow 2^{\frac{1-\alpha}{2}}|\xi|^{1-\alpha}\leq (1+|\xi|^{2})^{\frac{1-\alpha}{2}}\\ & \Rightarrow |\xi|^{2-\alpha}\leq 2^{\frac{\alpha-1}{2}}|V_{\alpha}(\xi)| \stackrel{1<\alpha<2}{\leq}\sqrt{2}|V_{\alpha}(\xi)|. 
\end{align*}
Since $1<\alpha<2$ and $|\xi|\geq 1$, we have $|\xi|^{2-\alpha}\leq |\xi|$ and thus $\min\{|\xi|,|\xi|^{2-\alpha}\}\leq \sqrt{2}|V_{\alpha}(\xi)|$ in this case. Now, if $|\xi|<1$, then 
\begin{align*}
2^{\frac{1-\alpha}{2}}\leq (1+|\xi|^{2})^{\frac{1-\alpha}{2}}\Rightarrow 2^{\frac{1-\alpha}{2}}|\xi|\leq |V_{\alpha}(\xi)|\Rightarrow |\xi|\leq \sqrt{2}|V_{\alpha}(\xi)|, 
\end{align*}
and hence we see because of $|\xi|\leq |\xi|^{2-\alpha}$ in this regime that $\min\{|\xi|,|\xi|^{2-\alpha}\}\leq \sqrt{2}|V_{\alpha}(\xi)|$ holds, too. Hence $\min\{|\xi|,|\xi|^{2-\alpha}\}\leq \sqrt{2}|V_{\alpha}(\xi)|$ for all $\xi\in\R^{M}$. Lastly if the measurable function $u\colon\Omega\to\R^{M}$ is such that $V_{\alpha}(u)\in\lebe^{p}(\Omega;\R^{m})$, then we have 
\begin{align*}
\int_{\Omega}|u|^{(2-\alpha)p}\dif x & = \int_{\Omega\cap\{|u|\leq 1\}}|u|^{(2-\alpha)p}\dif x + \int_{\Omega\cap\{|u|> 1\}}|u|^{(2-\alpha)p}\dif x \\ & \leq \mathscr{L}^{n}(\Omega)+c(p)\int_{\Omega}|V_{\alpha}(u)|^{p}\dif x 
\end{align*}
The proof is complete. 
\end{proof}
We now proceed to the 
\begin{proof}[Proof of Lemma~\ref{lem:dororeduction}]
Fix $x_{0}\in\R^{n}$ and let $\widetilde{Q}$ be an arbitrary cube with $x\in\widetilde{Q}$ and $\mathscr{L}^{n}(Q)=t^{n}$. It is easy to see that there exists $K=K(n)>0$ such that the cube $Q=Q(x_{0},Kt)$ (which has center $x_{0}$) contains $\widetilde{Q}$. We then obtain  
\begin{align*}
\frac{1}{t^{\alpha}}\dashint_{\widetilde{Q}}|u-(u)_{\widetilde{Q}}|\dif y & = \frac{1}{|\widetilde{Q}|^{1+\frac{\alpha}{n}}}\int_{\widetilde{Q}}|u-(u)_{\widetilde{Q}}|\dif x \leq \frac{K^{n+\alpha}}{|Q|^{1+\frac{\alpha}{n}}}\int_{Q}|u-(u)_{\widetilde{Q}}|\dif x \\
& \stackrel{\text{Jensen}}{\leq}\frac{K^{n+\alpha}}{|Q|^{1+\frac{\alpha}{n}}}\int_{Q}\dashint_{\widetilde{Q}}|u(x)-u(y)|\dif y\dif x \\
& \leq \frac{K^{n+\alpha}}{|Q|^{1+\frac{\alpha}{n}}}\frac{K^{n}}{|Q|}\int_{Q}\int_{\widetilde{Q}}|u(x)-u(y)|\dif y\dif x \\
& \leq \frac{K^{2n+\alpha}}{|Q|^{1+\frac{\alpha}{n}}}\dashint_{Q}\int_{Q}|u(x)-u(y)|\dif y\dif x \\
& \leq \frac{K^{2n+\alpha}}{|Q|^{1+\frac{\alpha}{n}}}\dashint_{Q}\int_{Q}|u(x)-(u)_{Q}+(u)_{Q}-u(y)|\dif y\dif x\\
& \leq \frac{2K^{2n+\alpha}}{|Q|^{1+\frac{\alpha}{n}}}\int_{Q}|u(x)-(u)_{Q}|\dif x \leq 2K^{2n+\alpha}(\mathcal{M}_{\alpha}^{\#}u)(x_{0}). 
\end{align*}
We may therefore put $C(n,\alpha):=2K^{2n+\alpha}$. The proof is complete. 
\end{proof}
\subsection{Calder\'{o}n spaces}
Other than the frequently function spaces $\bv$, $\bv$ or $\besov_{p,q}^{s}$, the Calder\'{o}n spaces only appear at a single point in the main text and so their definition is given in this part of the appendix. Let $\alpha>0$ and, for a given map $u\in\lebe_{\locc}^{1}(\R^{n};\R^{N})$ define its fractional sharp maximal operator $\mathcal{M}_{\alpha}^{\#}u$ by \eqref{eq:sharpfrac}.
\begin{definition}[Calder\'{o}n Spaces on $\R^{n}$, {\cite[Chpt.~6]{DS}}]
Let $1\leq p \leq\infty$ and $\alpha>0$. The \emph{Calder\'{o}n space} $\mathscr{C}^{\alpha,p}(\R^{n};\R^{m})$ is defined by
\begin{align*}
& \mathscr{C}^{\alpha,p}(\R^{n};\R^{m}):=\{v\in\lebe^{p}(\R^{n};\R^{m})\colon\;\mathcal{M}_{\alpha}^{\#}v\in\lebe^{p}(\R^{n})\},
\end{align*}
and its elements are normed by $\|u\|_{\mathscr{C}^{s,p}}:=\|u\|_{\lebe^{p}(\R^{n};\R^{m})}+\|\mathcal{M}_{\alpha}^{\#}u\|_{\lebe^{p}(\R^{n})}$.  
\end{definition}
We finally link the Besov spaces to the Calder\'{o}n spaces: \index{Calder\'{o}n spaces}
\begin{lemma}[\cite{DS}, Theorems 7.1 and 7.5]\label{lem:devoresharpleylemma}
Let $\alpha>0$. Then for any $1\leq p < \infty$ there holds 
\begin{align*}
\besov_{p,p}^{\alpha}(\R^{n})\hookrightarrow \mathscr{C}^{\alpha,p}(\R^{n})\hookrightarrow \besov_{p,\infty}^{\alpha}(\R^{n}). 
\end{align*}
\end{lemma}
As an important consequence of the preceding lemma, we record the embeddings 
\begin{align}\label{eq:calderonbesovembedding}
\sobo^{\alpha,p}(\R^{n})\hookrightarrow \mathscr{C}^{\alpha,p}(\R^{n})\hookrightarrow \besov_{\infty}^{\alpha,p}(\R^{n}), \qquad 1<p<\infty,\;\alpha>0.
\end{align}

\subsection{Relaxation} \label{sec:relaxation}
As mentioned in the introduction, we now give justification for some results used in the main part of the paper. The primary aim of this section is to explain formula \eqref{eq:relaxed} and the existence of generalised minima. We thus recap the requisite foundational theory of functions of measures as exposed, e.g., in \cite{DT,Anz}.
\subsubsection{Convex Functions of Measures}\label{sec:convexfunctions of measures}
Given $m\in\mathbb{N}$, let $f\colon\R^{m}\to\R_{\geq 0}\R$ be a convex function of linear growth, i.e., $f$ satisfies \eqref{eq:lg} with the obvious modifications. In this situation, it can be shown that $f^{\infty}$ defined by \eqref{eq:recession} is well--defined, convex and positively $1$--homogeneous. Let $\mu$ be an $\R^{m}$--valued Radon measure of finite total variation on an open and bounded set $\Omega\subset\R^{n}$. We denote 
\begin{align*}
\mu=\mu^{a}+\mu^{s}=\frac{\dif\mu}{\dif\mathscr{L}^{n}}\mathscr{L}^{n}+\frac{\dif\mu}{\dif|\mu^{s}|}|\mu^{s}|
\end{align*}
its Radon--Nikod\'{y}m decomposition into its absolutely continuous and singular parts $\mu^{a},\mu^{s}$ with respect to Lebesgue measure, where $|\mu^{s}|$ denotes the total variation measure of $\mu^{s}$. We then define a new Radon measure $f[\mu]$ by 
\begin{align}\label{eq:newRadonmeasure}
f[\mu](A):=\int_{A}f\left(\frac{\dif\mu}{\dif\mathscr{L}^{n}}\right)\dif\mathscr{L}^{n}+\int_{A}f^{\infty}\left(\frac{\dif\mu}{\dif |\mu^{s}|}\right)\dif |\mu^{s}|,\qquad A\in\mathscr{B}(\Omega), 
\end{align}
where $\mathscr{B}(\Omega)$ denotes the Borel--$\sigma$--algebra on $\Omega$. We note that, by positive $1$--homogeneity of $f^{\infty}$, this gives rise to a well-defined measure indeed. Linking this to the area functional as required for the definition of area-strict convergence, for a given map $u\in\bd(\Omega)$, we have with $f:=\sqrt{1+|\cdot|^{2}}$ that $\sqrt{1+|\E u|^{2}}(\Omega):=f[\E u](\Omega)$. 

We turn to formula \eqref{eq:relaxed} for the relaxed functional as given for $\bv$--functions in \cite{GMS1} and find by a straightforward applications of the results of \textsc{Goffman \& Serrin} \cite{GoffSerrin} that, given an open and bounded Lipschitz subset $\Omega$ of $\R^{n}$ together with a Dirichlet datum $u_{0}\in\ld(\Omega)$, we have
\begin{align}\label{eq:relaxinequality}
\overline{\mathfrak{F}}_{u_{0}}[u] = \inf\left\{\liminf_{k\to\infty}\mathfrak{F}[u_{k}]\colon\;\begin{array}{c} (u_{k})\subset\mathscr{D}_{u_{0}}:=u_{0}+\ld_{0}(\Omega) \\ u_{k}\to u\;\text{in}\;\lebe^{1}(\Omega;\R^{n})\end{array}\right\}.
\end{align}
We pick a ball $\ball=\ball(z,R)$ with $\Omega\Subset\ball$. By surjectivity of the trace operator $\trace\colon\ld(U)\to\lebe_{\mathcal{H}^{n-1}}^{1}(\partial U;\R^{n})$ on bounded Lipschitz subsets $U$ of $\R^{n}$ (see Section \ref{sec:bd}) that there exists $v_{0}\in\ld(\ball\setminus\overline{\Omega})$ such that $\trace(v_{0})|_{\partial\ball}=0$ and $\trace(v_{0})|_{\partial\Omega}=\trace(u_{0})|_{\partial\Omega}$ $\mathcal{H}^{n-1}$--a.e. on $\partial\!\ball$ or $\partial\Omega$, respectively. Given $u\in\bd(\Omega)$, we put 
\begin{align}\label{eq:extension}
\widetilde{u}(x):=\begin{cases} 
u(x)&\;\text{for}\;x\in\Omega,\\
v_{0}(x)&\;\text{for}\;x\in\ball\setminus\overline{\Omega}.
\end{cases}
\end{align}
Then there holds $\widetilde{u}\in\bd(\ball)$, and by the interior trace theorem as recalled in Section \ref{sec:bd} we have 
\begin{align*}
\E \widetilde{u} & = \E^{a}\widetilde{u}+\E^{s}\widetilde{u} = \E^{a}u\mres\Omega + \E^{s}u\mres\Omega + \E^{s}\widetilde{u}\mres\partial\Omega + \E^{a}\widetilde{u}\mres(\ball\setminus\overline{\Omega})\\
& = \mathscr{E}u\mathscr{L}^{n}\mres\Omega+\frac{\dif \E u}{\dif |\E^{s}u|}|\E^{s}u|+\trace(u-v_{0})\odot\nu_{\partial\Omega}\mathcal{H}^{n-1}\mres\partial\Omega + \mathscr{E}v_{0}\mathscr{L}^{n}\mres(\ball\setminus\overline{\Omega}). 
\end{align*}
We insert this expression for $\mu=\E\widetilde{u}$ with $A=\ball$ into \eqref{eq:newRadonmeasure} and obtain 
\begin{align}\label{eq:derivationofsplitting}
\begin{split}
f[\E\widetilde{u}](\ball) & = \int_{\Omega}f(\mathscr{E}u)\dif\mathscr{L}^{n}+\int_{\Omega}f\left(\frac{\dif \E u}{\dif |\E^{s}u|}\right)\dif |\E^{s}u| \\ & + \int_{\partial\Omega}f^{\infty}\left(\trace(u-v_{0})\odot\nu_{\partial\Omega}\right)\dif\mathcal{H}^{n-1} + \int_{\ball\setminus\overline{\Omega}}f(\mathscr{E}v_{0})\dif\mathscr{L}^{n}. 
\end{split}
\end{align}
If we then aim for minimising $f[\E\widetilde{u}](\ball)$ over all $u\in\bd(\Omega)$, we see by constancy of the very last term of the preceding expression that it does not affect the minimiser $v\in\bd(\Omega)$, and thus a function $v\in\bd(\Omega)$ minimises $f[\E\widetilde{u}](\ball)$ if and only if it minimises the relaxed functional given by \eqref{eq:relaxed}. 

We conclude this section by recalling two results due to Reshetnyak concerning the (lower semi--)continuity of convex functions of measures in the version as given in \cite{BS1}.
\begin{proposition}[{\textsc{Reshetnyak}}, \cite{Reshetnyak}]\label{prop:Reshetnyak}
Let $m\in\mathbb{N}$, $\Omega\subset\R^{n}$ open and let $(\mu_{k})$ be a sequence of $\R^{m}$--valued Radon measures of finite total variation which converges to a $\R^{m}$--valued Radon measure of finite total variation $\mu$ on $\Omega$ in the weak*--sense. Moreover, assume that all measures $\mu_{k}$ and $\mu$ take values in some closed convex cone $K\subset\R^{m}$. Then the following holds: 
\begin{enumerate}
\item \emph{Lower Semicontinuity.} If $\widetilde{f}\colon K \to [0,\infty]$ is a lower semicontinuous function, then there holds 
\begin{align*}
\int_{\Omega}\widetilde{f}\left(\frac{\dif \mu}{\dif |\mu|}\right)\dif |\mu| \leq \liminf_{k\to\infty}\int_{\Omega}\widetilde{f}\left(\frac{\dif \mu_{k}}{\dif |\mu_{k}|}\right)\dif |\mu_{k}|.
\end{align*}
\item If $\mu_{k}\to \mu$ strictly\footnote{In the sense that $\mu_{k}\stackrel{*}{\rightharpoonup}\mu$ and $|\mu_{k}|(\Omega)\to|\mu|(\Omega)$ as $k\to\infty$.} as $k\to\infty$ and $\widetilde{f}\colon K \to [0,\infty)$ is a continuous and $1$--homogeneous function, then there holds
\begin{align*}
\int_{\Omega}\widetilde{f}\left(\frac{\dif \mu}{\dif |\mu|}\right)\dif |\mu| = \lim_{k\to\infty}\int_{\Omega}\widetilde{f}\left(\frac{\dif \mu_{k}}{\dif |\mu_{k}|}\right)\dif |\mu_{k}|.
\end{align*}
\end{enumerate}
\end{proposition}
\subsubsection{Generalised Minima: Existence and Characterisations}
We now pass on to the verification of \eqref{eq:relaxinequality} and establish the existence of generalised minima.
\begin{proposition}\label{prop:equivalencegenmins}
Let $\Omega$ be an open and bounded Lipschitz subset of $\R^{n}$. Given a convex integrand $f\colon\R_{\sym}^{n\times n}\to\R$ with \eqref{eq:lg} and a boundary datum $u_{0}\in\ld(\Omega)$, define $\overline{\mathfrak{F}}_{u_{0}}$ by \eqref{eq:relaxed}. Then there exists a generalised minimiser of $\mathfrak{F}$ in the sense of \eqref{eq:GMdef}. 

Moreover, the following are equivalent for $u\in\bd(\Omega)$: 
\begin{enumerate}
\item $u$ is a generalised minimiser in the sense of \eqref{eq:GMdef}. 
\item $u$ is the weak*-limit of an $\mathfrak{F}$-minimising sequence $(u_{k})\subset\mathscr{D}_{u_{0}}(:= u_{0}+\ld_{0}(\Omega))$. 
\item $u$ is the strong $\lebe^{1}$-limit of an $\mathfrak{F}$-minimising sequence $(u_{k})\subset\mathscr{D}_{u_{0}}$.
\end{enumerate}
\end{proposition}
\begin{proof}
We begin with a preparatory remark. We choose an open and bounded Lipschitz subset $\widetilde{\Omega}\subset\R^{n}$ with $\Omega\Subset\widetilde{\Omega}$. Given $u_{0}\in\ld(\Omega)$, by surjectivity of the trace operator on $\ld$ (see Section~\ref{sec:bd}), we may extend $u_{0}$ by some $v_{0}\in\ld(\ball\setminus\overline{\Omega})$ to a function $\widetilde{u_{0}}\in\ld_{0}(\widetilde{\Omega})$. We now invoke the straightforward generalisation of \cite[Chpt.~2.3.1]{Bild1} whose proof we leave to the interested reader:

\emph{
Given $\widetilde{\Omega}$ and $u_{0}$ as above, let $u\in\bd(\Omega)$ and denote its extension to $\widetilde{\Omega}$ via $\widetilde{u_{0}}$ by $\widetilde{u}$. Then there exists $(u_{k})\subset u_{0}+\hold_{c}^{\infty}(\Omega;\R^{n})$ such that $\widetilde{u_{k}}\to \widetilde{u}$ area--strictly in $\widetilde{\Omega}$ as $k\to\infty$, where $\widetilde{u_{k}},\widetilde{u}$ denote the extensions of $u_{k},u$ to $\widetilde{\Omega}$ by $\widetilde{u}$, respectively.}

We turn to the actual proof, and choose $\widetilde{\Omega}\equiv\ball$ as above before \eqref{eq:derivationofsplitting}.

\emph{Step 1. Existence of a generalised minimiser.} 
By \eqref{eq:lg} we have $m:=\inf_{u\in\bd(\Omega)}f[\E\widetilde{u}](\ball)>-\infty$ and so there exists a sequence $(u_{k})\subset\bd(\Omega)$ and $v\in\bd(\ball)$ such that $\widetilde{u_{k}}\stackrel{*}{\rightharpoonup}v$ in $\bd(\ball)$ as $k\to\infty$. By Proposition~\ref{prop:Reshetnyak}(a), $f[\E v](\ball)\leq \liminf_{k\to\infty}f[\E\widetilde{u_{k}}](\ball)=m$. Since $\widetilde{u_{k}}|_{\ball\setminus\overline{\Omega}}=v_{0}$, $v|_{\ball\setminus\overline{\Omega}}=v_{0}$ and so we conclude from \eqref{eq:derivationofsplitting} that $u:=v|_{\Omega}$ is a generalised minimiser in the sense of \eqref{eq:GMdef}. Now, since $\mathscr{D}_{u_{0}}\subset\bd(\Omega)$ and $\overline{\mathfrak{F}}_{u_{0}}|_{\mathscr{D}_{u_{0}}}=\mathfrak{F}|_{\mathscr{D}_{u_{0}}}$, we have $\inf_{\bd(\Omega)}\overline{\mathfrak{F}}_{u_{0}}\leq \inf_{\mathscr{D}_{u_{0}}}\mathfrak{F}$. On the other hand, let $u\in\gm(\mathfrak{F};u_{0})$ and apply the above area-strict approximation strategy to obtain a sequence $(u_{k})\subset u_{0}+\hold_{c}^{\infty}(\Omega;\R^{n})$ such that $\widetilde{u_{k}}\to \widetilde{u}$ area-strictly as $k\to\infty$. Then we obtain by \eqref{prop:Reshetnyak} -- as the ultimate term on the right side of \eqref{eq:derivationofsplitting} is constant -- 
\begin{align}
\begin{split}
\overline{\mathfrak{F}}_{u_{0}}[u] & = f[\E\widetilde{u}](\ball)-\int_{\ball\setminus\overline{\Omega}}f(\mathscr{E}v_{0})\dif x \\ & = \lim_{k\to\infty}f[\E\widetilde{u_{k}}](\ball)-\int_{\ball\setminus\overline{\Omega}}f(\mathscr{E}v_{0})\dif x = \lim_{k\to\infty}\mathfrak{F}[u_{k}] \geq \inf_{\mathscr{D}_{u_{0}}}\mathfrak{F}. 
\end{split}
\end{align}
Altogether, we have therefore established the absence of gaps, i.e.,
\begin{align}\label{eq:nogaps}
\min_{\bd(\Omega)}\overline{\mathfrak{F}}_{u_{0}}=\inf_{\bd(\Omega)}\overline{\mathfrak{F}}_{u_{0}}=\inf_{\mathscr{D}_{u_{0}}}\mathfrak{F}. 
\end{align}
\emph{Step 2. Proof of the claimed equivalences.} 
Ad (a)$\Rightarrow$(b) and (a)$\Rightarrow$(c). Let $u\in\gm(\mathfrak{F};u_{0})$ and choose an area-strictly approximating sequence $(u_{k})\subset u_{0}+\hold_{c}^{\infty}(\Omega;\R^{n})$ as indicated. Then, employing formula \eqref{eq:derivationofsplitting} with $\widetilde{\Omega}\equiv\ball$, we obtain $f[\E\widetilde{u_{k}}](\ball)\to f[\E \widetilde{u}](\ball)$ by virtue of Proposition~\ref{prop:Reshetnyak}. By constancy of the ultimate term in \eqref{eq:derivationofsplitting} and the fact that area-strict convergence implies both $\lebe^{1}$- and weak*-convergence, we conclude by means of \eqref{eq:nogaps}. Ad (b)$\Rightarrow$(c). This follows trivially as weak*-convergence implies strong $\lebe^{1}$-convergence. Ad (c)$\Rightarrow$(a). Let $(u_{k})\subset\mathscr{D}_{u_{0}}$ be an $\mathfrak{F}$-minimising sequence. By \eqref{eq:relaxinequality}, we obtain for all $v\in\bd(\Omega)$ 
\begin{align}
\overline{\mathfrak{F}}_{u_{0}}[u] \leq \liminf_{k\to\infty}\mathfrak{F}[u_{k}] = \inf_{\mathscr{D}_{u_{0}}}\mathfrak{F} \stackrel{\eqref{eq:nogaps}}{=}\min_{\bd(\Omega)}\overline{\mathfrak{F}}_{u_{0}}. 
\end{align}
The proof is complete. 
\end{proof}


\begin{thebibliography}{99}
\bibitem{AcerbiFusco} Acerbi, E.; Fusco, N. Regularity for minimizers of nonquadratic functionals: the case $1<p<2$. J. Math. Anal. Appl. 140 (1989), no. 1, 115--135.
\bibitem{ACD} Ambrosio, L.; Coscia, A.; Dal Maso, G.: Fine properties of functions with bounded deformation. Arch. Rational Mech. Anal. 139 (1997), no. 3, 201-238.
\bibitem{AFP} Ambrosio, L.; Fusco, N.; Pallara, D.: Functions of bounded variation and free discontinuity problems. Oxford Mathematical Monographs. The Clarendon Press, Oxford University Press, New York, 2000.
\bibitem{Anz} Anzellotti, G.: The Euler equation for functionals with linear growth. Trans. Amer. Math. Soc. 290 (1985), no. 2, 483-501.
\bibitem{AnzeGiaq} Anzellotti, G.; Giaquinta, M.: Existence of the displacement field for an elastoplastic body subject to Hencky's law and von Mises yield condition. Manuscripta Math. 32 (1980), no. 1-2, 101--136.
\bibitem{AdHe} Adams, D.R.; Hedberg, L.I.: Function Spaces and Potential Theory. Grundlehren der mathematischen Wissenschaften 314, Springer. 
\bibitem{Baba} Babadjian, J--F.: Traces of functions of bounded deformation. Indiana Univ. Math. J.  64  (2015),  no. 4, 1271--1290.
\bibitem{BS1}  Beck, L.; Schmidt, T.: On the Dirichlet problem for variational integrals in BV. J. Reine Angew. Math. 674 (2013), 113-194.
\bibitem{Bild1} Bildhauer, M.: Convex variational problems. Linear, nearly linear and anisotropic growth conditions. Lecture Notes in Mathematics, 1818. Springer-Verlag, Berlin, 2003. x+217 pp.
\bibitem{Bild3} Bildhauer, M.: A priori gradient estimates for bounded generalised solutions of a class of variational problems with linear growth. J. Convex Ana. 9 (2002), 117--137.
\bibitem{Bild4} Bildhauer, M.; Fuchs, M.: Partial regularity for variational
integrals with (s,µ,q)-growth. (2001) Calculus of Variations 13, 537-560.
\bibitem{Bild2} Bildhauer, M.; Fuchs, M.: On a class of variational integrals with linear growth satisfying the condition of $\mu$-ellipticity. Rend. Mat. Appl. (7) 22 (2002), 249-274 (2003).
\bibitem{BildhauerFuchsMingione} Bildhauer, M.; Fuchs, M.; Mingione, G.: A priori gradient bounds and local $\hold^{1,\alpha}$--estimates for (double) obstacle problems under non-standard growth conditions. Z. Anal. Anwendungen 20 (2001), no. 4, 959--985.
\bibitem{BBM}  Bourgain, J. ; Brezis, H. ; Mironescu, P.: Limiting embedding theorems for $W^{s,p}$ when $s\uparrow1$ and applications. Dedicated to the memory of Thomas H. Wolff. J. Anal. Math. 87 (2002), 77--101. 
\bibitem{CKP} Carozza, M.; Kristensen, J.; Passarelli Di Napoli, A.: Higher differentiability of minimizers of convex variational integrals. Annales de l'Institut Henri Poincar\'{e} (C) Non Linear Analysis 28 (3), 395--411
\bibitem{CKP1} Carozza, M., J. Kristensen, and A. Passarelli: On the validity of the Euler-Lagrange system. Communcations in pure and applied analysis, vol. 14, 2015. 
\bibitem{CMS}  Christiansen, E.; Matthies, H.; Strang, G.: The saddle point of a differential program. Energy methods in finite element analysis, pp. 309-318, Wiley, Chichester, 1979.
\bibitem{CFM} Conti, S., Faraco, D., Maggi, F.: A new approach to counterexamples to $\lebe^{1}$--estimates: Korn’s inequality, geometric rigidity, and regularity for gradients of separately convex functions. Arch. Rat. Mech. Anal. 175, 287--300 (2005).
\bibitem{DS} R.A. DeVore and R.C. Sharpley, "Maximal Functions Measuring Smoothness", Memoirs of Amer. Math. Soc (Num. 293) 47 Providence, 1984, 115 pp. 
\bibitem{DT}  Demengel, F.; Temam, R. Convex functions of a measure and applications. Indiana Univ. Math. J. 33 (1984), no. 5, 673-709.
\bibitem{Doro} Dorronsoro, J.R.: Mean Oscillation and Besov Spaces. Canad. Math. Bull. Vol. 28 (4) 1985, pp. 474--480.
\bibitem{ET} Ekeland, I.; Temam, R.: Convex analysis and variational problems. North--Holland, Amsterdam 1976.
\bibitem{ELM} Esposito, L.; Leonetti, F.; Mingione, G.: Sharp regularity for
functionals with (p,q)-growth. (2001) Journal of Differential Equations
204, 5-55.
\bibitem{Finn} Finn, R.: Remarks relevant to minimal surfaces and to surfaces of prescribed mean curvature, J. Anal. Math. 14 (1965), 139--160.
\bibitem{FM}  Fonseca, I.; M\"{u}ller, S.: Relaxation of quasiconvex functionals in $\bv(\Omega;\R^{p})$
for integrands $f(x,u,\nabla u)$, Arch. Ration. Mech. Anal. 123 (1993), 1-49. 
\bibitem{FM1} Fonseca, I.; M\"{u}ller, S.: Quasi--convex integrands and lower semicontinuity in $\lebe^{1}$, SIAM J. Math. Anal. 23 (1992), 1081--1098.
\bibitem{FuchsMingione} Fuchs, M.; Mingione, G.: Full $\hold^{1,\alpha}$--regularity for free and constrained local minimizers of elliptic variational integrals with nearly linear growth. Manuscripta Math. 102 (2000), no. 2, 227--250.
\bibitem{FS} Fuchs, M.; Seregin, G.: Variational methods for problems from plasticity theory and for generalized Newtonian fluids. Ann. Univ. Sarav. Ser. Math. 10 (1999), no. 1, iv+283 pp. 
\bibitem{GMS}  M. Giaquinta, G. Modica, and J. Souˇ cek, Functionals with linear growth in the calculus
of variations II, Commentat. Math. Univ. Carol. 20 (1979), 157-172.
\bibitem{Giusti} Giusti, E.: Direct methods in the calculus of variations. World Scientific Publishing Co., Inc., River Edge, NJ, 2003.
\bibitem{Gmeineder1} Gmeineder, F.: Regularity Theory for Variational Problems on BD. DPhil Thesis, University of Oxford, 2017.
\bibitem{GmPR} Gmeineder, F.: Partial Regularity for Convex Functionals on $\bd$. In Preparation.
\bibitem{GMS1} Giaquinta, M.; Modica, G.; Sou\v{c}ek, J.: Functionals with linear growth in the calculus of variations. Comm. Math. Univ. Carolinae 20 (1979), 143--172.
\bibitem{GoffSerrin} Goffman, C.; Serrin, J.: Sublinear functions of measures and variational integrals. Duke Math. J. 31 (1964), 159--178.
\bibitem{Woj} Kazaniecki, K.; Stolyarov, D.M.; Wojciechowski, M.: Anisotropic Ornstein non inequalities, Anal. PDE 10 (2017) 351--366.
\bibitem{KirchKrist}  Kirchheim, B.; Kristensen, J.: On Rank One Convex Functions that are Homogeneous of Degree One. Arch. Ration. Mech. Anal.  221  (2016),  no. 1, 527--558. 
\bibitem{KM1} Kristensen, J.; Mingione, G.: The singular set of minima of integral functionals. Arch. Ration. Mech. Anal. 180 (2006), no.~3, 331-398.
\bibitem{KM2} Kristensen, J.; Mingione, G.: The singular set of Lipschitzian minima of multiple integrals. Arch. Ration. Mech. Anal. 184 (2007), no.~2, 341--369.
\bibitem{Marc1} Marcellini, P.: Regularity and existence of solutions of elliptic equations with p,q-growth conditions. Journal of Differential Equations (1991), vol. 90, pp. 1-30. 
\bibitem{Marc2} Marcellini, P.: Regularity for elliptic equations with general growth conditions. Journal of Differential Equations (1993), vol. 105, pp. 296-333.
\bibitem{Min1} Mingione, G.: The singular set of solutions to non-differentiable elliptic systems. Arch. Ration. Mech. Anal. 166 (2003), no.~4, 287--301.
\bibitem{Min2} Mingione, G.: Bounds for the singular set of solutions to nonlinear elliptic systems. Calc. Var. Partial Differential Equations 18 (2003), no.~4, 273--400.
\bibitem{MironescuSickel} Mironescu, P.; Sickel, W.: A Sobolev non embedding. (English summary) 
Atti Accad. Naz. Lincei Rend. Lincei Mat. Appl. 26 (2015), no. 3, 291--298. 
\bibitem{Ornstein} Ornstein, D.: A non--equality for differential operators in the $L^{1}$--norm. Arch. Rational Mech. Anal. 11 1962 40--49.
\bibitem{Reshetnyak} Reshetnyak, Yu. G.: Weak convergence of completely additive vector functions on a set. Siberian Mathematical Journal (1968), Volume 9, Issue 6, 1039-1045.
\bibitem{Santi} Santi, E.: Sul problema al contorno per l'equazione della superfici di area minima su domini limitati qualunque, Ann. Univ. Ferrara, N. Ser., Sez. VII 17 (1972), 13-26.
\bibitem{Seregin1} Seregin, G.: Variation-difference schemes for problems in the mechanics of ideally elastoplastic media, U.S.S.R. Comput. Math. Math. Phys. 25 (1985), 153--165; translation from Zh. Vychisl. Mat. Mat. Fiz. 25 (1985), 237--253.
\bibitem{Seregin2} Seregin, G.: Differential properties of solutions of variational problems for functionals of linear growth, J. Sov. Math. 64 (1993), 1256--1277; translation from Probl. Mat. Anal. 11 (1990), 51--79.
\bibitem{Seregin3} Seregin, G.: Two-dimensional variational problems of the theory of plasticity, Izv. Math. 60 (1996), 179--216; translation from Izv. Ross. Akad. Nauk, Ser. Mat. 60 (1996), 175--210.
\bibitem{Seregin4} Seregin, G.: Differentiability properties of weak solutions of certain variational problems in the theory of perfect elastoplastic plates. Appl. Math. Optim. 28 (1993), 307--335. 
\bibitem{Smith} Smith, K. T.: Formulas to represent functions by their derivatives.
Math. Ann. 188 1970 53-77. 
\bibitem{Spector}  Spector, D.: Simple proofs of some results of Reshetnyak. Proc. Amer. Math. Soc. 139 (2011), no. 5, 1681-1690. 
\bibitem{Stein} Stein, Elias M. Harmonic analysis: real-variable methods, orthogonality, and oscillatory integrals. With the assistance of Timothy S. Murphy. Princeton Mathematical Series, 43. Monographs in Harmonic Analysis, III. Princeton University Press, Princeton, NJ, 1993.
\bibitem{Santi} Santi,E.: Sul problema al contorno per l’equazione delle superfici di area minima su domini limitati qualunque, Ann. Univ. Ferrara, N. Ser., Sez. VII 17 (1972), 13--26.
\bibitem{Serrin}  J. Serrin, The problem of Dirichlet for quasilinear elliptic differential equations with many
independent variables, Philos. Trans. R. Soc. Lond., Ser. A 264 (1969), 413–496.
\bibitem{ST} Strang, G.; Temam, R.: Functions of bounded deformation. Arch. Rational Mech. Anal. 75 (1980/81), no. 1, 7-21. 
\bibitem{Strauss} Strauss, M.J.: Variations of Korn's and Sobolev's equalities. Partial differential equations (Proc. Sympos. Pure Math., Vol. XXIII, Univ. California, Berkeley, Calif., 1971), pp. 207--214. Amer. Math. Soc., Providence, R.I., 1973.
\bibitem{Suquet}  Suquet, Pierre-Marie Sur un nouveau cadre fonctionnel pour les équations de la plasticité. (French) C. R. Acad. Sci. Paris Sér. A-B 286 (1978), no. 23, A1129-A1132.
\bibitem{Tartar} Tartar, L.: An Introduction to Sobolev Spaces and Interpolation Spaces. Lecture Notes of the Unione Matematica Italiana, Vol. 18 (2007)
\bibitem{Triebel0} Triebel, H.: Theory of function spaces. Reprint of 1983 edition [MR0730762]. Also published in 1983 by Birkh\"{a}user Verlag [MR0781540]. Modern Birkh\"{a}user Classics. Birkh\"{a}user/Springer Basel AG, Basel, 2010. 285 pp.
\bibitem{VS} Van Schaftingen, J.: Limiting Sobolev Inequalities for Vector Fields and Cancelling Differential Operators, Jour. Euro. Math. Society, 2013. 

\end{thebibliography}
\end{document}